\documentclass[]{interact}

\usepackage{epstopdf}
\usepackage[caption=false]{subfig}

\usepackage[numbers,sort&compress]{natbib}
\bibpunct[, ]{[}{]}{,}{n}{,}{,}

\theoremstyle{plain}
\newtheorem{theorem}{Theorem}[section]
\newtheorem{lemma}[theorem]{Lemma}
\newtheorem{corollary}[theorem]{Corollary}

\theoremstyle{definition}

\newtheorem{example}[theorem]{Example}
\newtheorem{assumption}{Assumption}

\theoremstyle{remark}
\newtheorem{remark}{Remark}

\begin{document}


\title{Global stabilization of the planar Ricker system with noisy PBC}

\author{
\name{E. Braverman\textsuperscript{a}\thanks{Corresponding author. Email:
maelena@ucalgary.ca; maelena@math.ucalgary.ca} and  A. Rodkina\textsuperscript{b}}
\affil{\textsuperscript{a}Department of Mathematics and Statistics, University of Calgary,\\
2500 University Drive N.W., Calgary, AB T2N 1N4, Canada; \textsuperscript{b}Department of Mathematics,
The University of the West Indies,\\ Mona Campus, Kingston, Jamaica}
}

\maketitle

%

\begin{abstract}
We apply Prediction-Based control (PBC) in order to  stabilize globally  a positive equilibrium of  a planar Ricker's equation. 
We  construct a closed invariant set in a strictly positive domain for the controlled map
and derive conditions on control parameters
ensuring that the increments of a specially constructed  Lyapunov function are  nonpositive on this set.
By stochastic  perturbation of the parameters we decrease the average values  of controls  providing global, as well as local, stabilization. Computer simulations illustrate our results.
\end{abstract}


\bigskip

\section{Introduction}

Difference systems describe models of population ecology, where distinctive time stages can be identified, for example, semelparous species experiencing reproduction event only once per season, 
so that their dynamics can be adequately described by difference equations. 
Simple one-dimensional  difference equations can exhibit complicated behaviour, such as multiple attractors and transition to chaos. 
However, real-world problems typically include either structured populations  or several interacting species leading to a system with even more sophisticated behaviour. A population can have spatial structure, where the area is divided into patches, and the reproduction rate at a patch is negatively influenced by overpopulation in adjacent patches, and in addition this influence is one-step delayed
\cite{yamamoto}, or be divided into juveniles-adults with different dynamics \cite{tang_li}. 
A system can also describe two species, such as host-parasite interactions (see \cite{parasitoid,tang_li}  and references therein).
Qualitative theory for discrete systems of population dynamics has been an area of active research for a long time \cite{Kang1,LSO,Kang2,Smith}.

It is quite a typical situation when a two-dimensional model describes two competing species, where for each of $x,y$ the amount at the next stage is proportional to the current amount 
 \begin{equation}
 \label{eq:general}
x_{n+1} = x_n {\mathcal R}_1 (x_n,y_n),\quad   y_{n+1} = y_n {\mathcal R}_2 (x_n,y_n),
\end{equation}
where the reproduction rates ${\mathcal R}_i (x,y)$ decrease in both $x$ and $y$, $i=1,2$.  
Its particular cases are  a competitive planar Ricker model  with exponential-type per capita growth rates ${\mathcal R}_1$  
 \begin{equation}
 \label{eq:Rick}
 \left\{
 \begin{array}{l}
x_{n+1}= \displaystyle x_n\exp\{r-x_n-ay_n\},\\
y_{n+1}= \displaystyle y_n\exp\{s-bx_n-y_n\}, \\ x_0, y_0>0,  n\in \mathbb N_0 =  \{ 0,1,2, \dots \} = \mathbb N \cup \{0 \}
\end{array}
\right.
 \end{equation}
considered in  \cite{yedaneh}
and the  Leslie-Gower model with ${\mathcal R}_i=(A_i+B_ix+C_iy)^{-1}$.
To explore global stability, Lyapunov functions is one of the main tools  \cite{BHEBL,girard} that we also later utilize.

Similarly to one-dimensional maps, systems can exhibit unstable and chaotic behaviour. There are many methods developed to stabilize such systems.
The control type which proposes to combine a state variable with the vector map value at one of the next steps is called Prediction-Based Control (PBC), see \cite{FL2010,uy99}.
PBC was generalized to include two or more future states \cite{Ayers}. An important feature of its original and generalized version is that original, not shifted,  fixed points are stabilized.
Another property of this control tool is that it is unconditional, and its application does not depend on the state variable. 
This leads to simplicity in application; however, methods with switching and conditional control can save on control resources. Examples of conditional controls are in-box method~\cite{cushing2024chaos,desharnais2001chaos} used to regulate flour beetle population, adaptive limiter control~\cite{sah2013stabilizing,Segura2019}, 
threshold harvesting~\cite{Liz2022} applied in most resources regulation policies, see also review paper \cite{Zelinka2023} for additional examples.
Including stochasticity in discrete models, for example, random switching  \cite{han,liberzon}  is motivated by problems in engineering, as well as in mathematical economics \cite{deng} and stochastic game theory  \cite{aberkane}. For including stochasticity in PBC method 
see \cite{BKRPhys, BR2023, pbclocsys}, for possible stabilization of cycles and multiple equilibrium points \cite{ BKRcycl, BRmult}.

In the version we apply, PBC substitutes the map with a linear combination of the current and the next state, where the coefficient of the current state (between zero and one) corresponds to the control intensity. The simpest case when the predicted state is defined by the map corresponds to the zero controls. 
PBC with the intensities $\alpha$ and $\beta$ in $x$ and $y$, respectively, applied to \eqref{eq:general}  gives
 \begin{equation}
 \label{eq:general_PBC}
x_{n+1} = \alpha x_n +  (1- \alpha) x_n {\mathcal R}_1 (x_n,y_n),~   y_{n+1} = \beta  y_n + (1-\beta) y_n  {\mathcal R}_2 (x_n,y_n),
~ \alpha,\beta \in [0,1).
\end{equation}
After the  application of  PBC as in \eqref{eq:general_PBC},
equation  \eqref {eq:Rick} takes the form 
 \begin{equation}
 \label{eq:RickPBC}
 \left\{
 \begin{array}{ll}
x_{n+1}=&x_n\left [(1-\alpha)\exp\{r-x_n-ay_n\}+\alpha\right]=f(\alpha, x, y),\\
y_{n+1}=&y_n\left[(1-\beta)\exp\{s-bx_n-y_n\}+\beta\right]=g(\beta, x, y).
\end{array}\right.
 \end{equation}
Note that the models where the next-stage value is a weighted average of the state variable and the value at the previous stage appear as natural modifications for systems of population dynamics, not as a control application.
For example, in \cite{kulenovic},  system \eqref{eq:general_PBC}
was introduced and treated as a modification of the  Leslie-Gower model. 

Denoting 
\begin{equation}
 \label{def:Ups}
 \Upsilon=\Upsilon_{\alpha, \beta}:=\left( \begin{array}{cc} \alpha & 0\\ 0& \beta \end{array} \right),\quad \alpha, \beta\in [0, 1),
\end{equation}
 \begin{equation}
 \label{def:Talphabeta}
 \begin{split}
  U:&=(x, y),\\
 T(U):&=\left(f(1, x, y), \, g(1, x, y)\right)=\left( x\exp\{r-x-ay\}, \, y\exp\{s-bx-y\}\right),\\
  T_{\alpha, \beta}(U):&=\left(f(\alpha, x, y), \, g(\beta, x,y)\right)\\
  &=\left(x\left [(1-\alpha)\exp\{r-x-ay\}+\alpha\right], \, y\left[(1-\beta)\exp\{s-b-y\}+\beta\right]\right),\\
 \end{split}
 \end{equation}
 we can rewrite equation \eqref{eq:RickPBC} as 
 \[
 U_{n+1}=(I-\Upsilon_{\alpha, \beta})T(U_n)+\Upsilon_{\alpha, \beta} U_n, \quad \mbox{or} \quad U_{n+1}=T_{\alpha, \beta}U_{n}.
 \]

Everywhere in the paper we assume that 
\begin{equation}
 \label{cond:mainposeq}
 r>as, \quad s>br, \quad 1>a,b>0.
 \end{equation}
 We are mostly interested in $r, s>2$, as for $0<r,s \leq 2$, global asymptotic  stability of the coexistence equilibrium was justified in \cite[Page 583]{BHEBL}. 
 
 We set 
 \begin{equation}
 \label{def:pqK}
 p:=\frac{r-as}{1-ab}, \quad q:=\frac{s-br}{1-ab}, \quad  K=(p, q)
 \end{equation}
 and note that conditions \eqref{cond:mainposeq} guarantee that the lines $y=s-bx$ and $y=\frac{r-x}a$ intersect at $K$,
 which is a unique positive equilibrium of $T$, as well as of $T_{\alpha, \beta}$. 
There are 3 more equilibriums, $(0, s)$, $(r, 0)$, $(0, 0)$, however we find low bounds  $\tilde \alpha, \tilde \beta\in (0,1)$ of parameters  $\alpha, \beta$ and  construct the invariant for $T_{\alpha,\beta}$ set $D_{\alpha,\beta}$ which does not include these equilibriums, for $\alpha>\tilde\alpha$, $\beta>\tilde \beta$.

 Along with \eqref{eq:RickPBC}  we also consider the case when  both control parameters are variable, i.e.
\begin{equation}
 \label{eq:RickPBCvar}
 \left\{
 \begin{array}{ll}
x_{n+1}=&x_n\left [(1-\alpha_{n+1})\exp\{r-x_n-ay_n\}+\alpha_{n+1}\right],\\
y_{n+1}=&y_n\left[(1-\beta_{n+1})\exp\{s-bx_n-y_n\}+\beta_{n+1}\right],
\end{array}\right.
 \end{equation}
where $\alpha_n, \beta_n\in (0, 1)$ for each $n\in \mathbb N_0$, and $(x_0, y_0)\in \mathbb R^2_+$. For $T_{\alpha, \beta}(x,y)$  defined as in \eqref{def:Talphabeta} and $U_n:=(x_n, y_n)$, equation \eqref{eq:RickPBCvar} can be written as 
\begin{equation}
\label{def:Tn+1}
(x_{n+1}, y_{n+1})=T_{\alpha_{n+1}, \beta_{n+1}}(x_n, y_n), \quad \mbox{or} \quad U_{n+1}=T_{\alpha_{n+1}, \beta_{n+1}}U_n, \quad n\in \mathbb N_0.
\end{equation}
To obtain a  global stability result for  system  \eqref{eq:RickPBCvar},  we apply  the approach from \cite{BHEBL}. We construct  a special Lyapunov function and derive the low bounds $\rho_1, \rho_2$ for controls guaranteeing that its increments are nonpositive on solutions of \eqref{eq:RickPBCvar}, whenever $\alpha_n > \rho_1$, $\beta_n>\rho_2$. We show that for each $\delta>0$, the ball neighbourhood $B(\delta, K)$ of $K$ can be reached in the nonrandom finite number of steps $S(\delta)$, and therefore prove global stability.

 Along with general variable controls, we consider  stochastically perturbed parameters $\alpha_n=\alpha+\ell \xi_{n}$, $\beta_n=\beta+\bar \ell \chi_{n}$, where $\alpha$, $\beta$, $\ell$, $\bar \ell$ are nonnegative numbers,  $(\xi_n)_{n\in \mathbb N}$, $(\chi_n)_{n\in \mathbb N}$ are sequences of independent identically distributed bounded random variables, such that $\xi, \chi\in [1-\varepsilon, 1]$ for each $\varepsilon>0$,  with positive probability. The main focus of this paper is to justify that  the introduction of  noise into the controls may improve stability results for  system \eqref{eq:RickPBC}. In other words,  the equilibrium solution  to system \eqref{eq:RickPBCvar} might be asymptotically stable (locally or globally) for some $ \ell, \bar \ell > 0$ even though it is not true for $ \ell=0$, $\bar \ell = 0$, i.e. for  system~\eqref{eq:RickPBC}.

The proof of global stability in the case of stochastically perturbed controls is quite distinct from the deterministic case. We use obtained earlier bounds $\rho_1$, $\rho_2$, however  the stochastic controls should exceed them only for a limited time $S$, which allows a solution to enter the neighborhood  of $K$ where local stability takes place. Applying a version of the Borell-Cantelli Lemma we  justify that there is  a random uniformly bounded across the sample space moment $\mathcal N$ such that $\alpha+\ell \xi_{\mathcal N+i}>\rho_1$,  $\beta+\bar \ell \chi_{\mathcal N+i}>\rho_2$ for all $i=0, 1, \dots, S$.  Applying the result based on the Lyapunov function, which is mentioned above, we conclude that the solution must enter  the ball $B(\delta, K)$ at  a random moment $\tau\le S$. When the solution is already in $B(\delta, K)$, we apply the local stability result. The main condition for local stability  is 
\begin{equation}
\label{cond:intrKolm}
\mathbb E\ln \|J_{\alpha+\ell\xi, \beta+\bar \ell\chi}\|\le -\nu<0,
\end{equation}
where $\nu>0$, $\|\cdot\|$ is some matrix norm, $J_{\alpha+\ell\xi,\beta+\bar \ell\chi}$  is a Jacobian of $T_{\alpha+\ell\xi, \beta+\bar \ell\chi}$ at the equilibrium~$K$. In the proof of this result we apply  the Kolmogorov's Law of Large Numbers and  approaches from \cite{BRMedv,pbclocsys}.

Let us emphasize  that in our main result, Theorem~\ref{thm:globstabst0}, not only the lower bounds of the noise are  not in the stabilization domain, but even the average control bounds without noise do not guarantee global stabilization of the positive equilibrium. The lower control bounds  with noise, however, should exceed the values sufficient for the entrance in the invariant subdomain separated from the axes in the first quadrant, earlier denoted as $\tilde{\alpha}$ and $\tilde{\beta}$.
For successful entrance into the local convergence neighbourhood of $K$, the upper bounds of the control with noise should exceed global stability bounds $\rho_1$, $\rho_2$ earlier determined and justified using the Lyapunov function.

The structure of the paper is as follows.

 In Section~\ref{sec:invar}  we find the lower bounds on the parameters  $\alpha, \beta$ such that the operator $T_{\alpha, \beta}$ has an invariant closed rectangle $\mathcal D_{\alpha, \beta}\subset (0, e^{r-1}]\times (0, e^{s-1}]$ separated from the axes. We also find the intervals $(\underline \alpha, \bar \alpha)$ and $(\underline \beta, \bar \beta)$ such that when the variable controls remain in these intervals, the solution to \eqref{eq:RickPBCvar} remains in $\mathcal D_{\underline\alpha, \underline\beta}$ if the initial values $(x_0, y_0)\in \mathcal D_{\underline\alpha, \underline\beta}$. If however the initial value $(x_0, y_0)\notin \mathcal D_{\underline\alpha, \underline\beta}$, $x_0, y_0>0$, we show that the solution reaches $\mathcal D_{\underline\alpha, \underline\beta}$ in a finite number of steps, which depends on $(x_0, y_0)$.
 
 In Section~\ref{sec:locstab} we derive conditions on control parameters  which guarantee local stability. In case of a constant control, for local stability it is sufficient that  the eigenvalues of the Jacobian $ J_{\alpha, \beta}$ of $T_{\alpha, \beta}$ at the equilibrium $K$ are inside the unit ball.  When controls are variable,  the Jacobian  $ J_{\alpha_n, \beta_n}$ is not constant. To obtain local stability, we derive conditions on the controls which ensure that a norm of $ J_{\alpha_n, \beta_n}$ can be  estimated by a constant $\lambda<1$.  We consider induced matrix norms of $ J_{\alpha_n, \beta_n}$ in  ${\mathbf \ell}^p$, such as the Euclidean ${\mathbf \ell}^2$ (spectral),  the traffic ${\mathbf \ell}^1$ or the maximum  ${\mathbf \ell}^{\infty}$ norm.  

In Section~\ref{sec:Lyapglob}, applying the approach from \cite{BHEBL}, we construct the Lyapunov function  and derive the low bounds on controls which guarantee that its increments are nonpositive on solutions of \eqref{eq:RickPBC}. This ensures that the solution starting at any $(x_0, y_0)$, $x_0, y_0>0$, reaches the ball $B(\delta, K)$  for any prescribed $\delta>0$ and therefore converges to $K$. We also show that for each $\delta>0$ there is a finite number $S(\delta)$ of steps to reach the ball $B(\delta, K)$. This number depends only on $\delta$ and the bounds for the control parameters, once the initial condition is in the invariant domain. Once the initial point is outside of the domain, after a certain number of iterations it gets into the domain, and this number can also be computed. 

In Section~\ref{sec:stoch},  the control parameters are stochastically perturbed. We present three local stability results: Lemma~\ref{lem:locstabst0}  where  the noise intensity is small enough for the stochastic control to act path-wise as a deterministic variable control, and essentially stochastic Lemma ~\ref{lem:locKolmtauk} and Theorem~\ref{thm:locKolm}.  Condition ~\eqref{cond:intrKolm} is the main one in Lemma~\ref{lem:locKolmtauk} and Theorem~\ref{thm:locKolm}, and it generalizes corresponding deterministic  condition in  Lemma~\ref{lem:locstabst0}.

Methods employed in the proof of  Lemma~\ref{lem:locKolmtauk} and Theorem~\ref{thm:locKolm}  are based on approaches from \cite{BRMedv,pbclocsys}  (see also \cite{Medvedev}). The results of Section~\ref{sec:Lyapglob}, Lemma ~\ref{lem:locKolmtauk} and Theorem~\ref{thm:locKolm} are applied to prove Theorem~\ref{thm:globstabst0},
where the estimates for the control parameters obtained in Section~\ref{sec:Lyapglob} to ensure global stability, are improved by stochastic perturbations. 

Section~\ref{sec:ex} illustrates theoretical results and confirms that theoretically established bounds lead to global stabilization. Moreover, sometimes weaker controls allow to stabilize, and simulations allow to find approximate bounds.

Section~\ref{sec:conclusions}, in addition to the summary, indicates some future research directions.

Most of the proofs are postponed to the Appendix.

{\bf Notations.} We denote by  $\mathbb N$ the set of positive integers, $\mathbb N_0 = \mathbb N \cup \{ 0 \} = \{0,1,2, \dots \}$, we will use various norms $\| \cdot \|$ in the space of two-dimensional vectors ${\mathbb R}^2$ such as  the traffic norm $\| (x,y)  \|_1 =|x|+|y|$, the Euclidean norm  $\| (x,y)  \|_2 =
\sqrt{x^2+y^2}$, and the maximum norm  $\| (x,y)  \|_{\infty} = \max\{|x|,|y|\}$. By the same symbol we denote the induced matrix norms.

\bigskip

\section{An Invariant Set}
\label{sec:invar}

In this section, for some $\tilde \alpha\in (0,1)$, $\tilde \beta\in (0,1)$ and all $(\alpha, \beta)\in (\tilde \alpha,1)\times (\tilde \beta,1)$ we construct an invariant set $D_{\alpha, \beta}$ for  $T_{\alpha, \beta}$, defined  by \eqref{def:Talphabeta}. Note that, as the maximum of the function $xe^{-x}$  is $e^{-1}$,  we have $T_{\alpha, \beta}(x,y) \in [0,e^{r-1}]\times [0, e^{r-1}]$ for any $(\alpha, \beta)\in [0,1)\times [0,1)$, $x \geq 0$, $y \geq 0$.
Our purpose is to find an invariant closed set $D_{\alpha, \beta}\subset (0, e^{r-1})\times (0, e^{r-1})$  which is separated from the axes in the first quadrant. 

\subsection{Upper bound}
\label{subsec:upbound}

We start with establishing the upper bound  for $g(\beta, x, y)=y[(1-\beta)e^{s-y-bx}+\beta]$. Since $g(\beta, x, y)\le g(\beta, 0, y)$ for all $x>0$, we concentrate on finding the upper bound for 
\begin{equation}
\label{def:gbetas}
\begin{split}
g(\beta, 0, y):=y\left[(1-\beta)e^{s-y} +\beta \right], \quad y\in [0, e^{s-1}], 
\end{split}
\end{equation}
for each  $\beta\in (0,1)$. We  find  the interval $[0, \mathcal H(\beta)]$ with the smallest $ \mathcal H(\beta)\in (0, e^{s-1})$ such that $g(\beta, 0, \cdot): [0, \mathcal H(\beta)]\to [0, \mathcal H(\beta)]$. 

We set for the parameters involved in the second equation of \eqref{eq:Rick}
\begin{equation}
\label{def:beta12}
\begin{split}
&\beta_1:=\frac{e^{s-2}}{e^{s-2}+1}, \quad \beta_2:=\frac{e^{s-1}-s}{e^{s-1}-2}, \quad H_2(\beta):=(1-\beta)e^{s-1} +2\beta,\\
&\mathcal H_2(\beta):=\max\{(1-\beta)e^{s-1} +2\beta,\, s\},
\end{split}
\end{equation}
and note that $\beta_1\in (0,1)$,  for each $s>0$, $\beta_2\in (0,1)$,  for $s>2$, $ H_2(\beta)<e^{s-1}$ for each $\beta\in (0,1)$, and
\begin{equation*}
\mathcal H_2(\beta)=s \quad\mbox{if} \quad \beta>\beta_2, \quad \mathcal H_2(\beta)=(1-\beta)e^{s-1} +2\beta,\quad\mbox{if} \quad \beta \leq \beta_2.
\end{equation*}
Similarly, for the parameters involved in the first equation of \eqref{eq:Rick}, we set 
\begin{equation}
\label{def:alpha12}
\begin{split}
&f(\alpha, x, 0):=y\left[(1-\alpha)e^{r-x} +\alpha \right], ~ x \in [0, e^{r-1}], \\ &  \alpha_1:=\frac{e^{r-2}}{e^{r-2}+1}, ~\alpha_2:=\frac{e^{r-1}-r}{e^{r-1}-2}, \\ &  H_1(\alpha):=(1-\alpha)e^{r-1} +2\alpha, \quad
\mathcal H_1(\alpha):=\max\{(1-\alpha)e^{r-1} +2\alpha,\, r\}.
\end{split}
\end{equation}
\begin{lemma}
\label{lem:locmax}
Let $g(\beta, 0, y)$ be defined as in \eqref{def:gbetas}, and $\beta_1$, $\beta_2$, $H_2(\beta)$,  $\mathcal H_2(\beta)$ be defined as in \eqref{def:beta12}. Then
\begin{enumerate}
\item [(a)] If $\beta < \beta_1$, the function $g(\beta, 0, \cdot)$ has a unique local maximum at the point $\bar y=\bar y(\beta)\in (1,2)$ such that $g(\beta, 0, \bar y(\beta))<H_2(\beta)$, while
$g(\beta, 0, \cdot)$ is monotone increasing for $\beta \geq \beta_1$;
\item  [(b)]  $H_2(\beta)<s$ if $\beta>\beta_2$;
\item  [(c)]  $g(\beta, x, y)\le \mathcal H_2(\beta) $ for each $\beta\in (0,1)$, $x\in [0, e^{r-1}]$, $y\in [0, \mathcal H_2(\beta)]$;
\item  [(d)] $g(\beta, x, y)\le H$ for each $H>\mathcal H_2(\beta)$, $\beta\in (0,1)$, $x\in [0, e^{r-1}]$, $y\in [0, H]$.
\end{enumerate}
Similarly, for $f(\alpha, x, 0)$ and $\alpha_1$, $\alpha_2$, $H_1(\alpha)$,  $\mathcal H_1(\alpha)$ defined  in \eqref{def:alpha12} we have
\begin{enumerate}
\item [(e)] If $\alpha < \alpha_1$, the function $f(\alpha, \cdot, 0)$ has a unique local maximum at the point $\bar x=\bar x(\alpha)\in (1,2)$ such that $f(\alpha, \bar x(\alpha), 0)<H_1(\alpha)$, while
$f(\alpha, \cdot, 0)$ is monotone increasing for $\alpha \geq \alpha_1$;
\item [(f)] $H_1(\alpha)<s$ if $\alpha>\alpha_2$;
\item [(g)] $f(\alpha, x, y)\le \mathcal H_1(\alpha) $ for each $\alpha \in (0,1)$, $y\in [0, e^{s-1}]$, $x\in [0, \mathcal H_1(\alpha)]$;
\item [(h)] $f(\alpha, x, y)\le H$ for each $H>\mathcal H_1(\alpha)$, $\alpha \in (0,1)$, $y\in [0, e^{s-1}]$, $x\in [0, H]$.
\end{enumerate}

\end{lemma}

The proof is deferred to Appendix, Section~\ref{subsec:Ap1}.

\subsection{Lower bound and the invariant rectangle}

Now recall that we assume  $s/b>r>2, \quad r/a>s>2$.
We set, for any $\alpha, \beta\in (0,1)$,
\begin{equation}
\label{def:c12} 
\begin{split}
&c_1(\beta):=r-a\mathcal H_2(\beta), \quad c_2(\alpha):=s-b\mathcal H_1(\alpha),\\
&\tilde \alpha:=\max\left\{\frac{e^{r-1} -s/b}{e^{r-1}-2}, \, 0\right\}\quad \tilde \beta:=\max\left\{\frac{e^{s-1} -r/a}{e^{s-1}-2}, \, 0\right\},\\
&\alpha_2:=\frac{e^{r-1} -r}{e^{r-1}-2}, \quad \beta_2:=\frac{e^{s-1} -s}{e^{r-1}-2}.
\end{split}
\end{equation}
Since $e^{r-1}>r$  and $s/b>r>2$,  we get 
\[
e^{r-1}-2>r-2>0, \quad e^{r-1} -s/b<e^{r-1} -r<e^{r-1} -2.
\]
So both fractions on the second line of  \eqref{def:c12} are nonnegative and are also less than 1.  Since $s/b>r$, $r/a>s$  and by  \eqref{def:beta12},  \eqref{def:alpha12}, \eqref{def:c12}, we get 
\[
\frac{e^{r-1} -s/b}{e^{r-1}-2}<\frac{e^{r-1} -r}{e^{r-1}-2}=\alpha_2, \quad \frac{e^{s-1} -r/a}{e^{s-1}-2}<\frac{e^{s-1} -s}{e^{r-1}-2}=\beta_2,\quad \alpha_2, \beta_2>0,
\]
which implies $\tilde \alpha\in [0, \alpha_2)$, $\tilde \beta \in [0, \beta_2).$
Since $\mathcal H_1(\alpha)=(1-\alpha)e^{r-1}+2\alpha$ when $\alpha<\alpha_2$  and $\mathcal H_1(\alpha)=r$ for $\alpha\in [\alpha_2, 1)$ we get by \eqref{def:c12} 
\begin{equation*}
c_2(\alpha) = \left\{ \begin{array}{ll} \displaystyle  s-b[(1-\alpha)e^{r-1}+2\alpha ]>0, & \quad\mbox{if} \quad \alpha\in (\tilde \alpha, \alpha_2),\\
s-br>0, & \quad\mbox{if}  \quad \alpha\in [\alpha_2, 1). \end{array} \right.
\end{equation*}
Similarly, since $\mathcal H_2(\beta)=(1-\beta)e^{s-1}+2\beta$ when $\beta<\beta_2$  and $\mathcal H_2(\beta)=s$ for $\beta\in (\tilde \beta, \beta_2)$ we get 
\begin{equation*}
\begin{split}
c_1(\beta) = \left\{ \begin{array}{ll} \displaystyle  r-a[(1-\beta)e^{s-1}+2\beta ]>0, &\quad\mbox{if} \quad \beta\in (\tilde \beta, \beta_2),  \\
r-as>0, & \quad\mbox{if}  \quad \beta\in [\beta_2, 1).    \end{array} \right.
\end{split}
\end{equation*}
Therefore we arrive at 
\[
c_2(\alpha)>0, \quad\mbox{if}  \quad \alpha\in (\tilde \alpha, 1), \quad c_1(\beta)>0, \quad\mbox{if}  \quad \beta\in (\tilde \beta, 1).
\]
Note also that  
\[
c_2(\alpha)<s\le \mathcal H_2(\beta), \quad c_1(\beta)<r\le \mathcal H_1(\alpha).
\]
For all $\alpha>\tilde \alpha$, $\beta > \tilde \beta$, we define lower bounds of the invariant rectangular domain
\begin{equation}
\label{def:mathcalH12} 
\begin{split}
&u_{1}(\alpha, \beta):=\min\{ue^{-u}:  u\in [c_1(\beta), \mathcal H_1(\alpha)]\},\\
&u_{2}(\alpha, \beta):=\min\{ue^{-u}:  u\in [c_2(\alpha),  \mathcal H_2(\beta)]\}, \\
&\underline c_1(\alpha, \beta):=\min\{x[(1-\alpha)e^{r-x-ay}+\alpha]\,: x\in [c_1(\beta), \mathcal H_1(\alpha)], \, y\in [0, \mathcal H_2(\beta)]\},\\
&\underline c_2(\alpha, \beta):=\min\{y[(1-\beta)e^{s-y-bx}+\beta]\,: x\in [0, \mathcal H_1(\alpha)], \, y\in [c_2(\alpha), \mathcal H_2(\beta)]\}.
\end{split}
\end{equation}

The values $\underline{c}_i(\alpha, \beta)$, $i=1,2$, are the left sides of the intervals constituting the invariant domain $D_{\alpha,\beta}$, while $u_i(\alpha, \beta)$, $i=1,2$, are used in the lower bounds of $\underline{c}_i(\alpha, \beta)$. 
The properties of these expressions, stated in Lemma~\ref{lem:alphainvar0}  allow to show that a solution to \eqref{eq:RickPBCvar} with variable controls $\alpha_n\ge \underline \alpha, \beta_n\ge \underline \beta$ remains in  $D_{\underline  \alpha,\underline  \beta}$  for all $n\in \mathbb N$, once it starts in this domain.

\begin{lemma}
\label{lem:alphainvar0}
Let \eqref{cond:mainposeq} hold, $\mathcal H_1(\cdot)$, $\mathcal H_2(\cdot) $ be defined as in \eqref{def:alpha12}, \eqref{def:beta12}, respectively;
$\tilde \alpha$, $\tilde \beta$,  $c_1(\beta)$, $c_2(\alpha)$ be defined as in \eqref{def:c12}, $\underline c_1(\alpha, \beta)$, $\underline c_2(\alpha, \beta)$ be defined as in \eqref{def:mathcalH12}.
\begin{enumerate}
\item  [(a)] Each function $\mathcal H_1(\cdot)$ and $\mathcal H_2(\cdot) $ does not increase  in $(\tilde\alpha, 1)$ and   $(\tilde \beta, 1)$, respectively.
\item  [(b)]  Each function $c_1(\beta)$ and $c_2(\alpha) $ does not decrease  in $(\tilde \beta, 1)$ and   $(\tilde \alpha, 1)$, respectively.
\item [(c)] If  $\tilde \alpha_1 \geq  \tilde  \alpha_2>\tilde \alpha$, $\tilde  \beta_1 \geq   \tilde \beta_2>\tilde \beta$, we have
\begin{equation}
\begin{split}
\label{ineq:ab12}
&u_1(\tilde  \alpha_1, \tilde \beta_1) \geq u_1(\tilde \alpha_2, \tilde  \beta_2), \quad u_2(\tilde \alpha_1,  \tilde \beta_1) \geq u_2(\tilde \alpha_2, \tilde \beta_2),\\
&\underline c_1(\tilde  \alpha_1,  \tilde  \beta_1) \geq \underline c_1(\tilde  \alpha_2, \tilde  \beta_2), \quad \underline c_2(\tilde \alpha_1, \tilde  \beta_1) \geq
\underline c_2(\tilde  \alpha_2, \tilde  \beta_2).
\end{split}
\end{equation}

\item[(d)]
For $\alpha>\tilde \alpha$, $\beta>\tilde \beta$\\
$\underline c_1(\alpha, \beta)\ge (1-\alpha)e^{r-a \mathcal H_2(\beta)}u_{1}(\alpha, \beta)+\alpha c_1(\beta)>0$, \\$\underline c_2(\alpha, \beta)\ge (1-\beta) e^{s-b \mathcal H_1(\alpha)}u_{2}(\alpha, \beta)+\beta c_2(\alpha)>0$.

\item[(e)] $\underline c_1(\alpha, \beta)\le c_1(\beta)$, \,\, $\underline c_2(\alpha, \beta)\le c_2(\alpha).$
\end{enumerate}
\end{lemma}

The proof is deferred to Appendix, Section~\ref{subsec:Ap2}.

Set
\begin{equation}
\label{def:Dalphabeta}
\begin{split}
&D_{\alpha, \beta}:=[\underline c_{1}(\alpha, \beta), \mathcal H_{1}(\alpha)]\times [\underline c_{2}(\alpha, \beta), \mathcal H_{2}(\beta)],\\
&D_{\alpha, \beta}(h):=[\underline c_{1}(\alpha, \beta)-h, \mathcal H_{1}(\alpha)]\times [\underline c_{2}(\alpha, \beta)-h, \mathcal H_{2}(\beta)],\\
&\mbox{for}\quad \alpha> \tilde \alpha, \quad \beta>\tilde \beta, \quad  h\in \left(0, \min\{\underline c_{1, \, }(\alpha, \beta), \, \underline c_{2}(\alpha, \beta)\}\right).
\end{split}
\end{equation}

\begin{lemma}
\label{lem:alphainvar}
Let \eqref{cond:mainposeq} hold, $\tilde \alpha$, $\tilde \beta$  be defined as in \eqref{def:c12} and $\alpha>\tilde \alpha$, $\beta>\tilde \beta$. Let $T_{\alpha, \beta}$, $D_{\alpha, \beta}$ and $D_{\alpha, \beta}(h)$ be defined as in \eqref{def:Talphabeta} and \eqref{def:Dalphabeta}, respectively. Then the following statements are valid.
\begin{enumerate}
\item[(a)] $
T_{\alpha, \beta}(\cdot, \cdot): D_{\alpha, \beta}\to D_{\alpha, \beta}$, \,\, $T_{\alpha, \beta}(\cdot, \cdot): D_{\alpha, \beta}(h)\to D_{\alpha, \beta}(h)$.
\item[(b)] 
If $\tilde \alpha_1 \geq \tilde \alpha_2 > \tilde \alpha$, $\tilde \beta_1  \geq \tilde \beta_2>\tilde \beta$ then $T_{\tilde \alpha_1, \tilde \beta_1}(\cdot, \cdot): D_{\tilde \alpha_2, \tilde \beta_2}\to D_{\tilde \alpha_2, \tilde \beta_2},$ \\$T_{\tilde \alpha_1, \tilde \beta_1}(\cdot, \cdot): D_{\tilde \alpha_2, \tilde \beta_2}(h)\to D_{\tilde \alpha_2, \tilde \beta_2}(h)$.
\end{enumerate}
\end{lemma}

The proof is deferred to Appendix, Section~\ref{subsec:Ap3}.

Now,  applying Lemma~\ref{lem:alphainvar} we proceed to the case of variable controls.

\begin{lemma}
\label{lem:Tninvar}
Let \eqref{cond:mainposeq} hold, $\tilde \alpha$, $\tilde \beta$  be defined as in \eqref{def:c12}, $\underline c_i(\alpha, \beta)$, $i=1,2$, be defined as in \eqref{def:mathcalH12}, 
$\underline \alpha\in (\tilde \alpha, 1)$, $\underline \beta\in (\tilde \beta, 1)$, $\bar \alpha\in (\underline \alpha,  1)$, $\bar \beta\in (\underline \beta, 1)$,  
$h\in \left(0, \min\{\underline c_{1}(\underline\alpha, \underline\beta), \, \underline c_{2}(\underline\alpha, \underline\beta)\}\right)$ be some numbers, and $D_{\underline \alpha, \underline\beta}$ and $D_{\underline \alpha, \underline\beta}(h)$ be defined as in  \eqref{def:Dalphabeta}.
Let variable controls $\alpha_n$ and $\beta_n$ for each $n\in \mathbb N$ satisfy
\[
\alpha_n\in (\underline \alpha, \bar \alpha), \quad \beta_n\in (\underline \beta, \bar \beta),
\]
and a sequence $(x_n, y_n)_{n\in \mathbb N}$ be a solution to \eqref{eq:RickPBCvar} with the initial value $(x_0, y_0)$.
\begin{enumerate}
\item  [(a)] If $(x_0, y_0)\in D_{\underline \alpha, \underline\beta}$, then $(x_n, y_n)\in D_{\underline \alpha, \underline\beta}$ for all $n\in \mathbb N$.
\item  [(b)] If $(x_0, y_0)\in D_{\underline \alpha, \underline\beta}(h)$, then $(x_n, y_n)\in D_{\underline \alpha, \underline\beta}(h)$ for all $n\in \mathbb N$.
\item [(c)]  Let  $(x_0, y_0)\notin D_{\underline \alpha, \underline\beta}(h)$, then there exists a number $\tilde S=\tilde S(x_0, y_0, \underline \alpha, \underline\beta, \bar \alpha, \bar\beta, h) \in {\mathbb N}$ such that $(x_n, y_n)\in D_{\underline \alpha, \underline\beta}(h)$  for any $n\ge \tilde S$.
\end{enumerate}
\end{lemma}
The proof is deferred to Appendix, Section \ref{subsec:Ap4}.

\bigskip

\section{Local Stability for PBC}
\label{sec:locstab}

Following the main stabilization goal of the paper, we first explore local stability of controlled Ricker map and find sufficient conditions on the control intensities to guarantee local stabilization. We start with constant controls as in  \eqref{eq:RickPBC}  in Section~\ref{subsec:constcongrol}. As the system is autonomous,
local asymptotic (and even exponential) stability of its positive equilibrium $K$ follows from the fact that the spectral radius of the Jacobian at $K$ is less than one 
\cite[P. 222, Corollary 4.34]{Elaydi}, see  \cite[Chapter 4.6]{Elaydi} for an overview of the linearisation method.
Next, in Section~\ref{subsec:PBCvar} we consider variable controls, which is a reasonable assumption when some perturbations are included 
and is a required framework, once stochastic component is incorporated in the controls.
For non-autonomous systems, the fact that the spectral radius of each Jacobian matrix is less than one is insufficient for stabilization. 
Here, the estimate that all the norms (in most cases, we apply the induced max-norm) are less than some $\lambda \in (0,1)$ is required.

\subsection{ PBC with constant controls}
\label{subsec:constcongrol}

 As usual, condition \eqref{cond:mainposeq} is supposed to hold. Consider \eqref{eq:RickPBC} and derive  sufficient conditions for local asymptotic stability of its positive equilibrium $K=(p,q)$, where $p, q$ are  defined in \eqref{def:pqK}. The corresponding Jacobian is
\begin{equation}
\label{Jacobian_control}
J_{\alpha,\beta}  =  \left(\begin{array}{ll}
(1-\alpha)(1-p) + \alpha  &  -(1-\alpha) ap \\ -(1-\beta) bq& (1-\beta)(1-q) + \beta  \end{array}\right).
\end{equation}
Denoting
\begin{equation}
\label{def:paqb}
p_\alpha:=(1-\alpha)p, \quad q_\beta=(1-\beta)q, 
\end{equation}
we can rewrite $J_{\alpha,\beta}$ in the following form
\[
J_{\alpha,\beta}  =  \left(\begin{array}{ll}
1-p_\alpha  &  -ap_\alpha \\ -bq_\beta& 1-q_\beta  \end{array}\right)
\]
and get the trace and the determinant of  $J_{\alpha,\beta}$
$$
{\rm tr}~J_{\alpha, \beta}=2-p_\alpha-q_\beta, \quad  {\rm det}~J_{\alpha, \beta}=1-p_\alpha-q_\beta+p_\alpha q_\beta(1-ab).
$$
 
Local stability of the equilibrium $K$ follows from the fact that both eigenvalues 
$\lambda_{1,2}=\frac{1}{2} ({\rm tr}~J_{\alpha, \beta}\pm\sqrt{{\rm tr}^2~J_{\alpha, \beta}-4{\rm det}~J_{\alpha, \beta} } )$ of $J_{\alpha,\beta}$  are inside the unit circle, which is equivalent to 
\begin{equation}
\label{rel:eigenvl}
\sqrt{{\rm tr}^2~J_{\alpha,\beta}-4 {\rm det}~J_{\alpha,\beta}}<\min \{2-{\rm tr}~J_{\alpha,\beta}, \, 2+{\rm tr}~J_{\alpha,\beta}\},
\end{equation}
and which can be possible only  if $|{\rm tr}~J_{\alpha,\beta}|<2$. Note that ${\rm tr}~J_{\alpha,\beta}^2-4~ {\rm det}_{\alpha,\beta}=(p_\alpha-q_\beta)^2+4p_\alpha q_\beta ab\ge 0$,  so both eigenvalues are real. If ${\rm tr}~J_{\alpha,\beta}\ge 0$ then \eqref{rel:eigenvl} holds for all $\alpha, \beta$. If ${\rm tr}~J_{\alpha,\beta}<0$ then 
\eqref{rel:eigenvl} holds for  $2{\rm tr}~J_{\alpha,\beta}+p_\alpha q_\beta (1-ab)>0$. Thus  we arrive at   sufficient conditions  for local stability of $K$:
\begin{equation}
\label{cond:IFFPBC}
\begin{split}
p_\alpha+q_\beta<4, \quad 2-p_\alpha-q_\beta+\frac{p_\alpha q_\beta (1-ab)}{2}>0.
\end{split}
\end{equation}
We set
\begin{equation}
\label{cond:abcontr2}
\begin{split}
&\alpha_*:=\max\{1-2/p, 0\}, \\
& \beta_*:= \max\left\{1-\frac2{q(1-ab)}-\frac {4ab}{q(1-ab)[(1-\alpha)(1-ab)p-2]}, \, \, 0\right\}, ~\mbox{if} ~ \alpha>\alpha_*.
\end{split}
  \end{equation}
Since
  \begin{equation*}
 \begin{split}
 & \frac2{q(1-ab)}\left[1+\frac {2ab} {[(1-\alpha)(1-ab)p-2]}\right]
 =\frac{2(1-ab)}{q(1-ab)}\times \frac {(1-\alpha)p-2} {[(1-\alpha)p-2]}=\frac 2q>0,
 \end{split}
 \end{equation*} 
 we get $\beta_*\in [0, 1)$.  Also, $\alpha_*=0$ if $p<2$, so we do not have any restrictions on $\alpha$, and, similarly,  $\beta_*=0$  if  $q<2$, and there is no restrictions on $\beta$. 

\begin{lemma}
\label{lem:locstPBCpar}
Let condition \eqref{cond:mainposeq} hold and let $\beta_*$, $\alpha_*$ be defined by \eqref{cond:abcontr2}.  Then
\begin{enumerate}
\item [(a)] condition \eqref{cond:IFFPBC}  is  sufficient for  local asymptotic stability  of the positive equilibrium $K$ of system \eqref{eq:RickPBC};
\item [(b)] if $\alpha\in (\alpha_*,1)$, $\beta \in (\beta_*,1)$  then condition \eqref{cond:IFFPBC} holds.

\end{enumerate} 
\end{lemma}

The proof is deferred to Appendix, Section~\ref{subsec:locstPBCpar}.

\begin{corollary}
\label{cor:locstab}
Let condition \eqref{cond:mainposeq} hold. 
\begin{enumerate}
\item [(a)] If  $(1-\alpha)p+(1-\beta)q<2$, both conditions in \eqref{cond:IFFPBC} hold, so it is sufficient for the  local stability.
\item [(b)] There always exist parameters  $\alpha^*, \beta^*$ such that any $\alpha \in (\alpha^*,1)$, $\beta \in (\beta^*,1)$ 
stabilize locally~\eqref{eq:RickPBC}.
\item [(c)] If for some $\alpha^*, \beta^*$ equilibrium $K$ to \eqref{eq:RickPBC} is locally stable then for any $\alpha \in (\alpha^*,1)$, $\beta \in (\beta^*,1)$, the equilibrium $K$ to \eqref{eq:RickPBC} is also  locally stable.
\end{enumerate}
\end{corollary}
\begin{proof}
Parts {\bf (a)-(b)} are straightforward corollaries of relations  \eqref{cond:IFFPBC}-\eqref{cond:abcontr2}.

To prove {\bf (c)} we note that $p_\alpha<p_{\alpha^*}$, $q_\beta<q_{\beta^*}$, so $p_\alpha+q_{\beta}<4$. Function $G(x, y)$, defined as in \eqref{def:G}, decreases 
with respect to each variable, $x\in \left(0, \frac 2{1-ab}\right)$ and $y\in \left(0, \frac 2{1-ab}\right)$, since $G'_x(x, y)<0$, $G'_y(x, y)<0$.
Then
\begin{align*}
G(p_\alpha, q_\beta) & =G(p_\alpha, q_\beta)-G(p_{\alpha^*}, q_\beta)+G(p_{\alpha^*}, q_\beta)-G(p_{\alpha^*}, q_{\beta^*})+G(p_{\alpha^*}, q_{\beta^*})
\\
& >
G(p_{\alpha^*}, q_{\beta^*})>0.
\end{align*}
Here $G(p_{\alpha^*}, q_{\beta^*})>0$, because   \eqref{cond:IFFPBC} is a sufficient local stability condition with controls  $\alpha^*, \beta^*$. 
\end{proof}



\subsection{PBC with variable controls}
\label{subsec:PBCvar}
 
The Jacobian $J_{\alpha_n, \beta_n}$ of  Ricker system \eqref{eq:RickPBCvar} with step-dependent control intensities 
 is not a constant matrix,  thus to prove  local stability we cannot just rely on the fact that its eigenvalues at each step are inside the unit circle,
the maximum of all norms should not exceed a number $\lambda\in (0,1)$. 
The next lemma gives sufficient local stability conditions for system \eqref{eq:RickPBCvar} in terms of the norm for the matrices  $J_{\alpha_n, \beta_n}$.

\begin{lemma}
\label{lem:locstabst1_var}
Let inequalities \eqref{cond:mainposeq} hold.  Then the equilibrium $K$ of \eqref{eq:RickPBCvar}  is locally  asymptotically stable, if anyone of the two  conditions holds.
\begin{enumerate}
\item [(a)] There exist $\lambda \in (0,1)$, a norm in ${\mathbb R}^2$,  and four constants $0 \leq \alpha_* <\alpha^*<1$, $0\leq \beta_* < \beta^*<1$ 
such that for every $n \in {\mathbb N}$, $\alpha_n$, $\beta_n$ and the Jacobian $J_{\alpha_n,\beta_n}$ at $K$ satisfy
 \begin{equation*} 
\alpha_n \in ( \alpha_* ,\alpha^* ), \quad \beta_n \in (\beta_*, \beta^*), \quad  \|   J_{\alpha_n,\beta_n} \| \leq \lambda.
\end{equation*}
\item [(b)] The constants $0 \leq \alpha_* <\alpha^*<1$, $0\leq \beta_* < \beta^*<1$  satisfy
\begin{equation}
 \label{eq:choice_q_1}
\alpha_*   \in \left(  \max\left\{ 1- \frac{2}{p(1+a)}, 0 \right\},  1  \right), \quad \beta_* \in  \left(  \max\left\{ 1- \frac{2}{q(1+b)}, 0 \right\}, 1  \right),
\end{equation}
and $ \alpha_n \in ( \alpha_* ,\alpha^* )$,  $\beta_n \in (\beta_*, \beta^*)$ for any $n\in \mathbb N_0$.
\end{enumerate}
\end{lemma}

The proof is postponed  to Appendix, Section~\ref{subsec:var}.

 \section{Lyapunov Function and Global Stability for the Planar Ricker Model with PBC}
 \label{sec:Lyapglob}

So far we proved local stability, evaluating the eigenvalues of the Jacobian of the controlled system. However,  we have not justified global asymptotic stability
for positive initial conditions. 
We start with constructing a Lyapunov function in Section~\ref{subsec:Lyapunov}, its design follows the scheme in \cite{BHEBL} for the planar Ricker model. 
Later, existence of an invariant subdomain and decrease of the Lyapunov function at each step will allow to conclude convergence of all solutions starting in this subdomain 
to the equilibrium $K$. However, this is not yet sufficient for the proof of  global asymptotic stability: in addition, we have to show that the invariant subdomain attracts all positive solutions.  This is the purpose of Section~\ref{subsec:glst} below.

 \subsection{ Lyapunov function}
\label{subsec:Lyapunov}

First, we design the Lyapunov function for the values of controls exceeding minimum levels and prove that it decreases on positive 
solutions distinct from $K$.

\medskip

 Consider \eqref{eq:RickPBC} and denote 
  \begin{equation}
 \label{def:XYTknal}
 \begin{split}
 &X:=r-x-ay, \, \,\,Y:=s-bx-y, ~ T_{\alpha, \beta}(x, y):=\left( x\kappa_\alpha, \, y\nu_\beta\right), \,\,   \alpha, \beta\in [0, 1),\\
 & \kappa_\alpha= \kappa_\alpha(X):=(1-\alpha)\exp\{X\}+\alpha, \quad \nu_\beta= \nu_\beta(Y):=(1-\beta)\exp\{Y\}+\beta.
  \end{split}
 \end{equation}
Define  the Lyapunov function
   \begin{equation}
 \label{def:LyapV}
V(x, y)=bx^2+ay^2 +2abxy-2rbx-2say, \quad (x,y)\in R^2_+.
 \end{equation}
 We want to find such $\rho_1, \rho_2\in (0,1)$ that, for $\alpha\in (\rho_1, 1)$, $\beta\in (\rho_2, 1)$,
 \begin{equation}
 \label{ineq:DV}
 \Delta V_{\alpha, \beta}(x,y)=V\circ T_{\alpha, \beta}(x,y)-V(x,y)\le 0.
 \end{equation}
We have 
  \begin{equation*}
 \begin{split}
  \Delta V_{\alpha, \beta}(x,y)&=bx^2 \kappa_\alpha^2+ay^2\nu_\beta^2+2abxy\kappa_\alpha\nu_\beta-2rbx\kappa_\alpha
  -2say\nu_\beta-bx^2-ay^2 -2abxy\\+2rbx+2say &=bx^2 [\kappa_\alpha^2-1]+ay^2[\nu_\beta^2-1]+2abxy[\kappa_\alpha\nu_\beta-1]-2rbx[\kappa_\alpha-1]-2say[\nu_\beta-1].
     \end{split}
 \end{equation*} 
Since
 \[
 \kappa_\alpha^2-1=[\kappa_\alpha-1]^2-2+2\kappa_\alpha, \quad  \nu_\beta^2-1=[\nu_\beta-1]^2-2+2\nu_\beta, 
 \]
 \[
 \kappa_\alpha\nu_\beta-1=[\kappa_\alpha-1][\nu_\beta-1]-2+\nu_\beta+\kappa_\alpha,
 \]
 we arrive at
   \begin{equation}
 \label{calc:1}
 \begin{split}
&  \Delta V_{\alpha, \beta}(x,y) \\ = & bx^2 [\kappa_\alpha-1]^2+ay^2[\nu_\beta-1]^2+2abxy[\kappa_\alpha-1][\nu_\beta-1]+2bx^2[\kappa_\alpha-1]\\&+2ay^2[\nu_\beta-1]
 +2abxy[-2+\nu_\beta+\kappa_\alpha]-2rbx[\kappa_\alpha-1]-2say[\nu_\beta-1]\\
 \le & bx[\kappa_\alpha-1]^2(r-X)+ay[\nu_\beta-1]^2(s-Y)-2bx[\kappa_\alpha-1]X-2ay[\nu_\beta-1]Y\\
  = & -bx[\kappa_\alpha-1][(\kappa_\alpha-1)(X-r)+2X]-ay[\nu_\beta-1][(\nu_\beta-1)(Y-s)+2Y].
     \end{split}
 \end{equation} 
 For $u\in \mathbb R$, $\varsigma\in(0, 1)$, $t>0$ we denote
\begin{equation}
\label{def:PhiPsi}
\Phi(\varsigma, u, t):=(1-\varsigma)(1-e^u)(t-u)+2u, \quad \Psi(\varsigma, u, t):=(1-\varsigma)\left(e^{u}-1\right)\Phi(\varsigma, u, t).
\end{equation}
Set $t=r$, $\varsigma=\alpha$ and compute the derivative of $\Phi$ with respect to $u$: 
 \begin{equation*}
 \begin{split}
 \Phi'_u(\alpha, u, r):=&(1-\alpha)[-e^u(r-u)-(1-e^u)]+2=(1-\alpha)e^u(u-r+1) +1+\alpha,\\
 \Phi''_{u^2}(\alpha, u, r):=&(1-\alpha)e^u(u-r+2).
 \end{split}
 \end{equation*} 
 The second derivative $\Phi''_{u^2}(\alpha, u, r)$  is negative for $u<r-2$, positive for $u>r-2$ and $\Phi''_{u^2}(\alpha, r-2, r)=0$ at $u=r-2$, so $\Phi'_u(\alpha, u, r)$ reaches its minimum at $u_{\min}=r-2$, and 
 \[
 \Phi'_u(\alpha, r-2, r)=(1-\alpha)e^{r-2}(r-2-r+1) +1+\alpha=-(1-\alpha)e^{r-2}+1+\alpha.
 \]
 Therefore $\Phi'(\alpha, r-2, r)>0$  for $\displaystyle \alpha>\frac{e^{r-2}-1}{1+e^{r-2}}$.
 Note that if $r\le 2$, the right hand-side of the latter inequality is less than 0, so it does not impose any restrictions on $\alpha$.
Since  for $\displaystyle  \alpha>\frac{e^{r-2} - 1}{1+e^{r-2}}$,  the minimum of derivative $\Phi'_u(\alpha, u, r)$ is positive, we conclude that $\Phi(\alpha, u, r)$ increases in $u\in \mathbb R$. By definition of $\Phi$  in \eqref{def:PhiPsi} we have $\Phi(\alpha, 0, r)=0$, so 
$\Phi(\alpha, u, r)>0$ for $u>0$ and $\Phi(\alpha, u, r)<0$ for $u<0$. Since $\kappa_\alpha-1=(1-\alpha)(e^u-1)>0$ for $u>0$ and   $\kappa_\alpha-1<0$ for $u<0$, 
this implies that $\Psi(\alpha, u, t)$ defined as in \eqref{def:PhiPsi} satisfies  $\Psi(\alpha, X, r)>0$ for any $X\neq 0$ and 
$\Psi(\alpha, 0, r)=0$. 

Consider now the second term in the last line of \eqref{calc:1}.
Acting similarly when in \eqref{def:PhiPsi} we put $t=s$, $\varsigma=\beta$, we conclude that $\Phi'(\beta, s-2, s)>0$ for  $\beta>\frac{-1+e^{s-2}}{1+e^{s-2}}$, and then $\Psi(\beta, Y, s)>0$ for any $Y\neq 0$ and 
$\Psi(\beta, 0, s)=0$. 

Therefore, for $\alpha>\rho_1$,  $ \beta>\rho_2$,  where 
 \begin{equation}
\label{def:underalpha}
\rho_1:=\frac{e^{r-2}-1}{1+e^{r-2}}, \quad  \rho_2:=\frac{e^{s-2}-1}{1+e^{s-2}},
\end{equation}
we get 
\[
\Delta V_{\alpha, \beta}(x,y):=V\circ T_{\alpha, \beta}(x, y)-V(x,y)\le -bx\Psi(\alpha, X, r)-ay\Psi(\beta, Y, s)\le 0.
\]
Note  that $\rho_1, \rho_2\in (0, 1)$ and the inequality above is strict  for each $(x,y)\neq (p,q)$, $x,y>0$, since $X\neq 0$ and $Y\neq 0$ when $x\neq p$, $y\neq q$.

So we have proved
\begin{lemma}
\label{lem:alphaLyap}
Let \eqref{cond:mainposeq} hold,
  $T_{\alpha, \beta}$, $V$, $\Delta V_{\alpha, \beta}$, and $\rho_1$,  $\rho_2$,  be defined as in \eqref{def:XYTknal}, \eqref{def:LyapV}, \eqref{ineq:DV},  and \eqref{def:underalpha},  respectively.
Then, for each $\alpha\in (\rho_1, 1)$ and $\beta\in (\rho_2, 1)$, we get  
\begin{equation}
\label{ineq:DeltaV}
\Delta V_{\alpha, \beta}(x, y)\le -bx\Psi(\alpha, X, r)-ay\Psi(\beta, Y, s)\le 0,
\end{equation}
and for each $(x,y)\neq (p,q)$, $x,y>0$, we have $\Delta V_{\alpha, \beta}(x,y)<0$.
\end{lemma}



\subsection{Positive invariant domain and global stability}
\label{subsec:glst}

Next, we refer to an invariant subdomain and verify that each solution with positive initial conditions eventually gets into this domain. 
Based on it and on the decrease of the Lyapunov function at each step, we get global asymptotic stability in the first quadrant.

\subsubsection{Auxiliary results}
\label{sec:global_auxil}

Below, $\| \cdot \|_2$ denotes the Euclidean norm in ${\mathbb R}^2$. 

\begin{lemma}
\label{lem:XYxy}
Let \eqref{cond:mainposeq} hold,  $X$, $Y$ be defined as in \eqref{def:XYTknal}, $\Psi$  be defined as in \eqref{def:PhiPsi},
$\rho_1, \rho_2$ be defined as in \eqref{def:underalpha}. Let $\bar \rho_1\in (\rho_1, 1)$,  $\bar \rho_2\in (\rho_2, 1)$ be fixed.
\begin{enumerate}
\item [(a)]  If $|x-p|^2+|y-q|^2\ge \delta_0^2$, for some $\delta_0>0$, then  
\begin{equation}
\label{def:rho}
\|(X,Y)\|_2 \ge \rho_0, \quad \max\{|X|, |Y|\}\ge \rho_0/2, \quad \mbox{where} \quad \rho_0:=\frac {\delta_0(1-ab)}2.
\end{equation}
\item [(b)] If $(x,y)\in [0, e^{r-1}]\times [0, e^{s-1}]$,  there exists  $H_1>0$ such that  $|X|\le H_1$,  $|Y|\le H_1$;
\item [(c)] There is  a function $\underline \psi: (0, \infty)\to (0, \infty)$ such that for any $\alpha\in [\rho_1, \bar \rho_1]$, $\beta\in [\rho_2, \bar \rho_2]$, $\rho>0$,
\[
\Psi(\alpha, u, r)\ge \underline\psi(\rho), ~~ \Psi(\beta, u, s)\ge \underline\psi(\rho), \quad \mbox{as soon as} \quad |u|\ge \rho.
\]
\end{enumerate}
\end{lemma}

The proof is postponed to Appendix, Section~\ref{subsec:XYxy}.

\begin{remark}
\label{rem:anynorm}
In the statement of Lemma~\ref{lem:XYxy}\,(a), we applied the Euclidean norm in ${\mathbb R}^2$.  
The same result is valid for other norms. Moreover, for the maximum norm the second inequality in \eqref{def:rho}
becomes  $\max\{|X|, |Y|\}\ge \rho_0$ rather than $\max\{|X|, |Y|\}\ge \rho_0/2$.
\end{remark}

\subsubsection{Number of steps to get into a neighbourhood of $K$}

Let $\tilde \alpha, \tilde \beta$ be defined by  \eqref{def:c12}, $ \rho_1, \rho_2$ be defined by \eqref{def:underalpha}, $V$ be defined by \eqref{def:LyapV}.
We set  
\begin{equation}
\label{def:globeta}
\begin{split}
 &\eta_1:=\max\{\tilde \alpha, \rho_1\}, \quad \eta_2:=\max\{\tilde \beta, \rho_2\},\\
&\bar M:=\max \left\{ |V(x,y)|: (x,y)\in [0, e^{r-1}]\times [0, e^{s-1}] \right\}.
\end{split}
\end{equation}

\begin{lemma}
\label{lem:stepsconsta}
Let \eqref{cond:mainposeq} hold, $\eta_1$, $\eta_2$  be defined as in \eqref{def:globeta}. Choose some numbers 
\begin{equation}
\label{def:barunder}
\bar \alpha\in (\eta_1, 1), \, \, \bar \beta\in (\eta_2, 1), \,\,  \underline \alpha\in (\eta_1, \bar\alpha),\,\, \underline \beta\in (\eta_2, \bar\beta).
\end{equation}
Let  $D_{\underline \alpha, \underline\beta}$ be defined as in  \eqref{def:Dalphabeta}, $(x_0, y_0)\in D_{\underline \alpha, \underline\beta}$,
variable controls  $\alpha_n$ and $\beta_n$, for each $n\in \mathbb N$, satisfy
\begin{equation}
\label{eq:bounds_1}
\alpha_n\in (\underline \alpha, \bar \alpha), \quad \beta_n\in (\underline \beta, \bar \beta),
\end{equation}
and $(x_n, y_n)_{n\in \mathbb N}$ be a solution to  \eqref{eq:RickPBCvar}.  Then

\begin{enumerate}
\item [(a)] 
For each $\delta_0>0$ there is a number $S=S(\delta_0, \underline \alpha, \underline \beta, \bar \alpha, \bar\beta)\in \mathbb N$, which is independent of $(x_0, y_0)$  and $\alpha_n, \beta_n$ such that  $\|(x_n, y_n)-K \|<\delta_0$  for some $n\le S$.

\item [(b)] 
The solution $(x_n, y_n)$ cannot stay outside of  $B(K, \delta_0)$ for more than $S$ steps overall.

\item [(c)] $\displaystyle \lim_{n\to \infty}(x_n, y_n)=K$.
\end{enumerate}
\end{lemma}

The proof is postponed to  Appendix, Section~\ref{subsec:stepsconsta}.

Combining Lemma~\ref{lem:stepsconsta} with Lemma~\ref{lem:Tninvar} leads to the global convergence result.  

\begin{theorem}
\label{thm:glstab}
Let \eqref{cond:mainposeq} hold, $\eta_1$, $\eta_2$,  $\underline \alpha$, $\underline\beta$, $\bar \alpha$, $\bar\beta$  be defined as in \eqref{def:globeta}-\eqref{def:barunder},   
$x_0>0$, $y_0>0$. 
Let $(\alpha_n)_{n\in \mathbb N}$ and $(\beta_n)_{n\in \mathbb N}$ be two sequences, for each $n\in \mathbb N$ satisfying \eqref{eq:bounds_1}. Then for any solution $(x_n, y_n)_{n\in \mathbb N}$ of \eqref{eq:RickPBCvar},
 we have $\lim\limits_{n\to \infty}(x_n, y_n)=K$.
\end{theorem}


\section{Stochastically Perturbed  Control}
\label{sec:stoch}

Now, let us consider the case when the control is stochastically perturbed. 
After introducing stochastic notions and basic statements, 
we establish some results on local stability of the positive equilibrium in this case.
First,  local stability conditions are just aligned with the variable control case, whether stochastic or not. 
Later, we justify that local stability conditions can be improved with noise in the following sense: 
while for average control intensities, the equilibrium $K$ is unstable, introduction of noise can stabilize $K$.

\subsection{Stochastic preliminaries}
\label{subsec:stochprel}
We consider a complete filtered probability space $(\Omega, {\mathcal{F}}$, $\{{\mathcal{F}}_n\}_{n \in 
\mathbb N}, {\mathbb P})$, where the filtration $(\mathcal{F}_n)_{n \in \mathbb{N}}$ is naturally generated by 
the two sequences of mutually independent random variables, \,\, $(\xi_{n})_{n\in \mathbb N}$ \,\, and $(\chi_{n})_{n\in \mathbb N}$, i.e. 
$\mathcal{F}_{n} = \sigma \left\{\xi_{s},\,\chi_{s},\, \, s=1, \dots, n\right\}$. 
\begin{assumption}
\label{as:noise1}
$(\xi_{n})_{n\in \mathbb N}$ and $(\chi_{n})_{n\in \mathbb N}$, are sequences of  mutually independent random variables such that each sequence consists of identically distributed random variables and  $|\xi_{n}| \le 1$, $|\chi_{n}|\le 1$,  for all $n\in \mathbb N$.
\end{assumption}
\begin{assumption}
\label{as:noise2}
The sequences $(\xi_{n})_{n\in \mathbb N}$ and $(\chi_{n})_{n\in \mathbb N}$ satisfy Assumption \ref{as:noise1}  
and, for each $\varepsilon>0$, $\mathbb P \{\xi\in (1-\varepsilon, \, 1] \}>0$, $\mathbb P \{\chi\in (1-\varepsilon, \, 1] \}>0$.
\end{assumption}
The standard abbreviation ``a.s.'' is used for either ``almost sure" or ``almost surely" with respect to a fixed probability measure $\mathbb P$.  For a detailed introduction of stochastic concepts and notations we refer the reader to \cite{Shiryaev96}.

In this section,  we consider the system with stochastically perturbed controls
\begin{equation}
 \label{eq:RickPBCstoch}
 \left\{
 \begin{array}{ll}
x_{n+1}=&x_n\left [(1-\alpha - \ell\xi_{n+1})\exp\{r-x_n-ay_n\}+\alpha+\ell\xi_{n+1} \right],\\
y_{n+1}=&y_n\left[(1-\beta - \bar \ell\chi_{n+1})\exp\{s-bx_n-y_n\}+\beta+\bar \ell\chi_{n+1} \right],
\end{array}\right.
 \end{equation}
where $\xi_{n}$  and $\chi_{n} $ satisfy Assumption \ref{as:noise1} and 
 \begin{equation}
\label{cond:ellhatell}
\alpha, \beta\in [0, 1), \quad \ell\in [0, \min\{\alpha, 1-\alpha\} ), \quad \bar \ell\in [0, \min\{\beta, 1-\beta\} ).
\end{equation}
In other words, the parameters $\alpha$ and $\beta$ of the correspondent control matrix $\Upsilon_{\alpha, \beta}$, see \eqref{def:Ups}, are stochastically  perturbed.

To obtain local stability  for system \eqref{eq:RickPBCstoch}, we estimate the norms  of products of  successive matrices $\left\| J_{\alpha_1, \beta_1}\cdot \dots \cdot  J_{\alpha_n, \beta_n}  \right\|$  on a set $\Omega_\gamma$ with probability arbitrarily close to one $\mathbb P(\Omega_\gamma)>1-\gamma$. 
Our approach is based on the application of the  Kolmogorov's Law of Large Numbers as in \cite[Page 391]{Shiryaev96},
see the general approach to local stabilization of systems subject to noisy PBC control in \cite {pbclocsys}.

\begin{lemma}
\label{lem:Kolm}[Kolmogorov's Law of Large Numbers]
Let $(\upsilon_{n})_{n\in\ \mathbb N}$ be a sequence of independent identically distributed random variables, where $\mathbb E |\upsilon_n|<\infty$, $n\in\mathbb{N}$, their common mean is $\mu:=\mathbb E \upsilon_n$, and the partial sum is $\displaystyle S_n:= \sum_{k=1}^n \upsilon_k$.	Then $
\displaystyle \lim_{n\to\infty} \frac{S_n}{n} = \mu$, a.s.
\end{lemma}

\begin{corollary}
\label{cor:Kolm1}
Let the  assumptions of Lemma \ref{lem:Kolm} hold and  $\mathbb E \upsilon_n=\mu$. Then, for each $\gamma\in (0, 1)$ and $\varepsilon>0$, there exist a nonrandom $\rm N=\rm N(\gamma, \varepsilon)$ and  $\Omega_\gamma=\Omega_\gamma(\varepsilon) \subset \Omega$ with $ \mathbb P(\Omega_\gamma)>1-\gamma$, such that, for $n\ge \rm N$, on $\Omega_\gamma$, 
\begin{equation*}
(\mu-\varepsilon) n\le \sum_{k=1}^n \upsilon_k\le (\mu+\varepsilon) n. 
\end{equation*}
\end{corollary}

A version  of the following lemma was justified in \cite{BRAllee} as a corollary of the Borel-Cantelli Lemma, and is used in the proof of global stability, Theorem~\ref{thm:globstabst0}.

\begin{lemma}
\label {lem:topor} 
Let the sequence $(\xi_n, \chi_n)_{n\in \mathbb N}$ satisfy Assumptions \ref{as:noise1}-\ref{as:noise2}. Then, for each nonrandom $S\in \mathbb N$, $\varepsilon\in (0, 1)$ and a random moment $\mathcal M$, we have  the probability\\
$
\displaystyle \mathbb P\Big\{ \mbox{there is random }   \mathcal N>\mathcal M:
~ \displaystyle (\xi_{\mathcal N+i}, \chi_{\mathcal N+i}) \in (1-\varepsilon, 1]\times (1-\varepsilon, 1], \,\,  i=0, 1, \dots, S  \Big\}=1.
$
\end{lemma}


\subsection{On local stability with stochastic control}
\label{sec:locstoch}

In the next lemma we consider the case when the perturbed control is in the domain of parameters guaranteeing local stability, and $(x_0,y_0)$ is close enough to the equilibrium. 
\begin{lemma}
\label{lem:locstabst0} 
Let \eqref{cond:mainposeq}, \eqref{cond:ellhatell} and Assumption~\ref{as:noise1} hold, and there exist $\lambda \in (0,1)$  and a norm $\| \cdot \|$ in ${\mathbb R}^2$ such that for every $n \in {\mathbb N}$, a.s.
\begin{equation}
\label{eq:Jacobian_norm0}
 \|   J_{\alpha+\ell \xi_n,\beta+\bar\ell \chi_n} \| \leq \lambda.
\end{equation}
Then there is $\delta_0 > 0$ such that any solution $(x_n,y_n)$  of \eqref{eq:RickPBCstoch}  with $(x_0,y_0) \in   B(K, \delta_0)$ satisfies $\displaystyle \lim_{n \to \infty} (x_n,y_n) =K$ a.s. 
\end{lemma}

\begin{proof}
Denote $\alpha^* = \alpha + \ell$, $\beta^* = \beta+ \bar \ell$, $\alpha_* = \alpha - \ell$, $\beta_* = \beta- \ell$, and consider the  variable controls $\alpha_n= \alpha+\ell \xi_{n+1}$, $\beta_n =  \beta + {\bar \ell} \chi_{n+1}$ corresponding to any realization of $\xi_n$, $\chi_n$.  Then, $\alpha_n \in [\alpha_*, \alpha^*]$ and $ \beta_n \in [\beta_*, \beta^*]$ satisfy, a.s.,  all the conditions of Lemma~\ref{lem:locstabst1_var}, Part (a). Thus,  
for any solution $(x_n,y_n)$  with $(x_0,y_0) \in   B(K, \delta_0)$
and $\alpha_n \in [\alpha_*, \alpha^*]$, $ \beta_n \in [\beta_*, \beta^*]$  we have  $\displaystyle \lim_{n \to \infty} (x_n,y_n) =K$. 
\end{proof}

\begin{remark}
\label{cor:stlocst}
Let \eqref{cond:mainposeq}, \eqref{cond:ellhatell} and Assumption~\ref{as:noise1} hold. 
Set $\alpha_*:=\alpha-\ell$, $\beta_*:=\beta-\bar \ell$, $\alpha^*:=\alpha+\ell$, $\beta^*:=\beta+\bar \ell$, then $0<\alpha_*<\alpha^*<1$, $0<\beta_*<\beta^*<1$, $\alpha_n \in [\alpha_*, \alpha^*]$ and $ \beta_n \in [\beta_*, \beta^*]$. 
 If $\alpha_*, \beta_*$ satisfy conditions \eqref{eq:choice_q_1} then the assumptions of Lemma~\ref{lem:locstabst1_var}\,(b) hold a.s., so  the conclusion of Lemma~\ref{lem:locstabst0} is valid, i.e., 
 there exists $\delta_0 > 0$ such that any solution $(x_n,y_n)$  of \eqref{eq:RickPBCstoch}  with $(x_0,y_0) \in   B(K, \delta_0)$ converges to the positive equilibrium  $\displaystyle \lim_{n \to \infty} (x_n,y_n) =K$ a.s.
\end{remark}

Now we proceed to the stochastic version of the local asymptotic stability lemma, when after the application of PBC with stochastic control, a certain number of consecutive  iterations of the corresponding variable map are contractions  with a positive probability, keeping a solution in any prescribed $\delta$-neighbourhood of the positive equilibrium $K$. 
Since  one of the focuses of this paper is the situation when  global stability is achieved by application of a noisy control, a solution enters  this neighborhood of $K$ for the first time at a random moment $\tau$.  We mostly  deal with $S=S(\delta)$, defined similarly to the statement of  Lemma~\ref{lem:stepsconsta}\, (a) (see more details later, in the proof of Theorem~\ref{thm:globstabst0}), which is the maximum number of steps necessary for a solution  of  \eqref{eq:RickPBCvar} to reach $B(\delta, K)$. 
This allows us 
to consider only a bounded  random moment $\tau\le S$. 
Note that the  random variable $\tau$ takes values in $\mathbb N_0$, is defined by the initial value $(x_0,y_0)$, $\delta>0$, the  lower and the upper bounds of control parameters, and   is  the moment of the first entrance of a solution into  the open ball $B(\delta, K)$.

Following the method of \cite{BRMedv,pbclocsys}, we formulate two results: Lemma~\ref{lem:locKolmtauk} where the system is restricted to the set $\Omega_{\tau k}$, where the moment $\tau$ is constant, $\tau=k$, and Theorem~\ref{thm:locKolm} on  local asymptotic stability, where, after entering the neighbourhood of the equilibrium $K$ at the random moment $\tau$, the solution converges to $K$ with a prescribed in advance probability close to 1.   The proofs of  Lemma~\ref{lem:locKolmtauk}  and Theorem~\ref{thm:locKolm} follow the steps similar to the proofs of  \cite[Lemma A.3]{BRMedv}  and  \cite[Theorem 7]{pbclocsys}, but for completeness we present most part of the proof of   Lemma~\ref{lem:locKolmtauk}   in Appendix, see Section~\ref{subsec:locKolmtauk}.

Let function $T$ be defined as in \eqref{def:Talphabeta}, $\tilde U=(x, y)$ and the Jacobian $J_{\alpha, \beta}$  be defined as in \eqref{Jacobian_control}, 
where  $J_{0,0}$  corresponds  to the no-control case. Since all the functions involved into \eqref{def:Talphabeta} are infinitely differentiable, there is a neighborhood $B(K, \delta_1)$ of the equilibrium $K$, a function $G: {\mathbb  R}^2\to {\mathbb R}^2$ and a constant $C=C(\delta_1)>0$ such that
\begin{equation}
\label{def:Tineq}
T(\tilde U)=K+ J _{0,0}[\tilde U-K]+G(\tilde U-K), \quad  |G(Y)|\le C|Y|^2, \quad \mbox{for} \quad | Y|\le \delta_1.
\end{equation}
For simplicity we denote $U:=\tilde U-K$, so that the map $T(U)=J _{0,0}U+G(U)$ has an  equilibrium at zero.

For  the diagonal matrix $\Upsilon(\alpha, \beta)$ defined as in \eqref{def:Ups},  $\alpha_n=\alpha+\ell \xi_{n+1}$, $\beta_n=\beta+\bar \ell \chi_{n+1}$,  and $J_{\alpha_n, \beta_n}$ defined as in \eqref{Jacobian_control_var}, we have
\[
J_{\alpha_{n}, \beta_{n}}=[I-\Upsilon_{\alpha_n, \beta_{n}} ]J _{0,0}+ \Upsilon_{\alpha_{n}, \beta_{n}},
\] 
so  system \eqref{eq:RickPBCstoch} can be written as 
\begin{equation}
\label{eq:RickPBCst1}
\begin{split}
&Z_{n+1}=J_{\alpha_{n}, \beta_{n}}Z_n+\hat G(\alpha_{n}, \beta_{n},Z_n), \\
&\mbox{where }\quad 
Z_n=(x_n-p, y_n-q), \quad \hat G(\alpha_{n}, \beta_{n},Z_n)=[I-\Upsilon_{\alpha_{n}, \beta_{n}} ]G(Z_n).
\end{split}
\end{equation}
Note that there exists a constant $C_1$ such that, for all $n\in \mathbb N$, 
$
\|I-\Upsilon_{\alpha_{n}, \beta_{n}}\|\le C_1,
$ 
therefore, for any $\|Z_n\|\le \delta$ and $C=C(\delta)$  as in \eqref{def:Tineq}, we get
\begin{equation}
\label{ineq:hatC}
\|\hat G_n \|\le  \hat C \|Y \|, \quad \mbox{where} \quad \hat C=C_1 C(\delta).
\end{equation}
The next condition  is the key assumption for our  local stability results.

\begin{assumption}
\label{as:Kolmloc}
 For $J_{\alpha, \beta}$  defined as in \eqref{Jacobian_control}, there exist $\nu>0$ and
$\alpha$, $\beta$, $\ell$, $\bar \ell$ satisfying \eqref{cond:ellhatell} such that
\begin{equation}
\label{cond:Eln}
\mathbb E\ln \|J_{\alpha+\ell \xi, \beta+\bar\ell \chi}\|=-\nu<0.
\end{equation}
\end{assumption}

\begin{remark}
\label{rem:stochdet}
Note that condition  \eqref{eq:Jacobian_norm0} implies inequality \eqref{cond:Eln}. 
In other words, the case when the norm of the variable Jacobian   can be estimated by the constant $\lambda\in (0, 1)$ is a partial case of condition \eqref{cond:Eln}. 
\end{remark}

Now we introduce some additional values and get auxiliary relations.

Fix $\gamma\in (0, 1)$, then, for  $\nu$ from Assumption \ref{as:Kolmloc}, by Corollary \ref{cor:Kolm1} of  
Lemma~\ref{lem:Kolm}, we find $\rm N=\rm N(\gamma, \nu)$ such that, for $\Lambda_{\gamma, 0}$ defined below,
\begin{equation}
\label{def:POgam}
{\mathbb P} \Lambda_{\gamma, 0}= {\mathbb P}  \left\{- \frac{3\nu n}{2} < \sum_{i=0}^n\ln \|J_{\alpha+\ell \xi_i, \beta+\bar\ell \chi_i}\| <- \frac{\nu n}{2}, \,\, \mbox{for all} \,\, n\ge \rm N\right\}>1- \frac{\gamma}{2}.
\end{equation}
Without loss of generality, we can assume that $\displaystyle \rm N > \frac{2 \ln 3}{\nu}  -1$ is large enough, so that ~
$\displaystyle \frac \nu 2-\frac {\ln 3}{\rm N+1}>0$.

For $\rm N$ as in \eqref{def:POgam}, $\nu$ from \eqref{cond:Eln}, $\delta_1$ from \eqref{def:Tineq}, $\hat C$ from \eqref{ineq:hatC}, we set 
 \begin{equation}
 \label{def:ethamu}
 \begin{split}
 &\mathbb J:=\max\{\|J_{\alpha+\ell e, \beta+\bar\ell \bar e} \|~~ : ~ \,\, e, \bar e\in [-1,1]\}, \\
 &\mu:=\min\left\{1, \, \frac \nu 2-\frac {\ln 3}{\rm N+1}\right\}, \quad \mathbb M:=\max_{n\ge N}\left\{\hat  Ce^{3\mu}e^{-\mu n}\sum_{j=N}^{n-1}e^{\mu j}e^{-j (\nu/2-\mu)}\right\},\\
 &\eta:=\frac 13 \max\left\{\delta_1, \quad  \left(
\hat C \mathbb J^{\rm N+1}\sum_{j =-1}^{\rm N-1} e^{3\mu+2\mu j }\right)^{-1}, \,\, \mathbb M^{-1}\right\},\\
&\delta<\min\left\{\delta_1, 1, \,\, \left( \mathbb J+\hat C\right)^{-1}\eta e^{-\mu}, \,\,\frac 13\eta e^{-\rm N-1}\mathbb J^{-\rm N-1} \right\}.
 \end{split}
 \end{equation}
 Note that $\mathbb M$ defined on the second line of \eqref{def:ethamu} is finite,  since 
$$
e^{-\mu n}\sum_{\tau=N}^{n-1}e^{\mu \tau}e^{-\tau(\nu/2-\mu)}  \leq   
\sum_{j=1}^{n-N} e^{-\mu j} < 
\frac{e^{-\mu}}{1-e^{-\mu}}. 
$$

Let  $(x, y)$ be a solution to \eqref{eq:RickPBCstoch} with the initial value $(x_0, y_0)$, 
\begin{equation}
\label{def:tau}
\tau=\tau(\omega, \delta):=\inf\{n\in \mathbb N: (x_{n(\omega)}, y_{n(\omega)})\in B(\delta, K)\}, \quad  \omega\in \Omega, 
\end{equation}
and  recall that $\tau\le S=S(\delta)$, see Lemma~\ref{lem:stepsconsta}\, (a). 
For $\rm N$ defined as in \eqref{def:POgam} and for each $k=0, \dots, S$,  set
\begin{equation}
\label{def:POgamk}
\begin{split}
&\Omega_{\tau k}=\{\omega\in \Omega:\tau=k\}, \quad \mbox{so} \quad \Omega=\bigcup_{k=0}^S\Omega_{\tau k},\\
&\Lambda_{\gamma, k}:=\left\{-  \frac{3\nu n}{2} < \sum_{i=k}^{k+n}\ln \|J_{\alpha+\ell \xi_i, \beta+\bar\ell \chi_i}\| <- \frac{\nu n}{2} \,\, \mbox{for all} \,\, n\ge \rm N\right\}.\\
\end{split}
\end{equation}
Since $\Omega_{\tau k}$ are mutually exclusive, we also have $\sum_{i=0}^S {\mathbb P} \Omega_{\tau k}=1.$
In general, $\Lambda_{\gamma, k}\neq \Lambda_{\gamma, j}$ for $k\neq j$, but since $\ln \|J_{\alpha+\ell \xi_i, \beta+\bar\ell \chi_i}\|$ are identically distributed, we have ${\mathbb P} \Lambda_{\gamma, k}= {\mathbb P} \Lambda_{\gamma, j}$ for all $k,j \in \mathbb N_0$, where ${\mathbb P}\Lambda_{\gamma, 0}$ is estimated in \eqref{def:POgam}. Then for all $k=0, 1, \dots, S$
and $n\ge \rm N$ we have $ {\mathbb P}   \Lambda_{\gamma, k}>1-\gamma/2$, and, on $ \Lambda_{\gamma, k}$,
\begin{equation}
\label{est:Jprod}
 \left\|\prod_{i=k}^{k+n}J_{\alpha+\ell \xi_i, \beta+\bar\ell \chi_i} \right\|\le \prod_{i=k}^{k+n}\|J_{\alpha+\ell \xi_i, \beta+\bar\ell \chi_i}\|\le e^{\sum_{i=k}^{k+n}\|J{\alpha+\ell \xi_i, \beta+\bar\ell \chi_i} \|}\le e^{-\nu n/2}.
\end{equation}

 Now, for each $k=0, 1, \dots, S$, we consider the following modification of  system   \eqref{eq:RickPBCst1} 
 \begin{equation}
\label{eq:RickPBCstk}
\begin{split}
&Z^{[k]}_{n+1}=J_{\alpha_{n}(k), \beta_n(k)}Z^{[k]}_n+\hat G\left(\alpha_n(k), \beta_n(k), Z^{[k]}_n\right), \quad Z^{[k]}_0=(\theta_1, \theta_2),\\
&\mbox{where }\quad \alpha_n(k)=\alpha+\ell \xi_{k+n+1}, \,\, \beta_n(k)=\beta+\bar \ell \chi_{k+n+1}.
\end{split}
\end{equation} 

\begin{lemma}
\label{lem:locKolmtauk} 
Let $\gamma\in (0, 1)$, Assumptions~\ref{as:noise1},  \ref{as:Kolmloc},  and condition \eqref{cond:ellhatell} hold. 
Let the numbers $\eta$, $\mu$, $\delta$ and sets $\Omega_{\tau k}$, $\Lambda_{\gamma, k}$  be defined by \eqref{def:ethamu} and \eqref{def:POgamk}, respectively, and
 $Z^{[k]}$  be a solution to \eqref{eq:RickPBCstk} with the initial value $Z^{[k]}_0=(\theta_1, \theta_2)$, where $(\theta_1, \theta_2)$ is a random variable such that $(\theta_1, \theta_2)\in B(\delta, \bar 0)$ on $\Omega_{\tau k}$.
Then, on $\Lambda_{\gamma, k}\cap \Omega_{\tau k}$, 
\[
\|Z^{[k]}_n\|\le \eta e^{-\mu n}, \quad n\in \mathbb N.
\]
\end{lemma}

The proof of Lemma~\ref{lem:locKolmtauk} is deferred to Appendix, see Section~\ref{subsec:locKolmtauk}.
 
Assume that a solution to \eqref{eq:RickPBCst1} either starts in $B(\delta, K)$, or an arbitrary control method brings it into $B(\delta, K)$.

\begin{theorem}
\label{thm:locKolm} 
Let $\gamma\in (0, 1)$, Assumptions~\ref{as:noise1},  \ref{as:Kolmloc}, conditions \eqref{cond:ellhatell} hold, and a random moment $\tau$ br defined as in \eqref{def:tau}.
Let the numbers $\eta$, $\mu$, $\delta$ be defined by \eqref{def:ethamu}, and
 $(x, y)$ be a solution to \eqref{eq:RickPBCstoch}  with the initial value $(x_\tau, y_\tau)=(\theta_1, \theta_2)\in B(\delta, K)$. Then 
\begin{enumerate}
\item [(i)] there exists a  set $\Omega_\gamma$, $\mathbb P[\Omega_\gamma]>1-\gamma$, such that, on $\Omega_\gamma$,
\begin{enumerate}
\item [(a)]$ \|(x_m, y_m)-K\|\le \eta e^{-\mu(m-\tau)}$ for all $m\ge \tau$,
\item [(b)] $\lim\limits_{m\to \infty}(x_m, y_m)=K$;
\end{enumerate}
\item [(ii)] there exist a  set $\Omega_\gamma^{[1]}$, $\mathbb P\left[\Omega_\gamma^{[1]}\right]>1-\gamma$ and a nonrandom $n_0\in \mathbb N_0$  such that, on $\Omega_\gamma^{[1]}$, 
\[
\|(x_m, y_m)-K\|\le \eta e^{-\mu(m-n_0)}, \mbox{~~for all ~~} m\ge n_0.
\]
\end{enumerate}

\end{theorem}
The proof of this theorem is based on Lemma~\ref{lem:locKolmtauk},   it follows closely   the proof of  \cite[Theorem 3.2]{BRMedv}
and thus is omitted.

\bigskip

\subsection{Improvement of global stability with noise}
\label{subsec:glstoch}

Recall that controls $\tilde \alpha, \tilde \beta$,  defined by \eqref{def:c12}, are sufficient for the entrance of the solution into  the invariant subdomain $\mathcal D_{\tilde \alpha, \tilde \beta}$ separated from zero, while $\rho_1$, $\rho_2$ are global stability bounds, determined  by \eqref{def:underalpha} and justified using the Lyapunov function in Section~\ref{subsec:Lyapunov}. 

We consider the case when not only the lower bounds of the stochastic control, $\alpha-\ell$, $\beta-\bar\ell$, are  not in the global stabilization domain, but even the average control bounds   without noise,  $\alpha, \beta$, do not guarantee stabilization of the positive equilibrium $K$. The lower control bound, $\alpha-\ell$, $\beta-\bar\ell$,  however, should exceed the values sufficient for the entrance into  $\mathcal D_{\tilde \alpha, \tilde \beta}$, i.e. $\tilde \alpha, \tilde \beta$.
 For successful entrance in the local convergence domain, the upper bounds of the control, $\alpha+\ell$, $\beta+\bar\ell$,  should exceed global stability bounds $\rho_1,\rho_2$.

In this section, we assume that  $\rho_1>\tilde{\alpha}$  or  $\rho_2>\tilde{\beta}$. 
Since the lower bounds of the controls should exceed $\tilde \alpha, \tilde \beta$, i.e. $\alpha-\ell>\tilde \alpha$, $\beta-\bar\ell>\tilde \beta$, we cannot  improve anything when  $\rho_1=\tilde{\alpha}$ or $\rho_2=\tilde{\beta}$. 
The next theorem, which  is the main result of the paper, shows that introduction of noise into control can decrease  lower bounds 
 for the control parameters  which guarantee that the solution will get into $B(\delta, K)$ with any $\delta>0$. After entering the ball $B(\delta, K)$  with appropriate small $\delta$,  condition \eqref{cond:Eln} starts acting and further, by Lemma~\ref{lem:locKolmtauk} and Theorem~\ref{thm:locKolm},   pushing the solution towards  $K$.

\begin{theorem}
\label{thm:globstabst0}
Let \eqref{cond:mainposeq} and Assumption \ref{as:noise2} hold, $\tilde \alpha$, $\tilde \beta$, $\rho_1$, $\rho_2$  be defined as in \eqref{def:c12}, \eqref{def:underalpha} and $\tilde \alpha<\rho_1$, $\tilde \beta < \rho_2.$
Assume  that 
\begin{equation}
\label{cond:ael0}
\begin{split}
&\alpha\in\left(\frac{\tilde \alpha+\rho_1}2,\,  \rho_1\right), \quad \ell\in\left(\rho_1-\alpha, \, \min\{\alpha-\tilde \alpha, 1-\alpha    \} \right),
\\
&\beta\in\left(\frac{\tilde \beta+\rho_2}2, \, \rho_2\right), \quad \bar\ell\in\left(\rho_2-\beta,  \, \min\{\beta-\tilde \beta, 1-\beta    \} \right),
\end{split}
\end{equation}
and $(x_n, y_n)$ is  a solution to system \eqref{eq:RickPBCstoch} with $\alpha_n=\alpha+\ell \xi_{n+1}$, $\beta_n=\beta+\bar \ell \chi_{n+1}$ and the initial value $(x_0, y_0)\in \mathcal D_{\tilde \alpha, \tilde\beta}$.
Then
\begin{enumerate}
\item [(a)]  For each $\delta>0$, there exists a bounded random moment $\tau=\tau(\delta)$  such that $(x_\tau, y_\tau)\in B(\delta, K)$ a.s. on $\Omega$.
\item [(b)] If Assumption \ref{as:Kolmloc} holds for parameters $\alpha$, $\beta$, $\ell$,  $\bar \ell$ satisfying \eqref{cond:ael0}  then $\lim\limits_{n\to \infty}(x_n, y_n)=K$ for any initial values $x_0, y_0>0$, a.s.

\end{enumerate}
\end{theorem}

\begin{proof}
 
Note that the intervals in the right-hand sides of each line in \eqref{cond:ael0} are not empty, since, by the first relations in the corresponding  lines in \eqref{cond:ael0}, we have  $\rho_1-\alpha<\alpha-\tilde \alpha$, $\rho_2-\beta<\beta-\tilde \beta$.

{\bf (a)} For any  $(\alpha, \ell)$ satisfying \eqref{cond:ael0} we have   
 \begin{equation*}
 \alpha-\ell>\tilde \alpha, \quad  \beta-\bar \ell>\tilde \beta, \quad  \rho_1<\alpha+\ell<1, \quad \rho_2<\beta+\bar \ell<1, \\
 \end{equation*} 
so we can set
\begin{equation}
 \label{def:lowboundglstab}
 \begin{split}
& \varepsilon:=\min\{\alpha+\ell-\rho_1, \,  \beta+\bar \ell-\rho_2\}>0, \quad \\
&\underline \alpha:=\alpha+\ell-\varepsilon/2>\rho_1, \,\,\, \underline \beta:=\beta+\bar \ell-\varepsilon/2>\rho_2,\quad \bar \alpha:=\alpha+\ell, \,\, \, \bar \beta:=\beta+\bar \ell,\\
&\hat \Omega=\left\{\omega\in \Omega:(\xi(\omega), \chi(\omega))\in \left(1-\frac \varepsilon{2\ell}, 1\right]\times \left(1-\frac \varepsilon{2\bar \ell}, 1\right]\right\}.
 \end{split}
 \end{equation}
 By Assumption~\ref{as:noise2}, $ \mathbb P(\hat \Omega)>0$.  Note that, on $\hat \Omega$, for $\underline \alpha, \underline \beta, \bar \alpha, \bar \beta$ defined as in \eqref{def:lowboundglstab}, we have 
 \[
1> \bar \alpha=\alpha+\ell\ge  \alpha_n=\alpha+\ell \xi_n\ge \alpha+\ell\left(1-\frac \varepsilon{2\ell}\right)=\alpha+\ell -\varepsilon/2= \underline \alpha, 
 \]
and, similarly,  $1>\bar\beta\ge \beta_n\ge \underline\beta$.  

Fix some $\delta>0$ and let the number $S=S(\delta, \underline \alpha, \underline \beta, \bar \alpha, \bar \beta)$ be as in Lemma~\ref{lem:stepsconsta}\,(a)   for $\underline \alpha$, $\underline \beta, \bar\alpha, \bar\beta$ defined as in \eqref{def:lowboundglstab}.
By Lemma~\ref{lem:topor}, there is a random moment $\mathcal N=\mathcal N(\delta)$ such that, for each $i=0, 1, \dots, S$,
    \begin{equation*}
    (\xi_{\mathcal N+i}, \chi_{\mathcal N+i})\in \left(1-\frac \varepsilon{2\ell}, 1\right]\times \left(1-\frac \varepsilon{2\bar \ell}, 1\right].
     \end{equation*}	
 Fix  some $j\in \mathbb N$, set 
\[
\Omega_j:=\{\omega\in \Omega: \mathcal  N=j \}=\left\{\omega\in \Omega: (\xi_{j+i}, \chi_{j+i})\in \left(1-\frac \varepsilon{2\ell}, 1\right]\times \left(1-\frac \varepsilon{2\bar \ell}, 1\right],\, i=0, \dots, S \right\}.
\]
Note that $\Omega_j$ is defined  by $ (\xi_{j}, \chi_{j}), (\xi_{j+1}\chi_{j+1}),  \dots,  (\xi_{j+S}, \chi_{j+S} )  $ and  $\displaystyle \cup_{j=0}^S\Omega_j=\Omega$.

Assume first that $(x_0, y_0)\in \mathcal D_{\tilde \alpha, \tilde \beta}(h)$, where $h$ is chosen as in Lemma~\ref{lem:Tninvar},  and let $(z, u)$ be a solution to  the system 
\begin{equation}
 \label{eq:RickPBCvarzu0}
 \left\{
 \begin{array}{ll}
z_{n+1}=&z_n\left [(1-\alpha_{n+1})\exp\{r-z_n-au_n\}+\alpha_{n+1}\right],\\
u_{n+1}=&z_n\left[(1-\beta_{n+1})\exp\{s-bz_n-u_n\}+\beta_{n+1}\right], \,\, n\in\mathbb N,
\end{array}\right.
 \end{equation}
where $\alpha_n=\alpha+\ell \xi_{j+n}$,  $\beta_n=\beta+\bar \ell  \xi_{j+n}$ are  considered path-wise  on $\Omega_j$. 
Since for $\underline \alpha, \underline \beta$  defined as in \eqref{def:lowboundglstab} and for all $n\le S$,
 we have  $\alpha_n\ge \underline \alpha>\rho_1$, $\beta_n\ge \underline \beta>\rho_2$, by Lemma~\ref{lem:stepsconsta}~(a), there exists $n_j=n_j(\omega)\le S$, $\omega \in \Omega_j$ s.t $(x_{n_j}, y_{n_j})\in B(K,\delta)$.

If, however, $(x_0, y_0)\notin \mathcal D_{\tilde \alpha, \tilde \beta}(h)$, by Lemma~\ref{lem:Tninvar}~(c), there is a non-random $\tilde S$ such that $(x_n, y_n)\in \mathcal D_{\tilde \alpha, \tilde \beta}(h)$ for all $n\ge \tilde S$. Recall that $\tilde S$ is defined only by $(x_0, y_0)$ and the lower and the upper bounds of the variable controls, which in this case are $\underline \alpha$, $\underline \beta, \bar\alpha, \bar\beta$, so it is the same for all $\omega\in \Omega$.  Therefore we start our consideration from the moment $\tilde S$, i.e. now the initial value for \eqref{eq:RickPBCvarzu0} is $(x_{\tilde S}, y_{\tilde S})$.

Based on the above, we conclude that for each $\omega\in \Omega$, there is a random moment $n=n(\omega)\le S$ such that $(x_{n(\omega)}, y_{n(\omega)})\in B(\delta, K)$.
 Let $\tau$ be the moment when the solution $(x_n, y_n)$ reaches the ball $ B(\delta, K)$  for the first time:
 \begin{equation}
 \label{def:tau1}
\tau=\tau(\omega):=\min\{n\le N(\delta): (x_{n(\omega)}, y_{n(\omega)})\in B(\delta, K)\}, \quad  \omega\in \Omega.
\end{equation}
So   $\tau\in \mathbb N$ is bounded. Also, the random variables $\xi_{k+i}, \chi_{k+i}$, $i\in \mathbb N$, are independent from the event  $\{\tau=k\}$, for each $k\in \mathbb N_0$.

\bigskip

{\bf (b)} Assume that there is a set $\Omega^*$, ${\mathbb P} \Omega^*>\varepsilon_0>0$ such that$\lim\limits_{n\to \infty}(x_n, y_n)\neq K$  and choose $\gamma<\varepsilon_0/2$.  By Theorem~\ref{thm:locKolm} with $\tau$ defined by \eqref{def:tau1}, there is $\Omega_\gamma$ , ${\mathbb P} (\Omega_\gamma)>1-\gamma>1-\varepsilon_0/2$ and $\delta$, defined by the last line in \eqref{def:ethamu},  such that  $\lim\limits_{n\to\infty}(x_n, y_n)=K$ on $\Omega_\gamma$.

Since $\Omega^*\subseteq \Omega\setminus \Omega_\gamma$ and $\mathbb P\left[  \Omega\setminus \Omega_\gamma\right] =1-\mathbb P\Omega_\gamma\le 1-(1-\gamma)=\gamma$, we arrive at the contradiction:
\[
\varepsilon_0 <{\mathbb P} \Omega^*\le {\mathbb P} \left[ \Omega\setminus \Omega_\gamma  \right]\le \gamma\le \varepsilon_0/2.
\]
\end{proof}
While in Part (a) of Theorem~\ref{thm:globstabst0}  and also in the following corollary, we justify local stabilization, numerical examples demonstrate that in fact we get global, not local asymptotic stability.

\begin{corollary}
Let $(\tilde x_n, \tilde y_n)$ be a solution to system \eqref{eq:RickPBCvar} with  variable controls $\tilde \alpha_n$, $\tilde \beta_n$ satisfying $\tilde \alpha_n \in [\alpha-\ell,\alpha+\ell]$, $\tilde \beta_n \in [\beta-\bar \ell,\beta+\bar\ell]$, where $\alpha, \beta, \ell, \bar \ell$ are from  \eqref{cond:ael0}. If there exists $\delta_0>0$ such that, as soon as $(\tilde x_0, \tilde y_0) \in B(K, \delta_0)$, we get  $\lim\limits_{n\to \infty}(\tilde x_n, \tilde y_n)=K$, we have $\lim\limits_{n\to \infty}(x_n, y_n)=K$ for any initial values $x_0, y_0>0$, a.s.
\end{corollary}

\begin{proof}
The proof of the corollary  mostly repeats the proof of Part {\bf (a)} of Theorem~\ref{thm:globstabst0}.
 The only difference is that system \eqref{eq:RickPBCvarzu0}  now is considered path-wise on all $\Omega$.  The result follows since, by Part\,{\bf (a)}, for each $\omega\in \Omega$ 
we get $(x_\tau, y_\tau)\in B(\delta_0, K)$, while controls satisfy $\alpha(\omega)\in [\alpha-\ell, \alpha+\ell]$ and  $\beta(\omega)\in [\beta-\bar \ell, \beta+\bar\ell]$.
\end{proof}

\bigskip

\subsection{The case of equal controls}
\label{sec:ab}

In this section,  we consider a simpler case when
the same control intensity is applied to each equation in \eqref{eq:Rick}, i.e. $\alpha_n=\alpha+\ell \xi_{n+1}=\beta_n$. 
In this case, the control matrix $\Upsilon$ defined in \eqref{def:Ups} becomes a scalar matrix $\Upsilon = \kappa I$.
This particular type of deterimniistic PBC was investigated in \cite{LP2014}.

In \cite{pbclocsys}, adaptive control values along different eigenspaces were designed applying  an invertible matrix $P$ to diagonalize $J_{0,0}$.
In the case of a scalar matrix, the transformed control matrix is the same as the original one $\tilde \Upsilon=P\Upsilon P^{-1} = \Upsilon = \kappa I$.
The choice of $\alpha_n=\beta_n$ can significantly simplify  calculations, but it is not optimal in the sense that we might apply stronger control intensities than  necessary for stabilization to one of the variables $x, y$. In particular, we have to choose $ \kappa = \max\{\alpha, \beta\}$.

The eigenvalues of  the Jacobian $J_{0, 0}$ without control, see \eqref{Jacobian_control}, are
\[
\lambda_{\max} =\frac {2-p-q+\sqrt{(p-q)^2+4abpq}}{2}, \quad \lambda_{\min}=\frac {2-p-q-\sqrt{(p-q)^2+4abpq}}{2} \, .
\]
Since $(p-q)^2+4abpq\le (p+q)^2-4pq(1-ab)< (p+q)^2$, we have $\lambda_{\min}<\lambda_{\max}<1+(p+q)/2-(p+q)/2=1$. 
We need to stabilize, once $\lambda_{\min} < -1$.

Applying \cite[Section 4.2, formula (4.2)]{pbclocsys}, we get the deterministic control parameter ${\mathcal A}$  for which $K$ is locally asymptotically stable:
\begin{equation}
\label{def:mathcalA}
\alpha>{\mathcal A}  :=\frac{-1-\lambda_{\min}}{1-\lambda_{\min}}
=1-\frac{4}{p+q +\sqrt{(p-q)^2+4abpq}} \, .
\end{equation}
We can also use Section~\ref{subsec:constcongrol} to calculate  ${\mathcal A}$:
the value of ${\mathcal A}$ can be also found as corresponding to the point of intersection of the lines $y=\frac {qx}p$ and $y=g(x)$, where $g(x) = (2-x)/(1-0.5(1-ab)x)$ is defined in Section~\ref{subsec:locstPBCpar} in the Appendix.
As expected, we get the same ${\mathcal A}$ as the minimum control intensity for local stability.
However, using the maximum norm and applying \eqref{eq:choice_q_1}, we get a bound 
\begin{equation}
\label{def:mathcalB}
{\mathcal B} :=    1- \min \left\{  \frac{2}{p(1+a)}, \frac{2}{q(1+b)}  \right\} \, .
\end{equation}
Note that in the  case $a=b$ we get $\mathcal A<\mathcal B$,  but if $a\neq b$, the relation  $\mathcal A>\mathcal B$ is possible (for more details see Section~\ref{subsec:mathcalAB} in the Appendix).

In the stochastic case, the general method of local stabilization introduced in \cite{pbclocsys}  is quite convenient, since it simplifies the calculations for the parameters of the stochastic control, which can improve the average values of one of the parameters, let it be $\alpha$.  For the stochastically perturbed parameter $\alpha+\ell \xi$, we need to have 
 \cite[Section 4.2, formula (4.3)]{pbclocsys}
\begin{equation}
\label{expect_old}
\mathbb E\ln |\alpha+\ell \xi+(1-\alpha-\ell \xi)\lambda_{\min} |<0 ~~ \Leftrightarrow ~~\mathbb E\ln |\alpha+(1-\alpha)\lambda_{\min}+\ell \xi(1-\lambda_{\min})|<0.
\end{equation}
However application of the method from \cite{pbclocsys}, when optimal controls are different,  might be not suitable  for global stability. 
Such situation is illustrated in Example \ref{equal_controls}, see Section~\ref{sec:ex}.


\section{Examples and Numerical Simulations}
\label{sec:ex}

Everywhere in examples and numerical simulations of \eqref{eq:RickPBCstoch} we assume that the sequences of random variables $(\xi_n)$,  $(\chi_n)$ 
satisfy  Assumption~\ref{as:noise1}, where each $\xi_n, \chi_n$ has a Bernoulli distribution (taking the values of $\pm 1$ with equal probability).

We recall that $r$ and $s$ characterise the reproduction rates of the first and the second species, respectively, while the coefficients $a$ and $b$ describe 
resource sharing: for example, if $a$ is close to one, the second species consumes as much of the resources of the first species as those of its own kind meaning 
that interspecific competition is nearly as strong as its intraspecific counterpart. If, on the other hand, $a$ and $b$ are small, this corresponds to quite weak   interspecific competition  favouring coexistence. The choice of the parameters was motivated by the following factors:
\begin{enumerate}
\item
both $r>2$ and $s>2$;
\item
$K$ is unstable, justifying control application, where the values of control are neither too small nor too large, which would undermine the stabilising effect of noise,
as the amplitude of the noise is less than both the control and its distance from one.
\end{enumerate}
To avoid the effect of non-symmetric and negligible/too strong resource sharing, we set $a=b=0.5$ in all the examples.
The values of $r,s$ were chosen in the segment $[2.5,4]$, leading to an attracting cycle. 
As we further illustrate, stabilising effect of noise is observed in all the examples.

\begin{example}
\label{ex:р3с2.5}
Let $r=3$, $s=2.5$, $a=b=0.5$, then $p=7/3$, $q=4/3$, and the positive equilibrium is $K=(p,q)$, see \eqref{def:pqK}. 
We start with the analysis of control intensities for local and global stabilization with deterministic control.

\bigskip

{\bf (1)  Deterministic control.} 
Let us compute   $\alpha_* \approx 0.142$ from  \eqref{cond:abcontr2} and fix $\alpha = 0.36 >0.142$.  Applying \eqref{cond:abcontr2}  again, we find $\beta_*=0.13636$. 
Then,  by Lemma~\ref{lem:locstPBCpar},  for any $\alpha \geq  0.36$ and  $\beta > 0.13636$, system \eqref{eq:RickPBC} has the locally stable equilibrium  $K$.
Compared to these values, using Sections \ref{sec:invar} and \ref{subsec:Lyapunov}, we get
\[
\tilde \alpha=0.44, \quad \rho_1=0.46, \quad \tilde \beta<0, \quad \rho_2=0.2249,
\]
so for $\alpha>0.46$ and $\beta>0.2249$ the equilibrium is globally stable. 

\bigskip

Next, let us proceed to stabilization by noise.
Here we choose smaller values of $\alpha,\beta$ where PBC control without noise leads to a stable two-cycle, not stabilization: a couple of $\alpha=0.34$, $\beta=0.1$ (see Fig.~\ref{figure_2}, left) and  $\alpha=0.3$, $\beta=0.15$ in  Fig.~\ref{figure_3}, left. For equal control values, we choose $\alpha=\beta=0.25$, also leading to a two-cycle,
see Fig.~\ref{figure_equal}, left.

{\bf (2) Local stabilization by noise}

We apply results for local stabilization, but later we will illustrate that numerical simulations demonstrate low dependency on initial conditions, leading us to the famous (and not yet proven even without noise) hypothesis on the equivalence of local and global stability conditions.
The attempt to identify discrete nonlinear maps (such as Ricker and logistic) for which local stability implies global stability (usually of the unique positive equilibrium) was 
partially successful with the classical paper of Singer \cite{Singer}. A notion of the Schwarzian derivative was introduced, and it was verified that for unimodal maps with a negative Schwarzian derivative local stability implies global attractivity. The scalar Ricker and some other population dynamics models also enjoy this property 
\cite{Liz2007}. However, as mentioned above, we are not aware of any significant progress for two-dimentional maps.
See Fig.~\ref{figure_local_global} for illustration of stabilization results.

{\bf (a)  Spectral norm.} 

We start with theoretical evaluation of the Jacobians with the spectral norm.
To stabilize  the point $(p,q)=(7/3,4/3)$, we choose a slightly higher control of the $x$-variable $\alpha=0.34$
than of the $y$-variable  one  $\beta =0.15$.   As
\[
J_{\alpha,\beta}  =  \left(\begin{array}{ll}
1-p_\alpha  &  -ap_\alpha \\ -bq_\beta& 1-q_\beta  \end{array}\right) \approx  \left(\begin{array}{ll} -0.54 & - 0.77 \\ -0.566667 & -0.133333 \end{array}\right),
\]
the Jacobian has  a smaller eigenvalue  $\lambda_1 \approx -1.02781<-1$, under such deterministic control, the equilibrium is unstable (there is a stable two-cycle).
We perturb  the controls $(\alpha, \beta)$ by  the noises $\xi_n, \chi_n$ with corresponding intensities 
$\ell=0.2$,  $\bar \ell=0.1$:
\[
\alpha_n=\alpha+\ell\xi_{n}=0.34+0.2 \xi_n \quad \beta_n=\beta+\bar\ell\chi_{n}=0.15+0.1 \chi_n.
\]
For the values of noise equal to $\pm 1$ and $ \pm 1$, respectively, for each of the controls, we get the Jacobians such that the largest eigenvalues of $J^T J$ are   
$2.43701$, $0.314298$, $2.15317$  and $0.592392$, respectively, thus the product of the norms (the square roots of the eigenvalues above) is 
$$
\|   J_{0.14,0.05}  \|_2    \| J_{0.54,0.25}  \|_2  \| J_{0.14,0.25}   \|_2    \| J_{0.54,0.05}   \|_2  \approx 0.988424 < 1.
$$ 
Thus,  
\[
\mathbb E\ln \left[\|  J_{\alpha_n,\beta_n}  \|_2 \right]<0,
\]
and we get stabilization.

In numerical simulations in Fig.~\ref{figure_2},  we see a noisy cycle for $\ell=0.08,\bar \ell=0.02$, a noisy equilibrium for $\ell=0.12,\bar \ell=0.05$,
and  stabilization for $\ell=0.12,\bar \ell=0.08$, though direct computation with the spectral norm for these levels of noise does not justify stabilization.
In the lower row of Fig.~\ref{figure_2}, solutions trajectories are illustrated in the $xy$-plane.

{\bf (b)  Maximum-norm}

In explicit stabilization results, we use the convenient maximum-norm, so let us evaluate the required noise for this norm.
We aim to stabilize by noise the point $K=(7/3,4/3)$ when $\alpha=0.36$, $\beta =0.132$. 
As
$\displaystyle J_{\alpha,\beta} $ 
has a smaller eigenvalue $\lambda \approx -1.0038$, with such deterministic control the equilibrium is unstable.

Let us perturb  the controls $(\alpha, \beta)$ by  the noise $\xi_n$ with corresponding intensities 
$\ell=0.211, \,  \bar \ell=0.118$:
~~$\displaystyle 
\alpha_n=\alpha+\ell\xi_{n}=0.36+0.211 \xi_n,  \quad \beta_n=\beta+\bar\ell\xi_{n}=0.132+0.118\xi_n$. 
For simplicity of computation, we assume the same Bernoulli noise in $\alpha$ and $\beta$ perturbations.
This still does not violate Assumption  1: the sequences $\xi, \chi$ can coincide, as long as their members are independent.
For the two Jacobian matrices corresponding to two different values of noises $\xi_{n}= \pm 1$ for both controls, the maximum norms of the Jacobian matrices are 
 $ \approx \max\{ 0.5015,0.5\},\max \{1.9785,0.972 \}$.  Hence 
\[
\mathbb E\ln \left[\|  J_{\alpha_n,\beta_n}  \|_{\infty}\right]
\leq \frac{1}{2} \left( \ln \max\{ 0.5015,0.5\} + \ln \max \{1.9785,0.972 \} \right) 
= \frac{1}{2} \ln 0.992218 <0.
\]
Therefore, by  Lemma~\ref{lem:Kolm}, there is  local asymptotic  stability of  $K$ with any prescribed probability, as formulated in Theorem~\ref{thm:locKolm}.

\begin{figure}[ht]
\centering
\vspace{-32mm}

\hspace{-10mm} \includegraphics[height=.265\textheight]{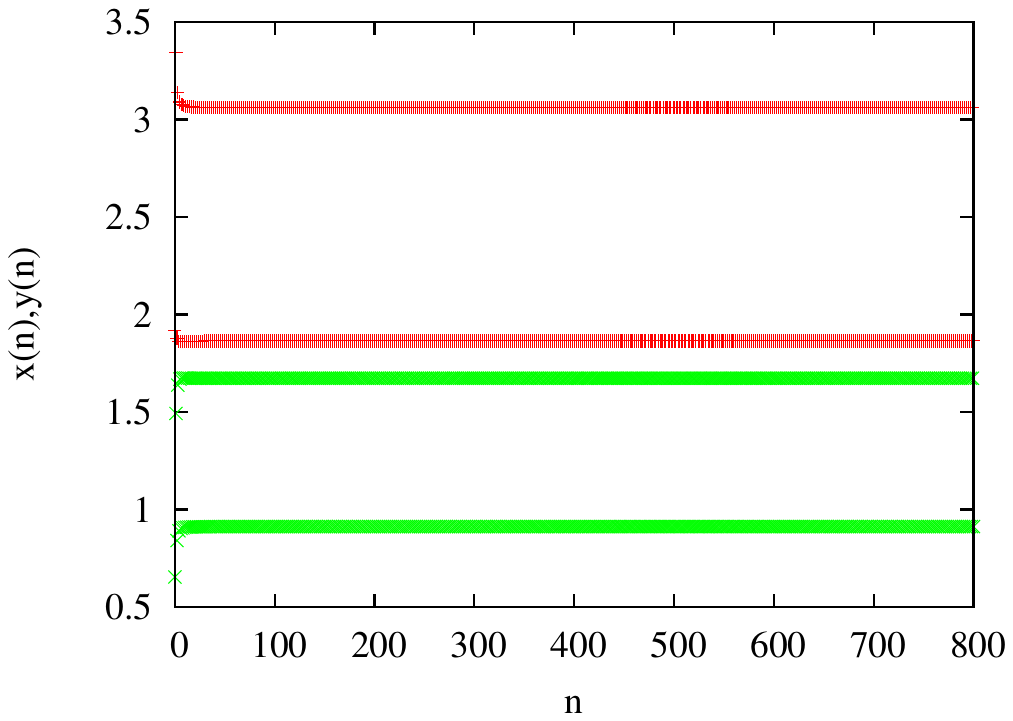}
\hspace{-10mm} \includegraphics[height=.265\textheight]{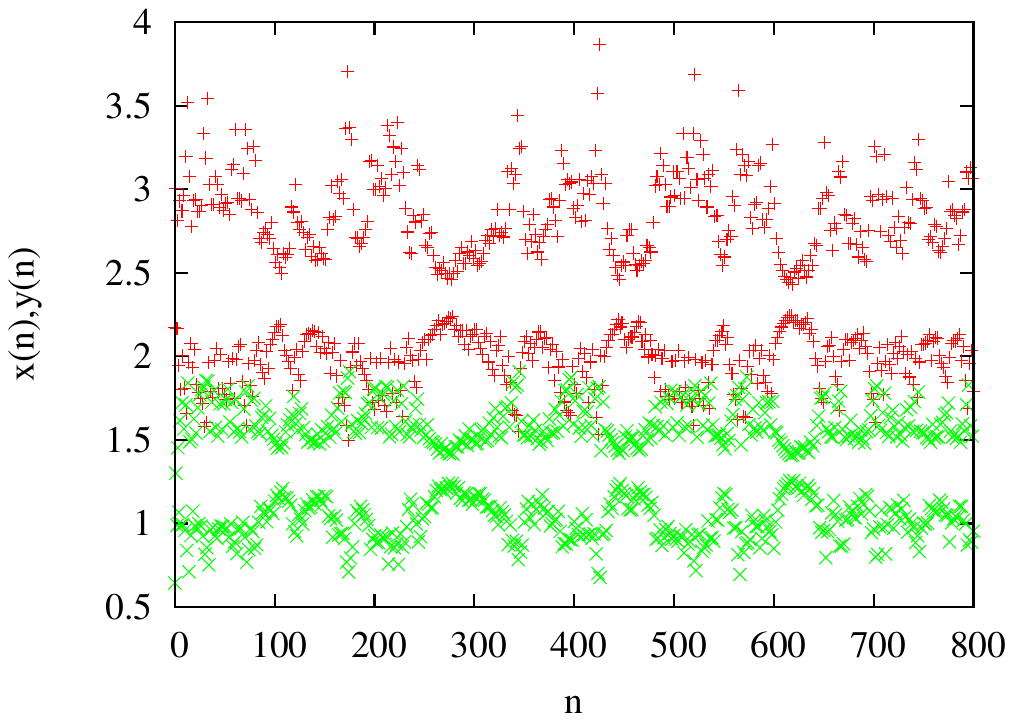}
\hspace{-10mm}\includegraphics[height=.265\textheight]{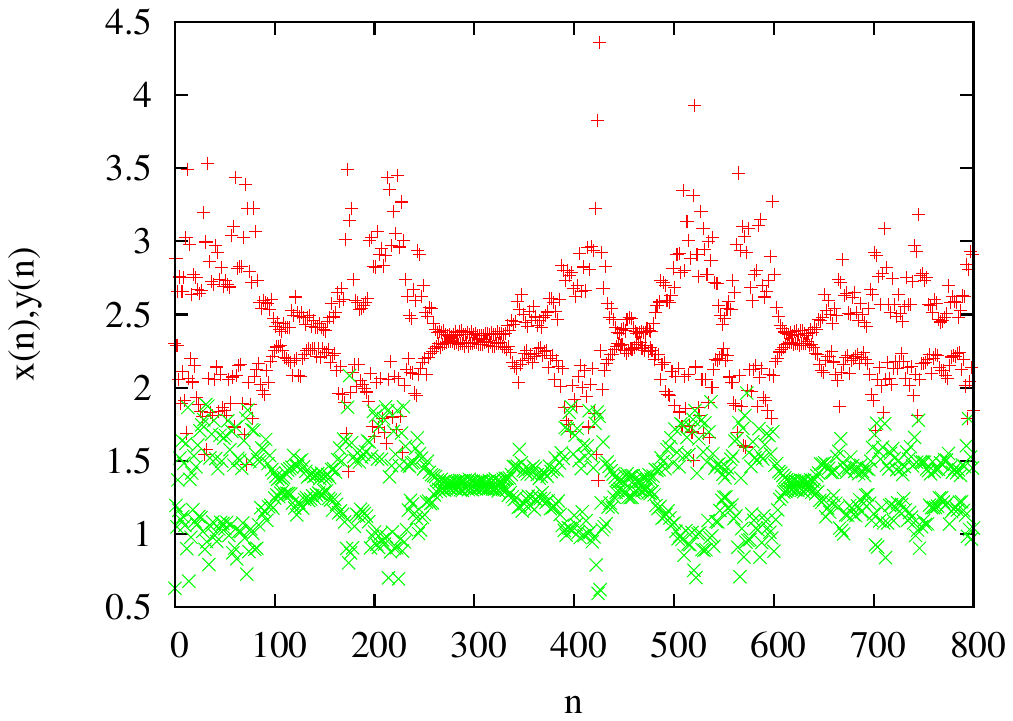} 
\hspace{-10mm}\includegraphics[height=.265\textheight]{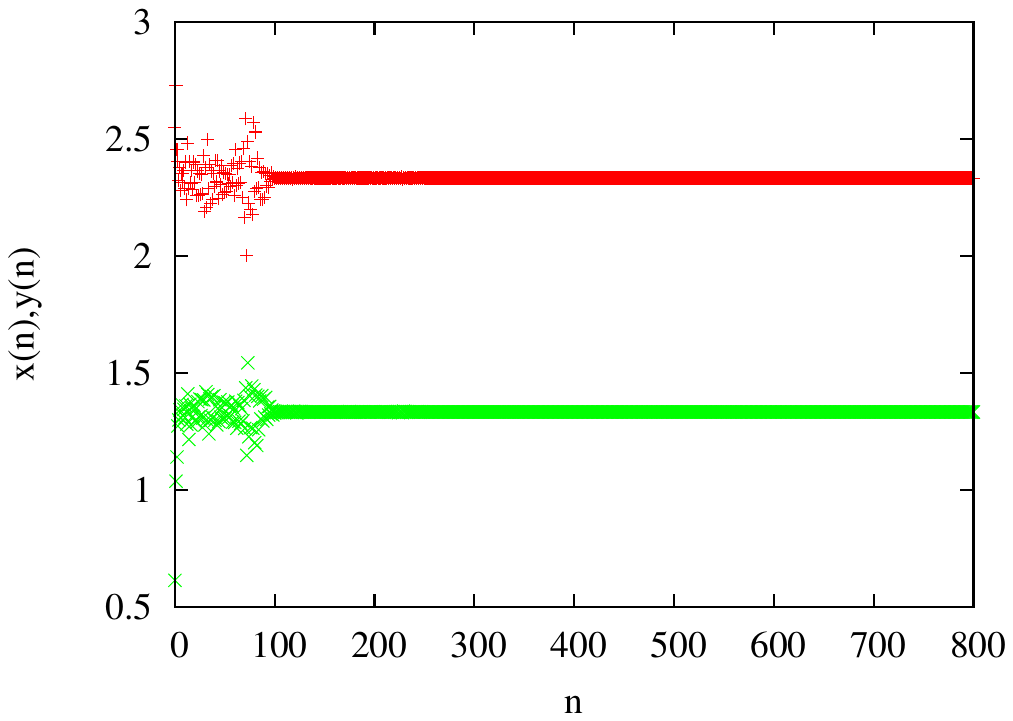} 
\vspace{-32mm}

\hspace{-10mm} \includegraphics[height=.265\textheight]{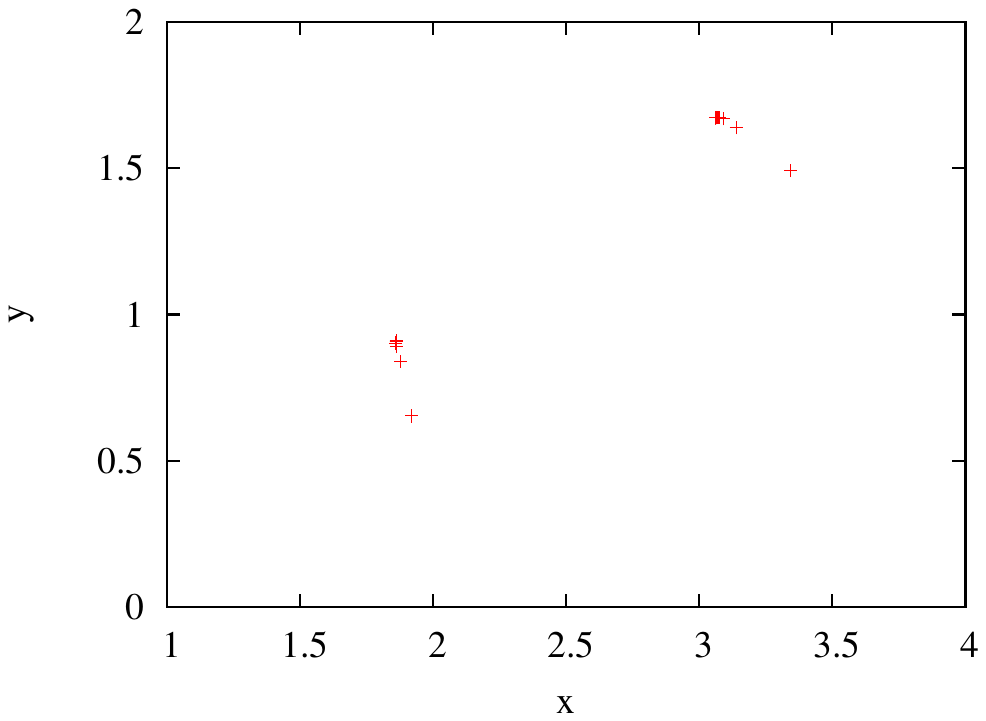}
\hspace{-10mm} \includegraphics[height=.265\textheight]{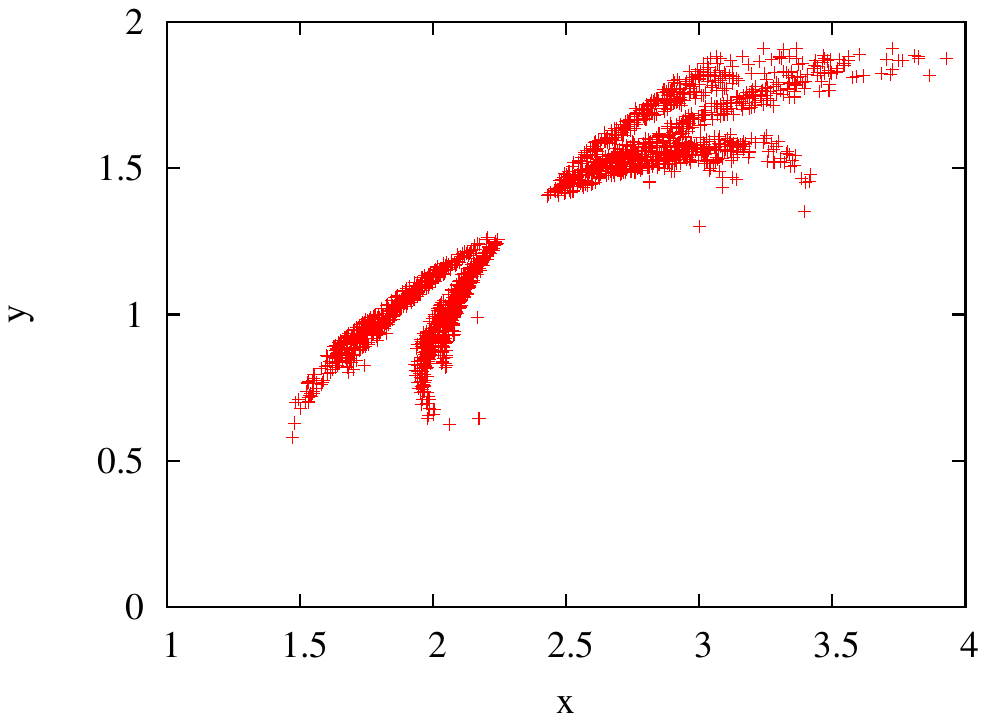}
\hspace{-10mm} \includegraphics[height=.265\textheight]{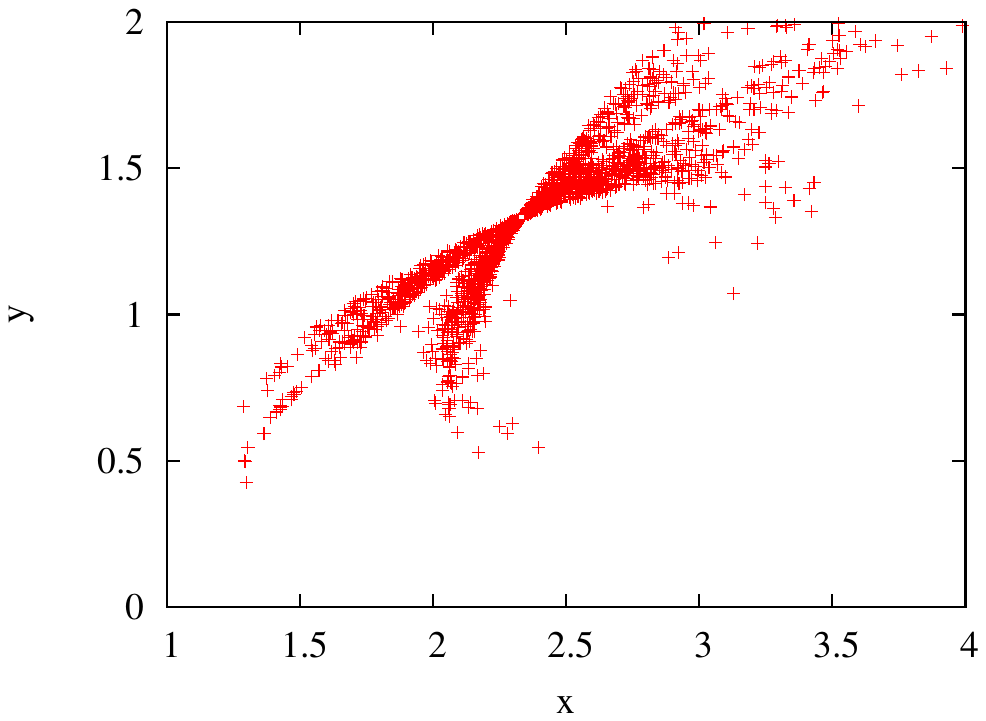}
\hspace{-10mm} \includegraphics[height=.265\textheight]{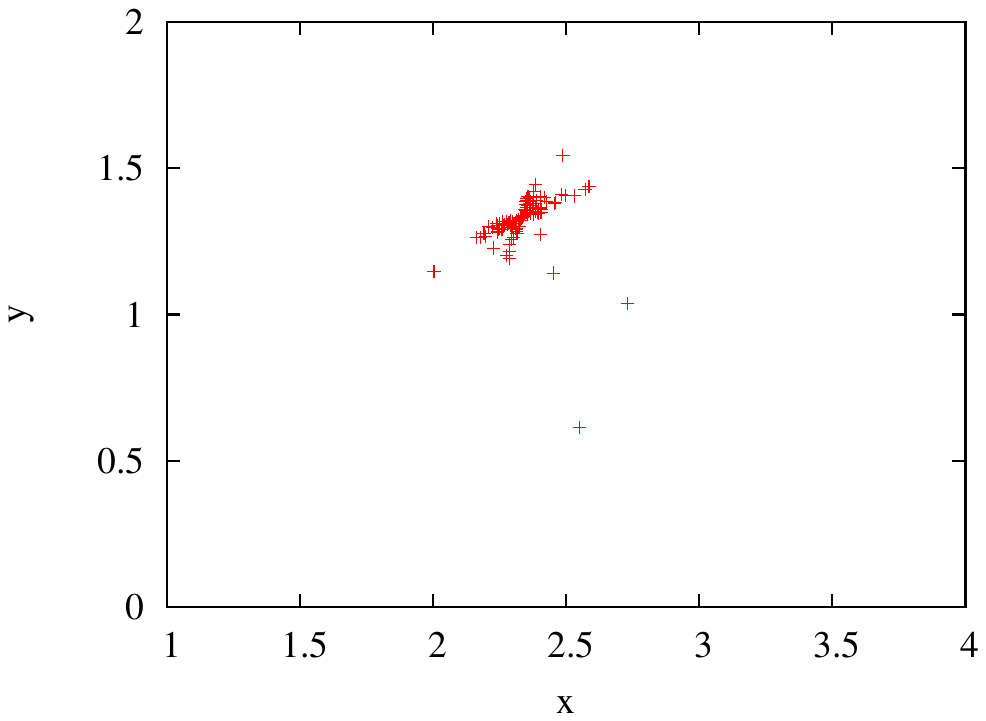}
\vspace{-6mm}

\caption{Runs for the Ricker map ($x$-coordinate in red and $y$ in green) for $r=3$, $s=2.5$, $a=b=0.5$,
$\alpha=0.34$, $\beta =0.1$  and (from left to right)  (1)  no noise;
(2) $\ell = 0.08$, $\bar{\ell} = 0.02$;
(3) $\ell = 0.12$, $\bar{\ell} = 0.05$;
(4) $\ell = 0.2$, $\bar{\ell} = 0.08$. The bottom row shows the corresponding points in $xy$-plane.
We observe (1) a stable two-cycle;
(2) a noisy two cycle; (3) a noisy equilibrium; (4) stabilization of the positive equilibrium $(7/3,4/3)$.
Everywhere the values of $x_0=4$, $y_0=1.1$ were taken to explore convergence.
 }
\label{figure_2}
\end{figure}

Next, for the same example $r=3$, $s=2.5$, $a=b=0.5$, let us take unequal, but closer values of $\alpha=0.3$ and $\beta =0.15$ with the lower level of $x$-control (compared to the couple of $(0.34,0.15)$ and $(0.36,0.132)$  for the computation in the spectral and the maximum norms, respectively) and  numerically evaluate the levels of noise necessary for stabilization.
In Fig.~\ref{figure_3}, we observe (left) a stable two-cycle without noise, (second) a noisy two-cycle for both $x$ and $y$ with smaller noise amplitudes
$\ell = 0.1$, $\bar{\ell} = 0.04$, noisy equilibrium trajectories  (third)  for  $\ell = 0.18$, $\bar{\ell} = 0.07$ and stabilization for 
  $\ell = 0.225$, $\bar{\ell} = 0.1$  (right).  The lower row of Fig.~\ref{figure_3} shows solutions in the $xy$-plane.
Note that in the numerical runs we took $x_0=4$, $y_0=1.1$ not quite close to the equilibrium both in  Fig.~\ref{figure_2} and Fig.~\ref{figure_3},
thus illustrating that in fact we get global stabilization.

\begin{figure}[ht]
\centering
\vspace{-32mm}

\hspace{-10mm} \includegraphics[height=.265\textheight]{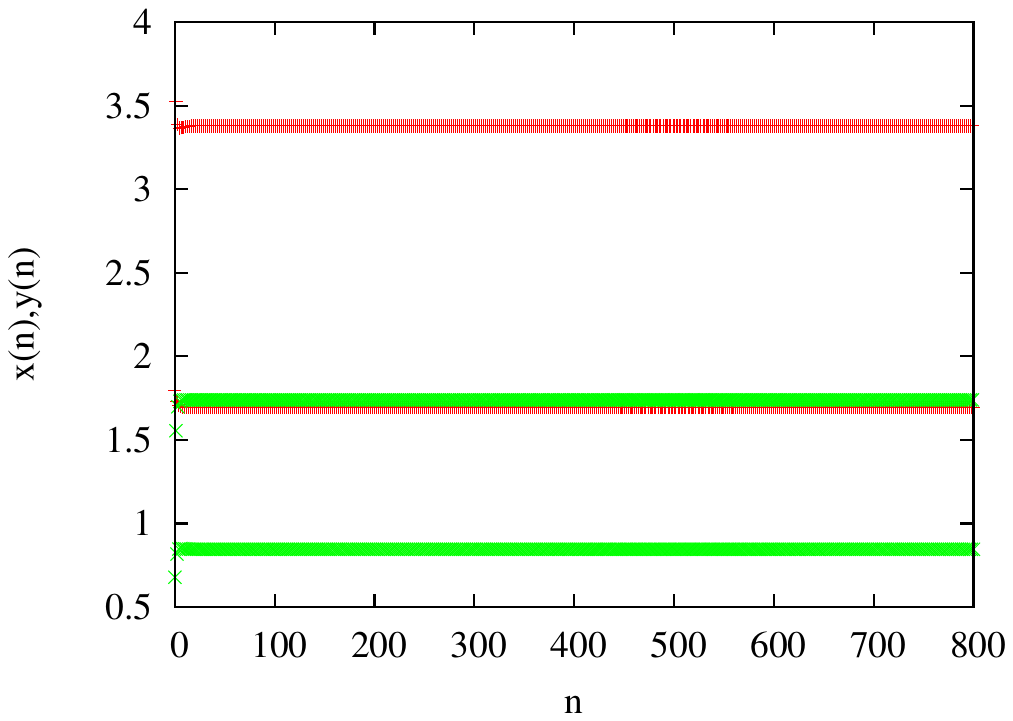}
\hspace{-10mm} \includegraphics[height=.265\textheight]{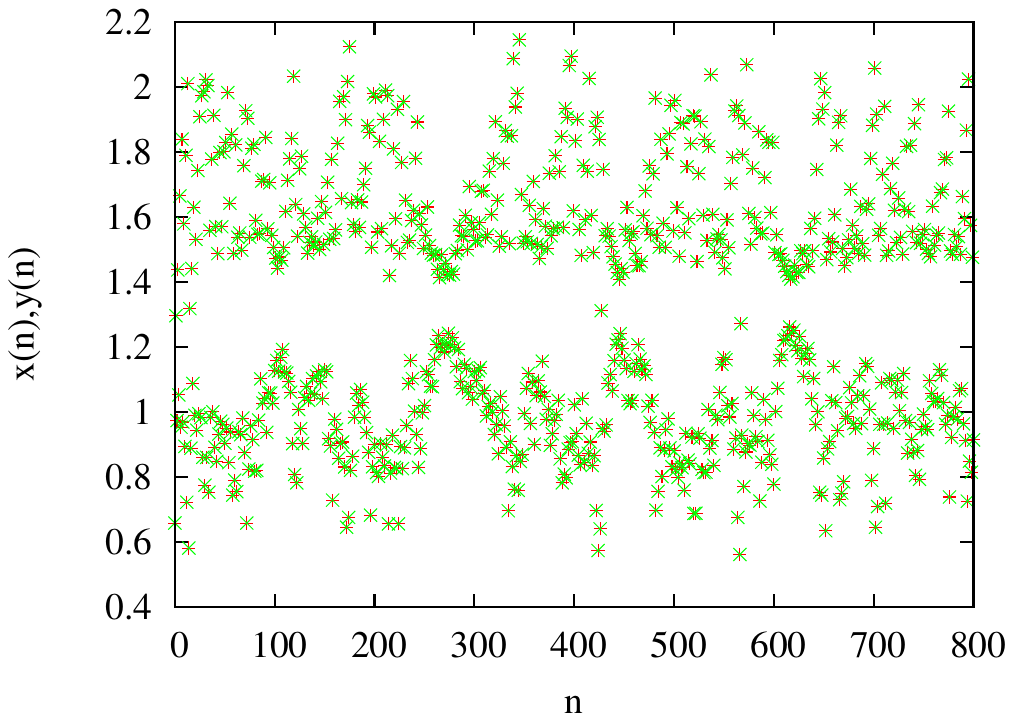}
\hspace{-10mm} \includegraphics[height=.265\textheight]{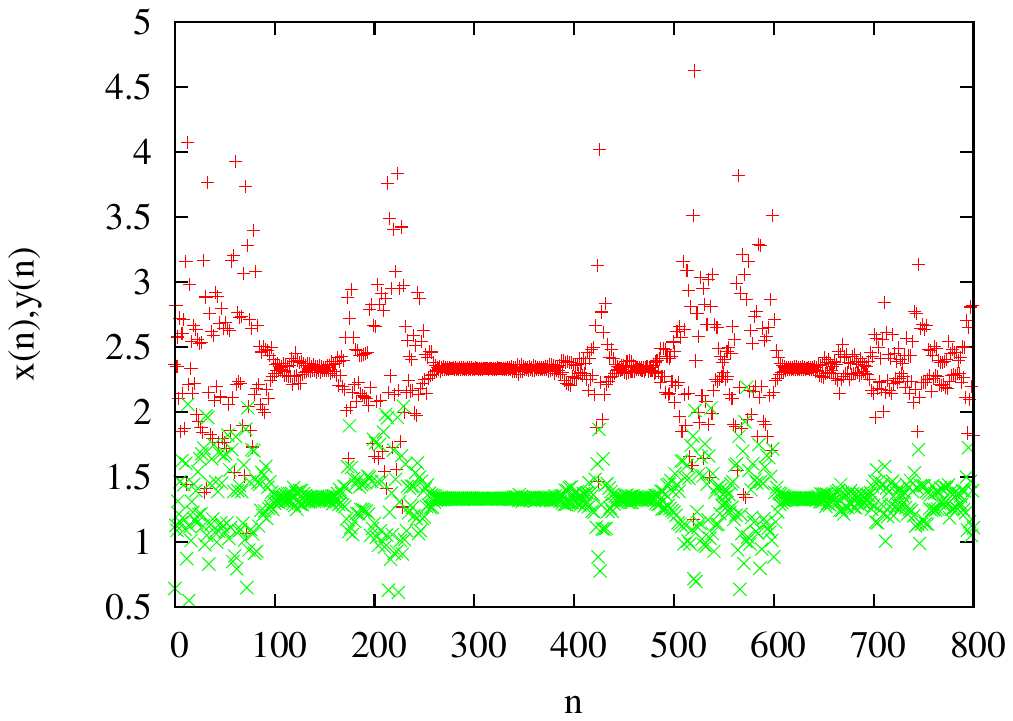}
\hspace{-10mm}  \includegraphics[height=.265\textheight]{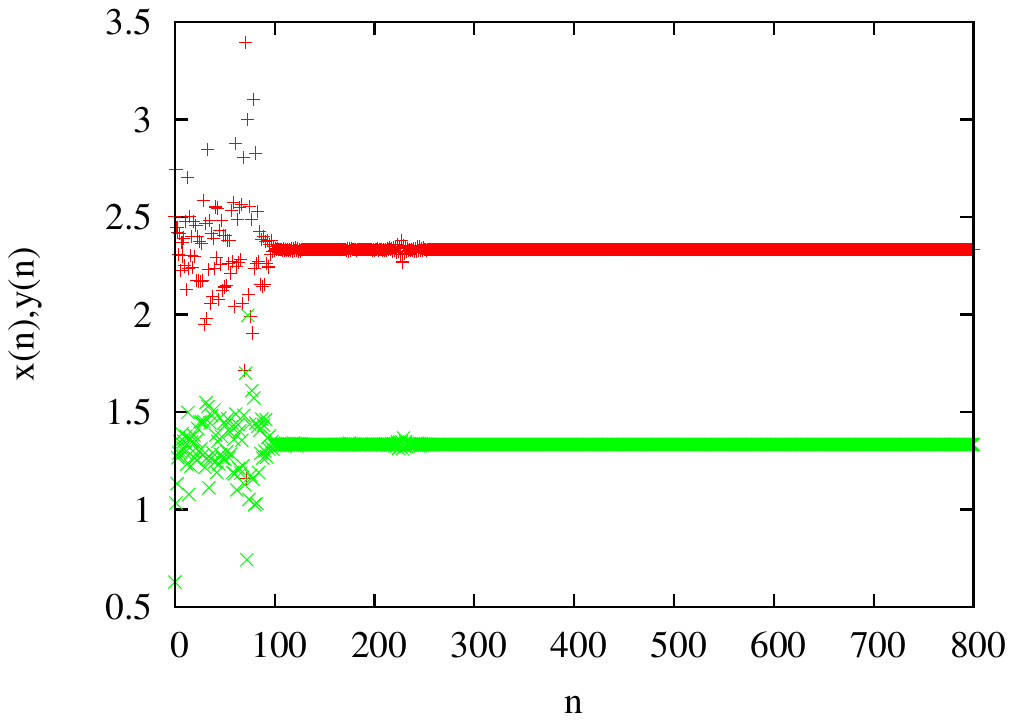} \vspace{-32mm} 

\hspace{-10mm} \includegraphics[height=.265\textheight]{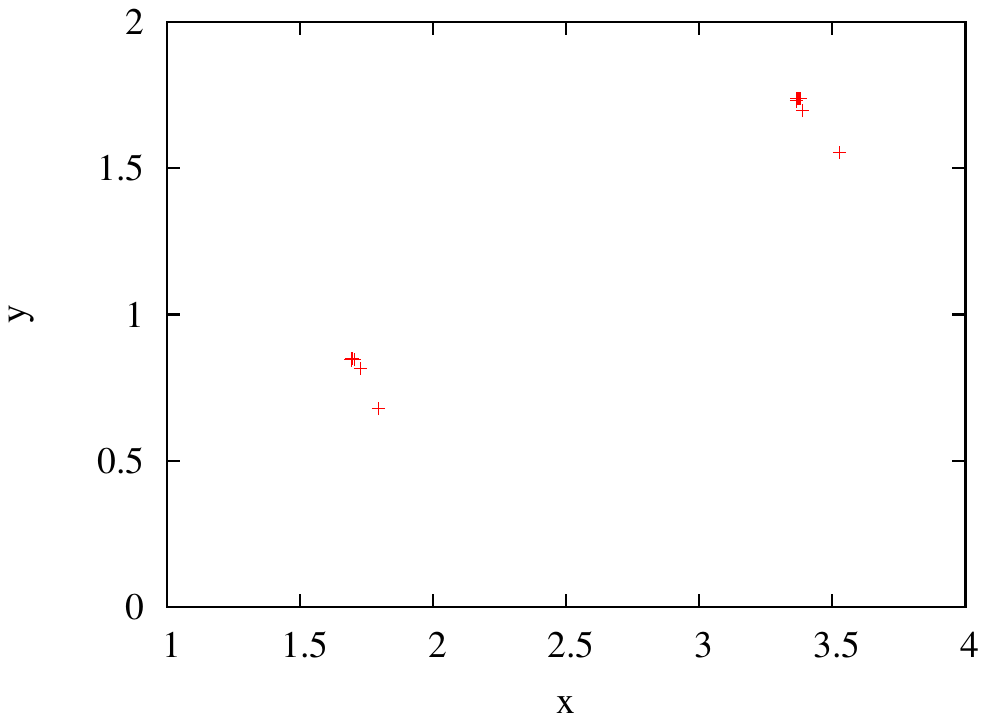}
\hspace{-10mm} \includegraphics[height=.265\textheight]{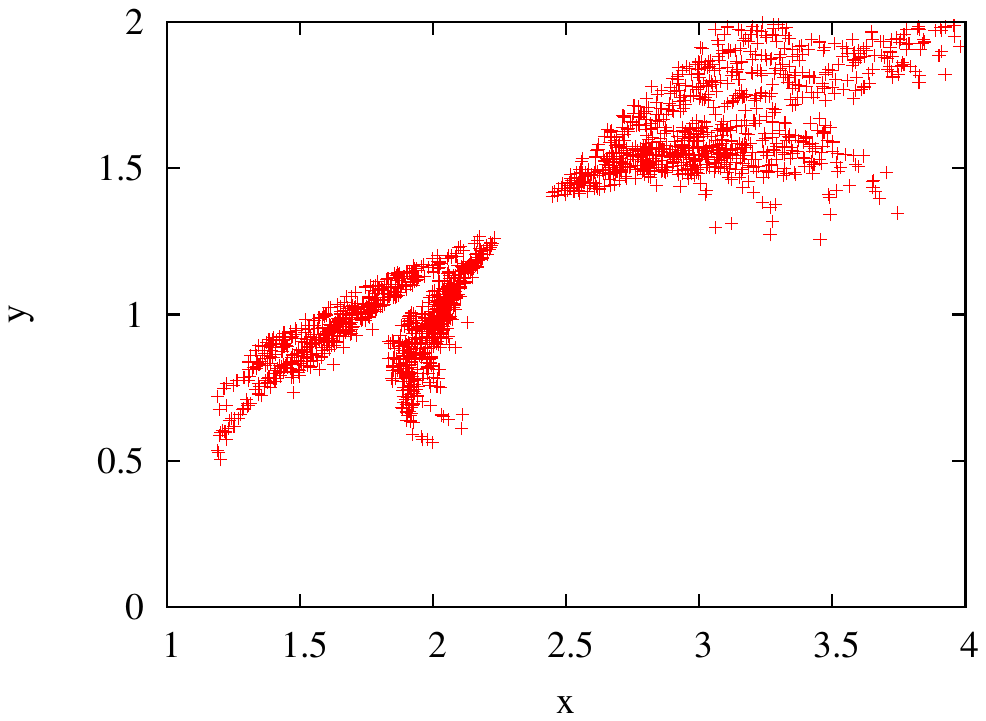}
\hspace{-10mm} \includegraphics[height=.265\textheight]{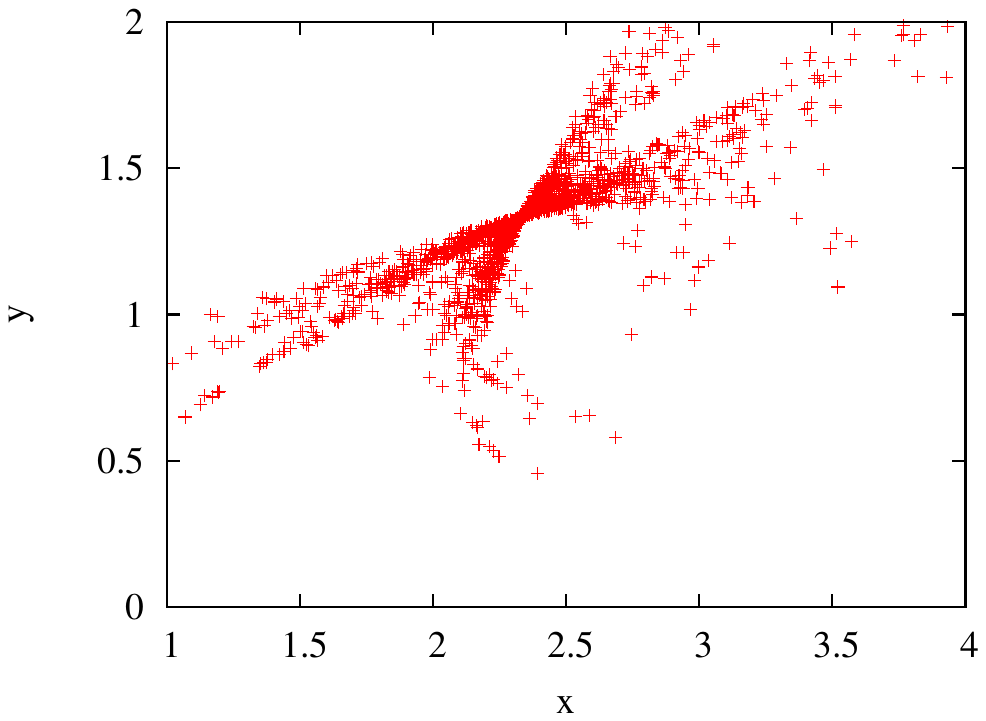}
\hspace{-10mm} \includegraphics[height=.265\textheight]{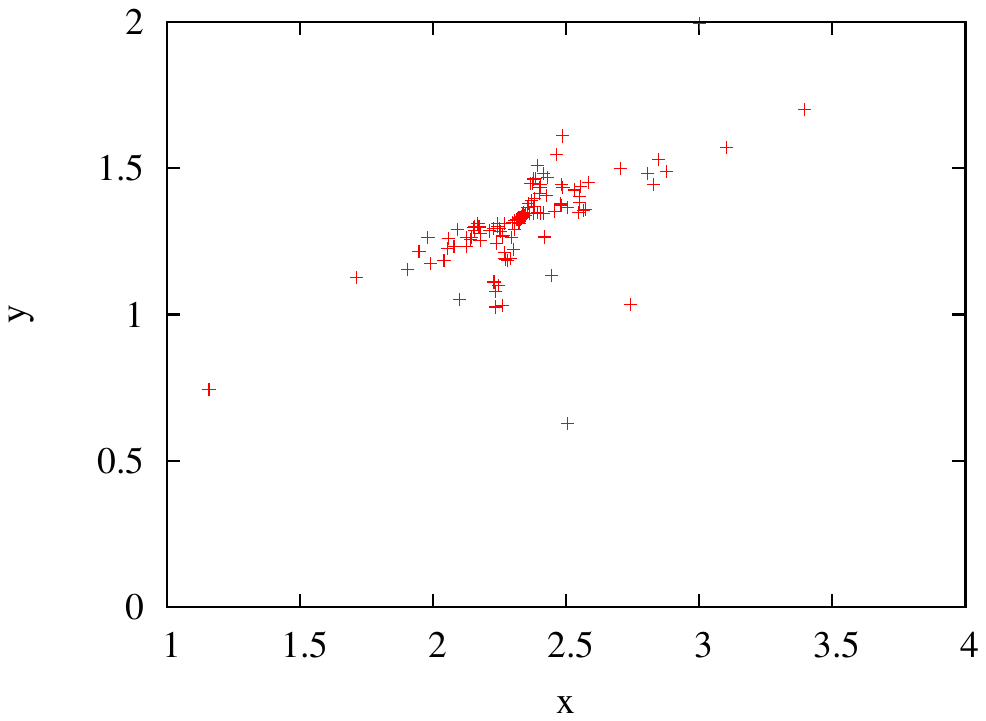}
\vspace{-6mm}

\caption{Runs for the Ricker map ($x$-coordinate in red and $y$ in green) for $r=3$, $s=2.5$, $a=b=0.5$,
$\alpha=0.3$, $\beta =0.15$  and (from left to right)  (1)  no noise;
(2) $\ell = 0.1$, $\bar{\ell} = 0.04$;
(3) $\ell = 0.18$, $\bar{\ell} = 0.07$;
(4) $\ell = 0.225$, $\bar{\ell} = 0.1$. The bottom row shows the corresponding points in $xy$-plane.
We observe (1) a stable two-cycle;
(2) a noisy two cycle; (3) a noisy equilibrium; (4) stabilization of the positive equilibrium $(7/3,4/3)$.
Everywhere the values of $x_0=4$, $y_0=1.1$ were taken.
 }
\label{figure_3}
\end{figure}
\end{example}

\bigskip

Next, consider the case of equal controls.  

\begin{example}
\label{equal_controls}
As previously, we take  $r=3$, $s=2.5$, $a=b=0.5$, so $p = \frac{7}{3}$, $q=\frac{4}{3}$. As it was shown in Section~\ref{sec:ab}  and Section \ref{subsec:mathcalAB}, 
(see also \eqref{def:mathcalA}, \eqref{def:mathcalB}), in this case   $\mathcal A\le \mathcal B$.  Indeed, for the Jacobian
$$
J = \left( \begin{array}{cc} 1-p & -ap \\ -bq & 1-q \end{array} \right) = - \frac{1}{3}  \left( \begin{array}{cc}  4 & 3.5 \\ 2 & 1   \end{array} \right),
$$   
the eigenvalues of $J$ are $-(5 \pm \sqrt{37})/6$,  
and $\lambda_{\min}\approx -1.8471$, the other eigenvalue is in the unit circle. So $\mathcal A = \frac{-1-\lambda_{\min}}{1-\lambda_{\min}}  \approx 0.2975 $, $\mathcal B=1- \frac  2{p(1+a)} \approx 0.4286$.  If we are interested only in local stability, we need to  apply the control $\mathcal A \approx 0.2975 $ only along the first eigenvalue, in other words we can apply the control matrix $\displaystyle \Upsilon_1=\left(\begin{array}{cc}  0.2975 &  0 \\ 0 & 0  \end{array}\right)$. 
The matrix $\Upsilon_1$ is not constant, so after its application, the  system  will no longer be in the PBC form, and  the  Lyapunov function from Section~\ref{subsec:Lyapunov} will not work. So instead of $\Upsilon_1$,  we consider 
 $\displaystyle \Upsilon_2=\left(\begin{array}{cc}  0.2975 &  0 \\ 0 &  0.2975  \end{array}\right)$. 

Now we proceed to the stochastic case. For Bernoulli $\xi$,  taking the values of  $\pm 1$ with probability $1/2$, condition \eqref{expect_old}  takes the form
\begin{align*}
& \mathbb E\ln |\alpha+(1-\alpha)\lambda_{\min}+\ell \xi(1-\lambda_{\min})| \\  = &
\frac 12 \ln |\alpha+(1-\alpha)\lambda_{\min}+\ell (1-\lambda_{\min})| + \frac 12  \ln |\alpha+(1-\alpha)\lambda_{\min} - \ell (1-\lambda_{\min})| 
\\ = & \frac 12 \ln \left| (\alpha+(1-\alpha)\lambda)^2-\ell^2(1-\lambda)^2 \right| \approx
\frac 12 \ln | (\alpha+(1-\alpha)(-1.847))^2-  8.1061\ell^2  |<0.
\end{align*}
This leads to the relation of $\alpha$ and $\ell$
$$
| (\alpha+(1-\alpha)(-1.847))^2-  8.1061 \ell^2  |<1.   
$$
Assume that the average control intensity is  $\alpha=0.25 <0.2975$, then $-1< 1.289 - 8.106 \ell^2  <1$ is equivalent to 
$ 0.0357<\ell^2 <0.282$, or to $\ell \in (0.189, 0.531)$. Since for PBC $\ell \geq \alpha$ are not considered, for $\alpha =0.25$,
$\ell \in  (0.189,0.25)$, we get stabilization.
Note that for $\alpha < {\mathcal A} \approx 0.2975$, solutions of the deterministic problem do not converge to $K$. 
In numerical simulations, we take smaller control intensity. For $\alpha =0.25$, without noise, 
solutions converge to a stable two-cycle,  see Fig.~\ref{figure_equal}, left. Introduction of noise $\ell = 0.18$ leads to a ``noisy equilibrium",  see Fig.~\ref{figure_equal},  the middle graphs. Once we introduce $\ell = 0.225$, we get stabilization due to introduction of noise (Fig.~\ref{figure_equal},  the right graphs)

\begin{figure}[ht]
\centering
\vspace{-42mm}

\hspace{-12mm} \includegraphics[height=.35\textheight]{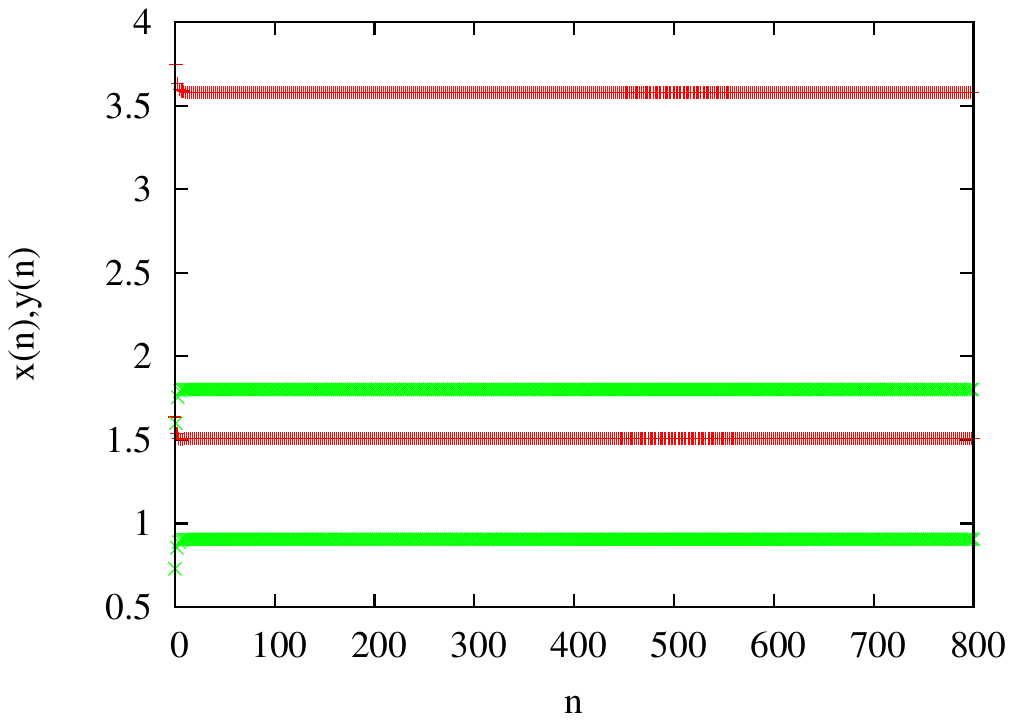}
\hspace{-12mm} \includegraphics[height=.35\textheight]{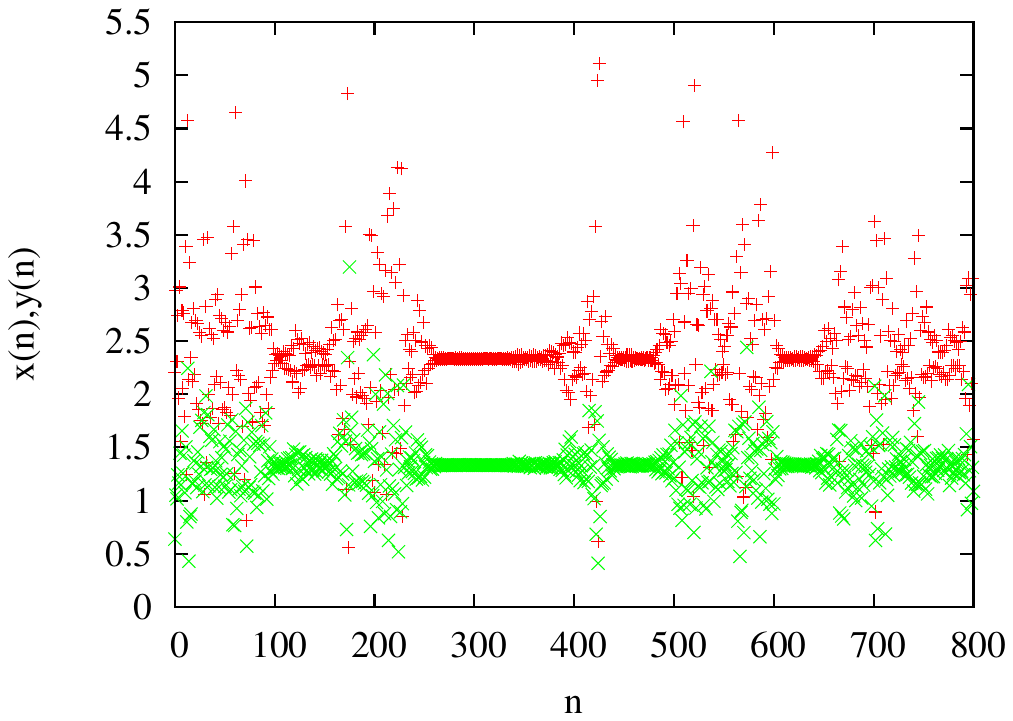}
\hspace{-12mm} \includegraphics[height=.35\textheight]{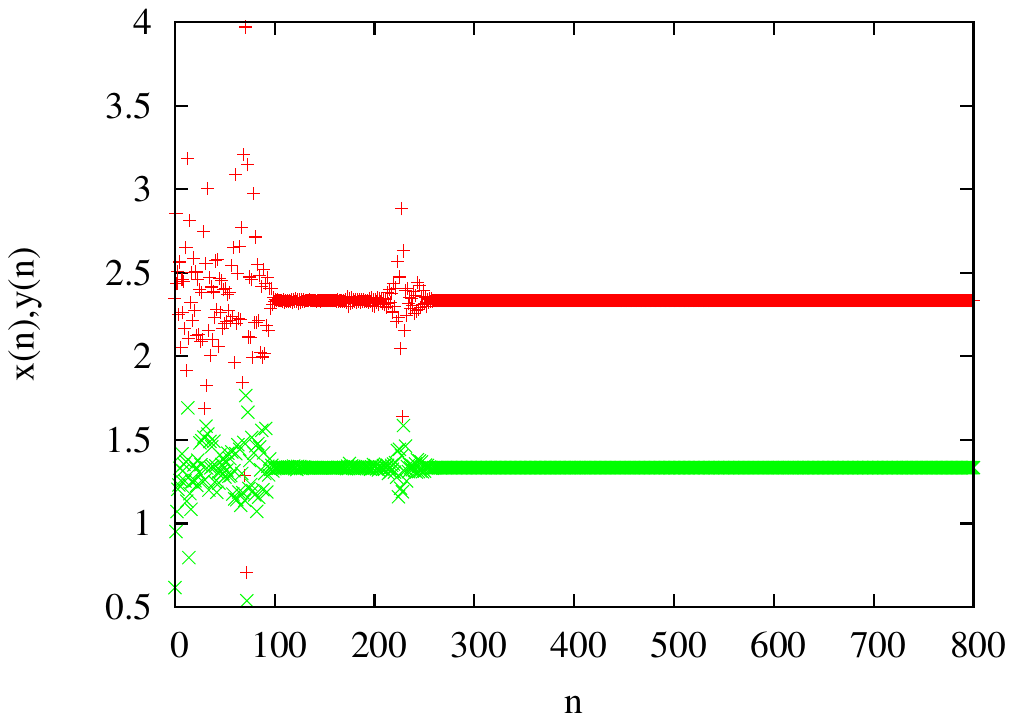}  \vspace{-42mm} 

\hspace{-12mm} \includegraphics[height=.35\textheight]{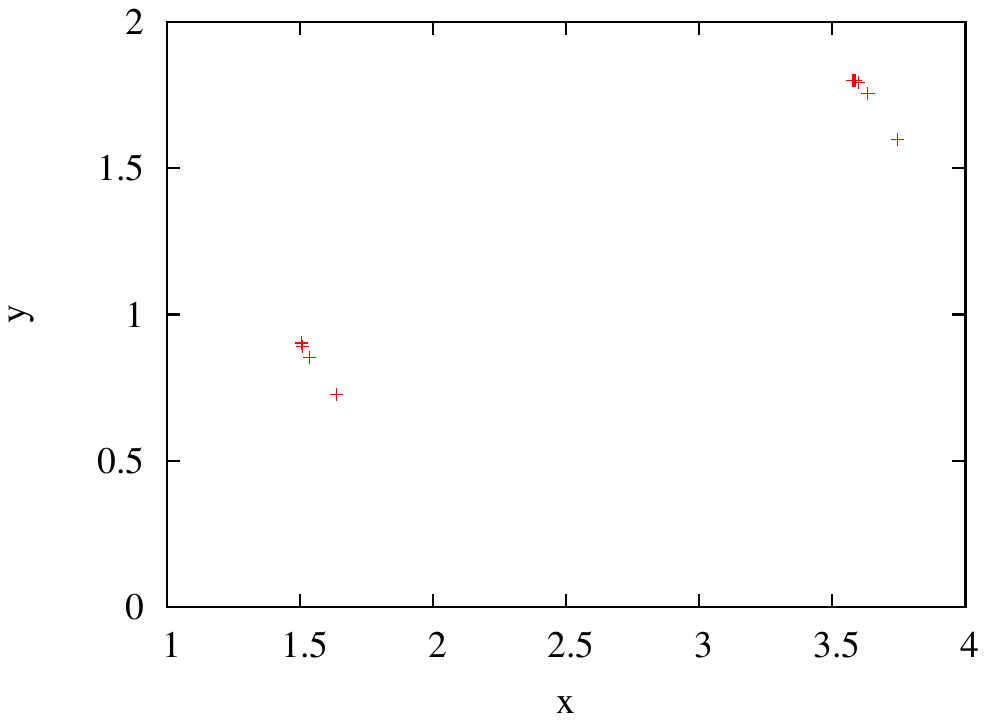}
\hspace{-12mm} \includegraphics[height=.35\textheight]{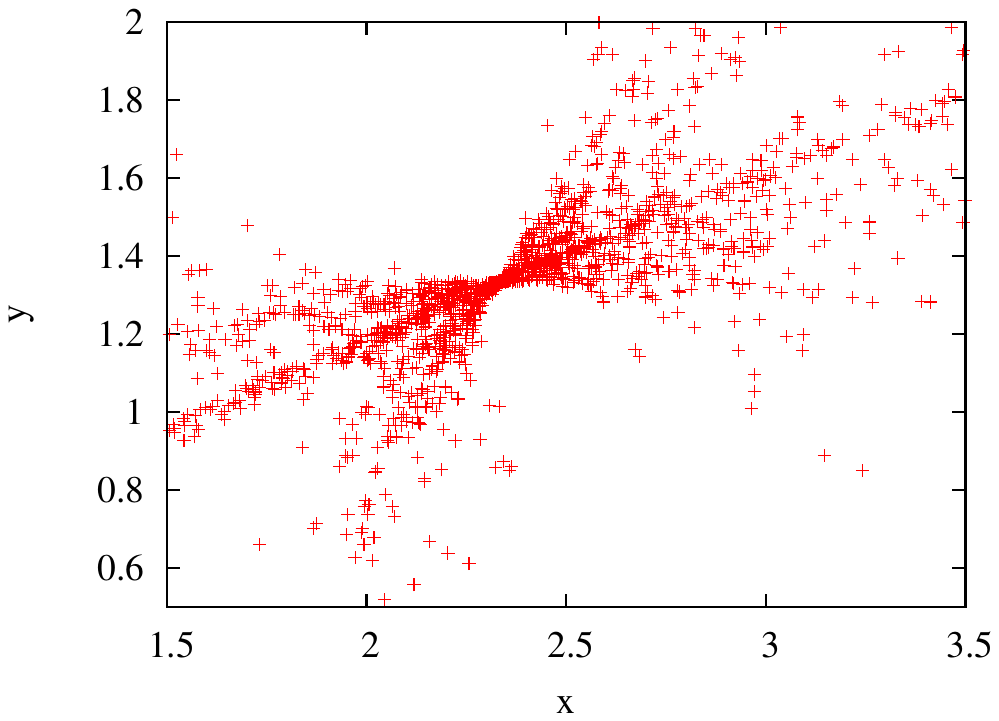}
\hspace{-12mm} \includegraphics[height=.35\textheight]{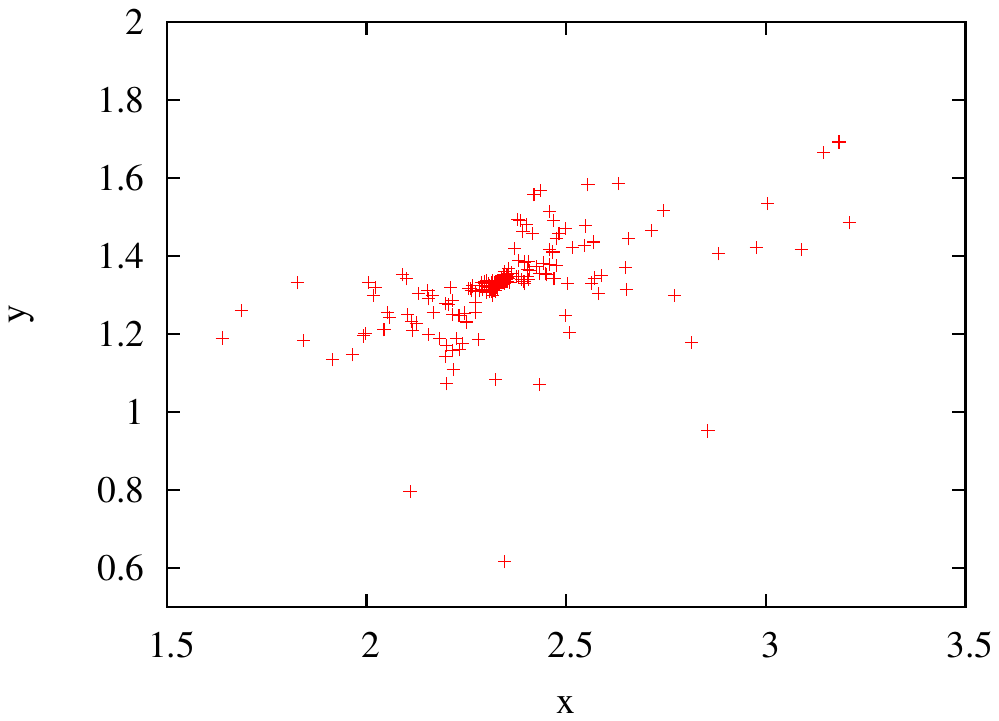}
\vspace{-8mm}

\caption{Runs for the Ricker map ($x$-coordinate in red and $y$ in green) for $r=3$, $s=2.5$, $a=b=0.5$,
$\alpha=\beta =0.25$  and (from left to right)  (1)  no noise;
(2) $\ell = 0.18$;
(3) $\ell = 0.225$. The bottom row shows the points in $xy$-plane.
We observe (1) a stable two-cycle;
(2) a noisy two cycle; (3) stabilization of the positive equilibrium $(7/3,4/3)$, with
 $(x_0,y_0)=(4,1.1)$ everywhere.
 }
\label{figure_equal}
\end{figure}

In addition, let us  take different initial points: first, we take the initial point quite close (within $10^{-6}$ in both coordinates) to the equilibrium. Here we get local stabilization, see Fig.~\ref{figure_local_global}, left and middle. In the left figure, we observe local stabilization: the solution stays in the small neighbourhood of $K$ for the first 500 iterations. However, if we extend history further, the accumulation of the numerical error will lead to temporary destabilization,  Fig.~\ref{figure_local_global}, middle, showing 2000 iterations. If we take the initial point $(4,1.1)$ quite far from $K=(7/3,4/3)$, eventually the solution arrives to the neighbourhood of $K$, see Fig.~\ref{figure_local_global}, right.
Everywhere above, we discussed local stabilization, but  Fig.~\ref{figure_local_global} to some extend illustrates its global character.

\begin{figure}[ht]
\centering
\vspace{-42mm}

\hspace{-12mm} \includegraphics[height=.35\textheight]{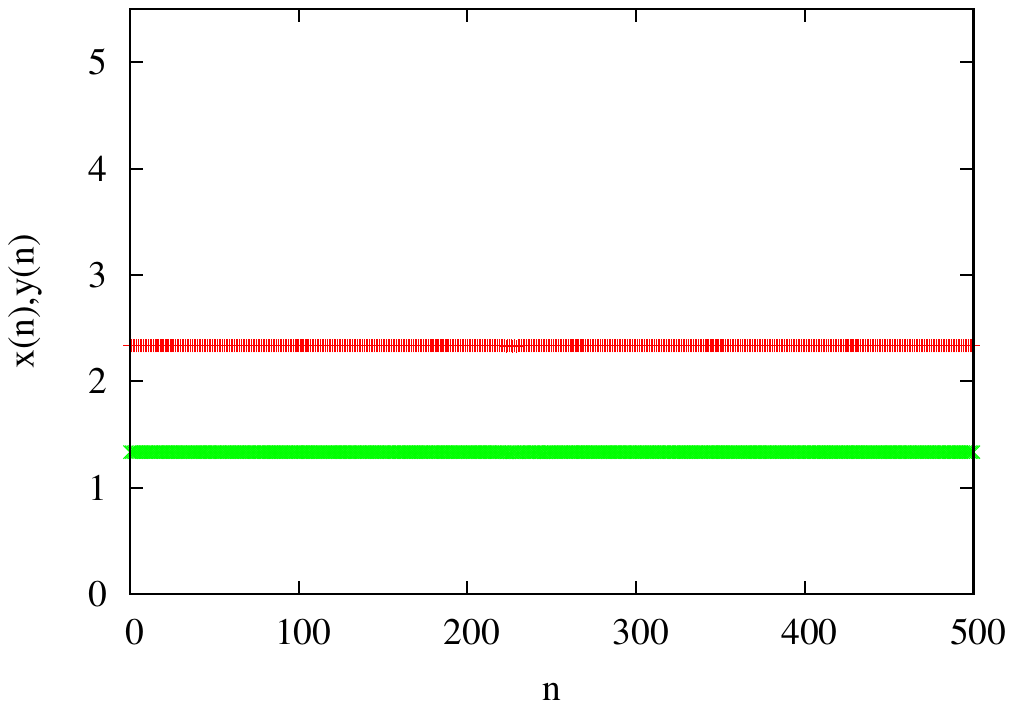}
\hspace{-12mm} \includegraphics[height=.35\textheight]{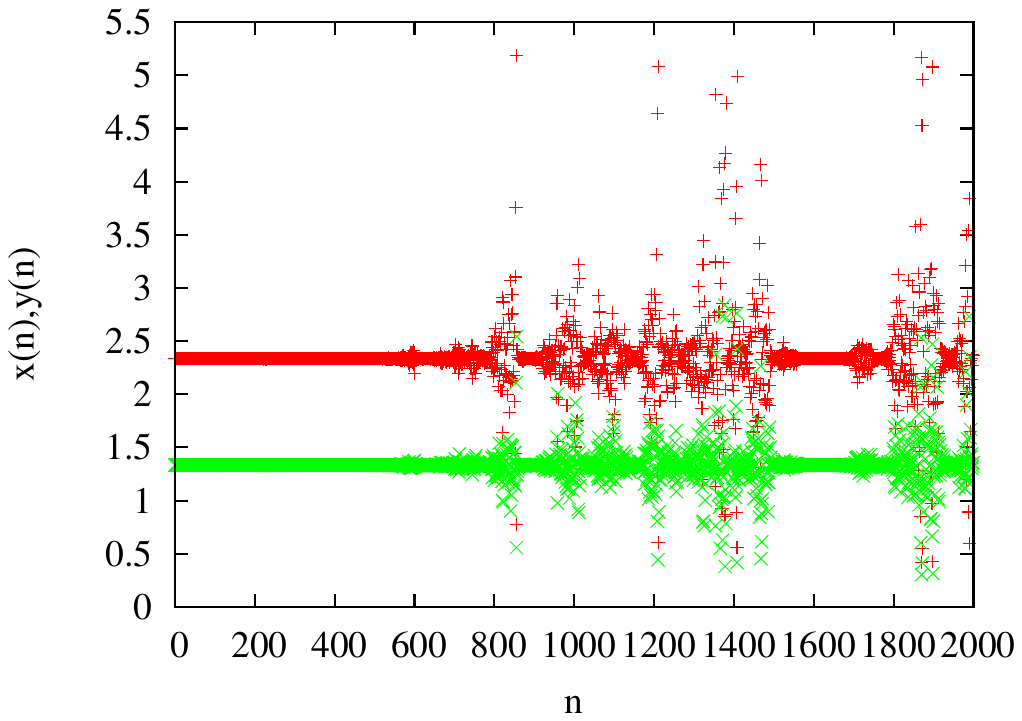}
\hspace{-12mm} \includegraphics[height=.35\textheight]{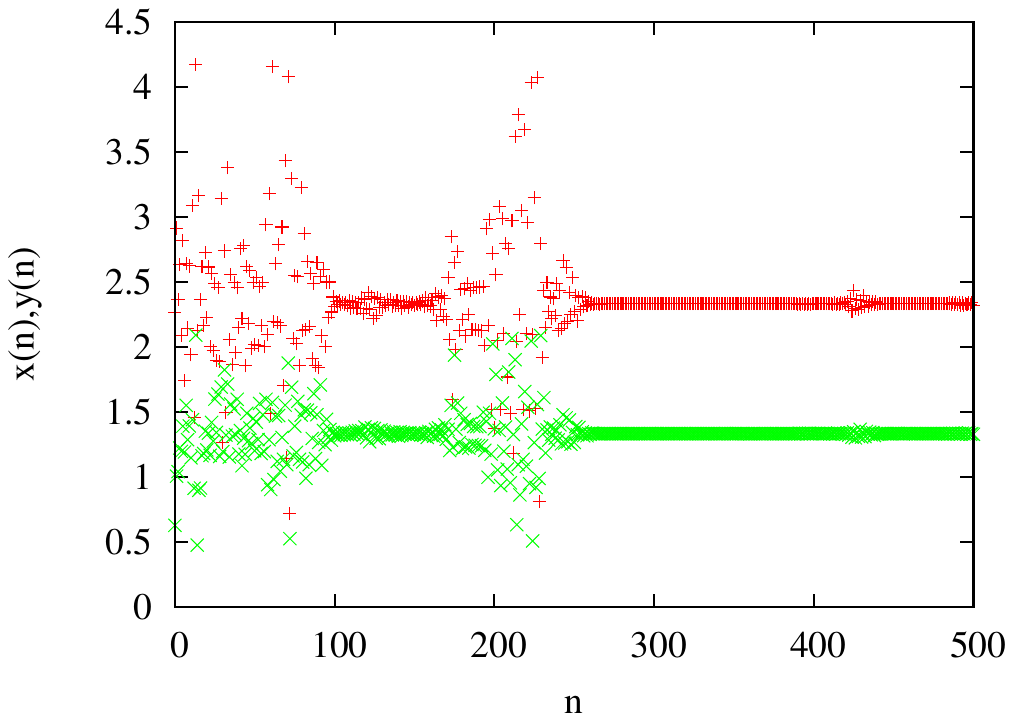}
\vspace{-8mm}

\caption{Runs for the Ricker map ($x$-coordinate in red and $y$ in green) for $r=3$, $s=2.5$, $a=b=0.5$,
$\alpha=\beta =0.25$, $\ell = 0.2$  and (from left to right)   (1)  $x_0=2.333333$, $y_0=1.333333$, 500 iterations;
(2) $x_0=2.333333$, $y_0=1.333333$, 2000 iterations;
(3) $x_0=4$, $y_0=1.1$, 500 iterations. 
 }
\label{figure_local_global}
\end{figure}

\end{example}

In all the examples above, the values of $r,s$ exceeded 2, but were not much bigger (3 and 2.5, respectively). Let us increase the value of $s$ to 4.
In this case, theoretically prescribed values for global stabilization exceed 0.5, however, numerical simulations show stabilization for $\alpha=0.3$, $\beta =0.45$,
once the noise amplitudes are large enough.

\begin{example}
\label{ex:more_unstable}
The stabilizing influence of noise is even more striking if higher control levels of $\alpha,\beta$ are required for bigger $r,s$.
Consider $r=3$, $s=4$, $a=0.5$, $b=0.4$ with $\alpha=0.3$, $\beta =0.45$.
As $s>r$, here stronger $y$-control is expected. 

In Fig.~\ref{figure_4}, we observe (left) a stable two-cycle without noise, (second) a noisy two-cycle for both $x$ and $y$ with smaller noise amplitudes
$\ell = 0.05$, $\bar{\ell} = 0.06$. When the noise values increase (third) to   $\ell = 0.1$, $\bar{\ell} = 0.1$, the trajectories have a noisy equilibrium form, while for 
 $\ell = 0.15$, $\bar{\ell} = 0.2$  (right), the positive equilibrium becomes globally stable. 

\begin{figure}[ht]
\centering
\vspace{-30mm}

\hspace{-10mm} \includegraphics[height=.265\textheight]{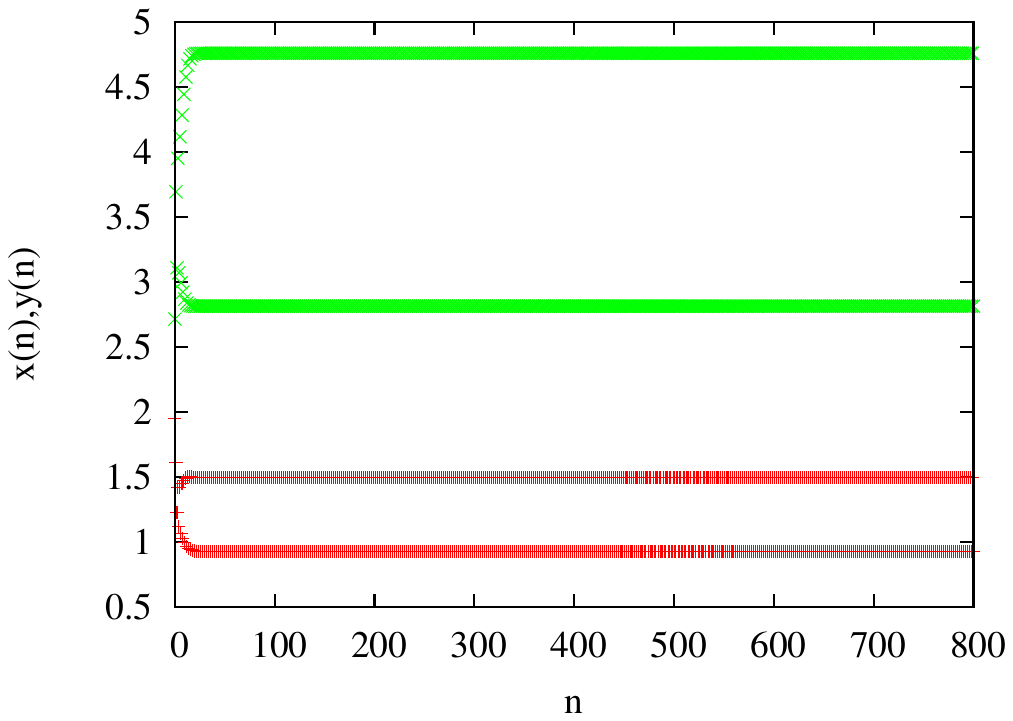}
\hspace{-10mm} \includegraphics[height=.265\textheight]{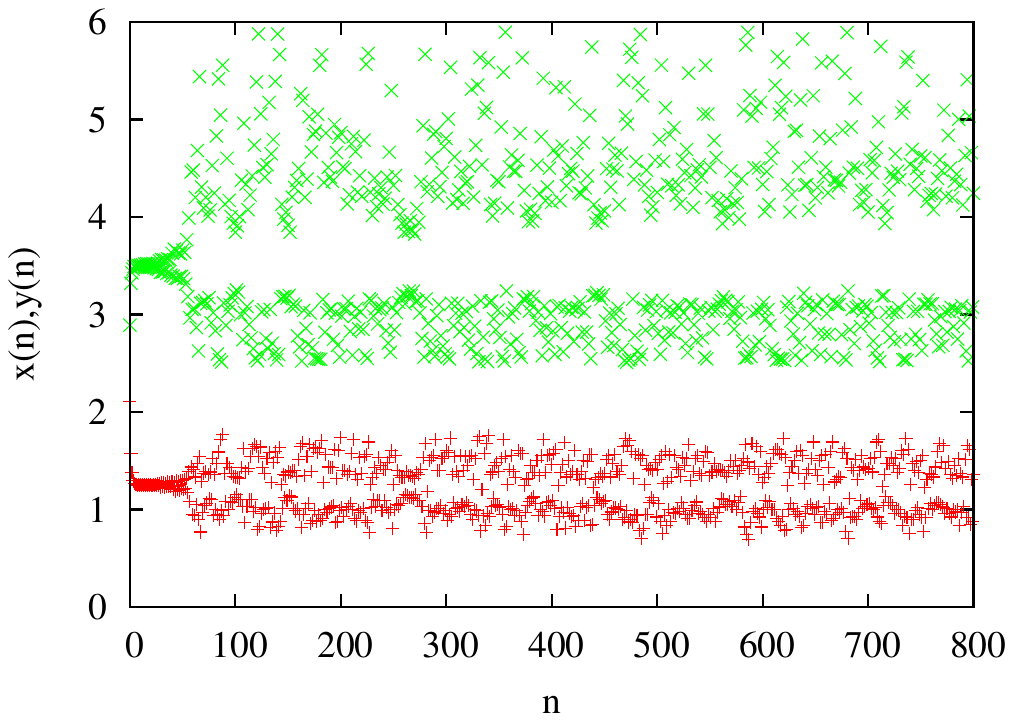}
\hspace{-10mm} \includegraphics[height=.265\textheight]{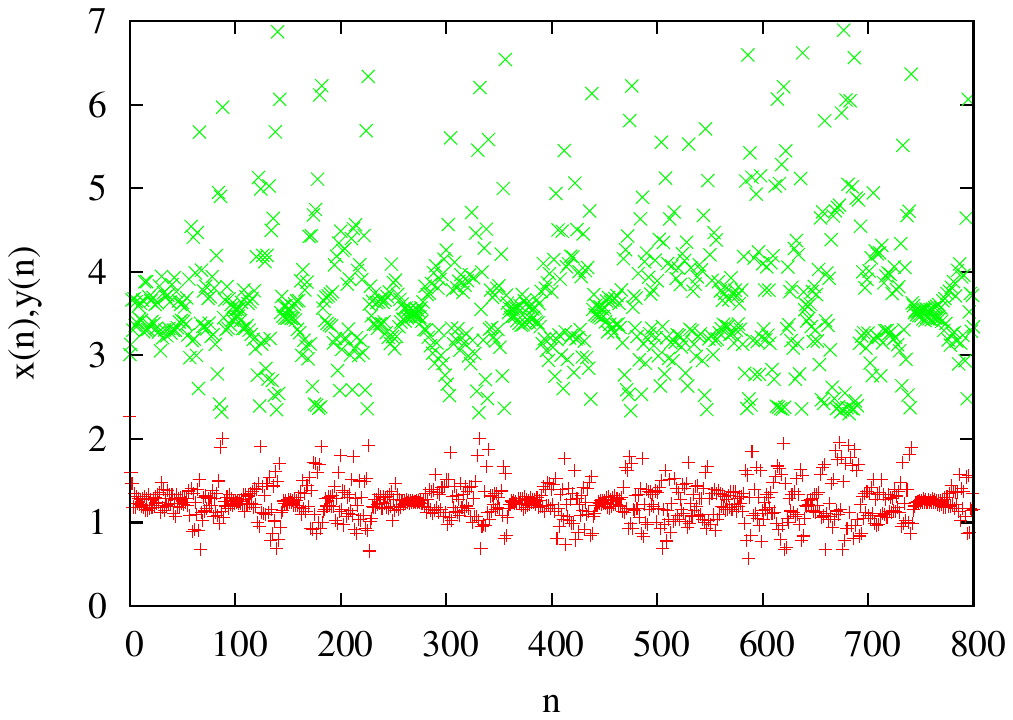}
\hspace{-10mm} \includegraphics[height=.265\textheight]{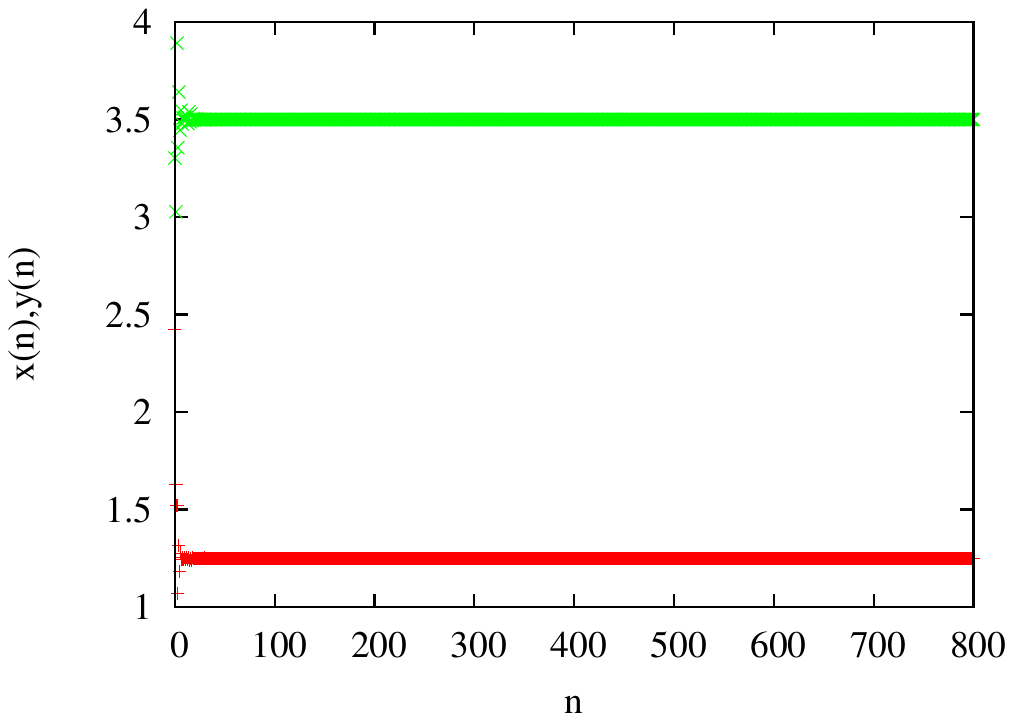} \vspace{-32mm} 

\hspace{-10mm} \includegraphics[height=.265\textheight]{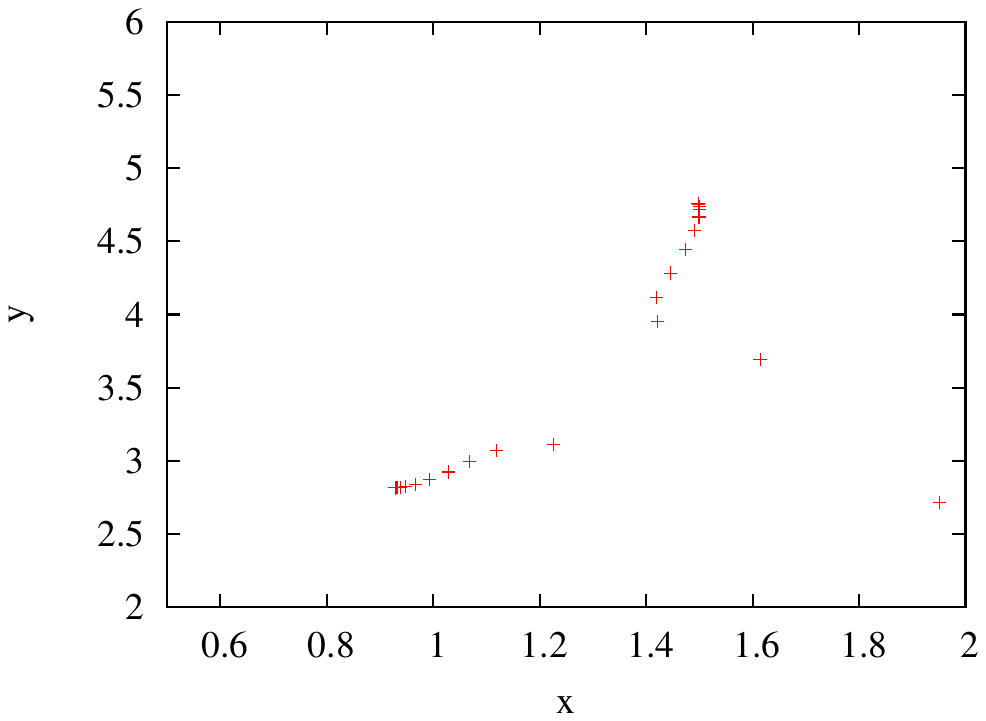}
\hspace{-10mm} \includegraphics[height=.265\textheight]{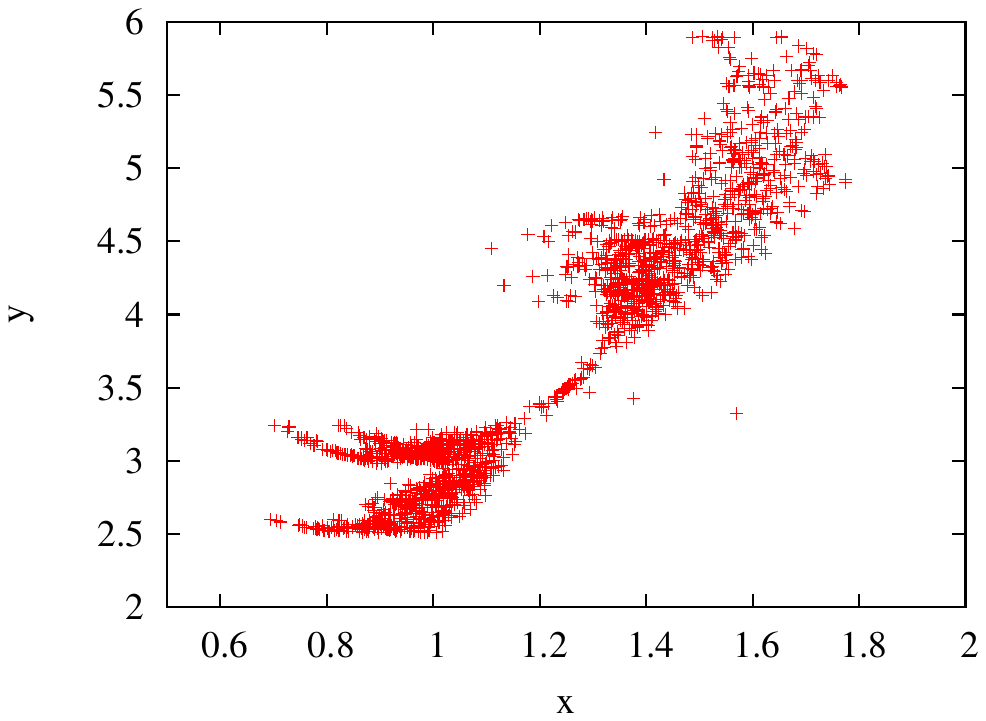}
\hspace{-10mm} \includegraphics[height=.265\textheight]{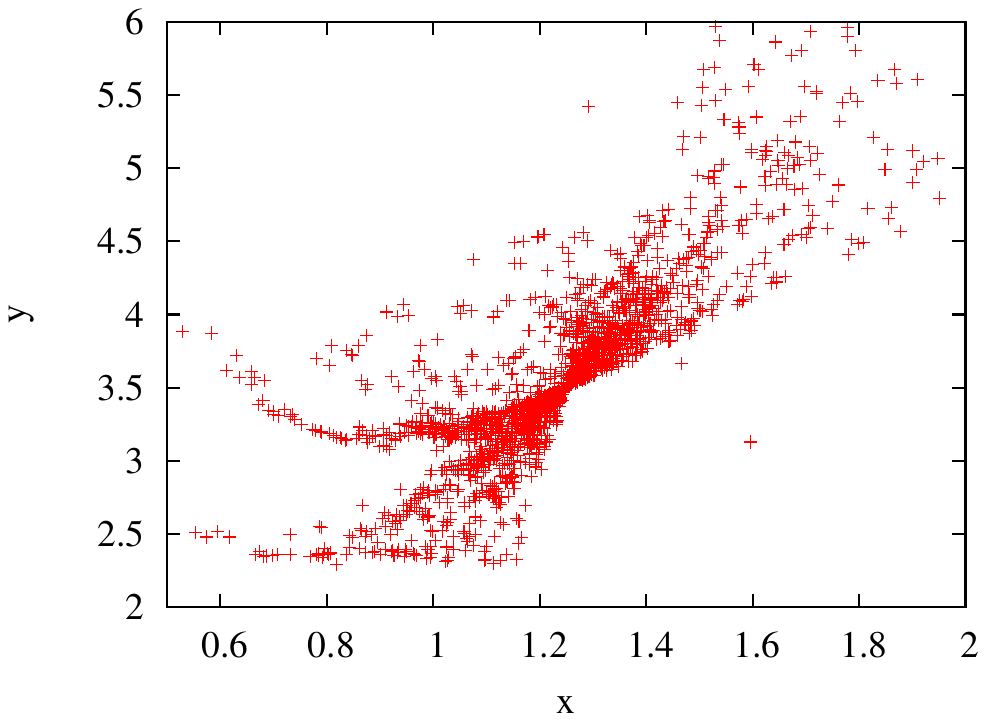}
\hspace{-10mm} \includegraphics[height=.265\textheight]{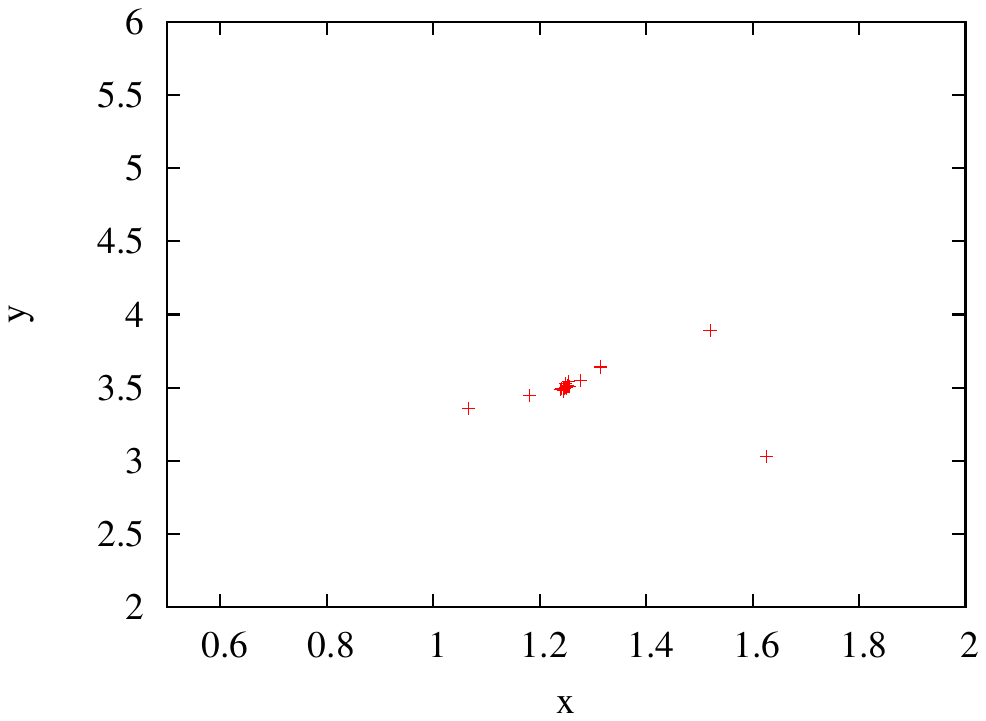}
\vspace{-6mm}

\caption{Runs for the Ricker map ($x$-coordinate in red and $y$ in green) for $r=3$, $s=4$, $a=0.5$, $b=0.4$,
$\alpha=0.35$, $\beta =0.45$  and (from left to right) (1)  no noise;
(2) $\ell = 0.05$, $\bar{\ell} = 0.06$;
(3) $\ell = 0.1$, $\bar{\ell} = 0.1$;
(4) $\ell = 0.15$, $\bar{\ell} = 0.2$. The bottom row shows the corresponding points in $xy$-plane.
We observe (1) a stable two-cycle;
(2) a noisy two cycle; (3) a noisy equilibrium; (4) stabilization of the positive equilibrium $(7/3,4/3)$.
Everywhere we chose $(x_0,y_0)=(4,1.1)$.
 }
\label{figure_4}
\end{figure}

\end{example}

\section{Conclusions and Discussion}
\label{sec:conclusions}

According to \cite{BR2023}, there are two types of systems and equilibrium points. 
Some of these points can be stabilized with strong enough PBC, others not.
The idea is easily illustrated by a one-dimensional map with multiple intersections of the line $y=x$.
Then, with strong enough PBC (the control coefficient being close enough to one), all locally unstable points $x^*$ with $(f(x)-x^*)(x-x^*)<0$ 
in some neighbourhood of $x^*$, $x \neq  x^*$ can be stabilized at once  \cite{BR2023}, while repellers with $(f(x)-x^*)(x-x^*)>0$ are in principle not stabilizable with this method
(unlike  some other approaches, such as the Target Oriented Control, allowing stabilization of a repeller). 
Further, systems can be classified as those whose stability can be improved when noise is introduced as a part of the stabilization parameters and those for which the situation deteriorates with noise \cite{pbclocsys}. 

In the present paper,   for the planar Ricker equation with variable PBC,  we
\begin{enumerate}
\item
illustrated that the unique positive equilibrium can be stabilized with PBC; moreover, the average of the noisy stabilizing parameters can be less than the stabilization bounds in the deterministic case;
\item
found sufficient conditions for PBC parameters guaranteeing local stabilization of the positive equilibrium;
\item
established global  stabilization  tests with the Lyapunov type function following the construction in  \cite{BHEBL};
to this end,  a closed invariant set in the first quadrant separated from the axes was designed;
\item
confirmed our results with numerical simulations and illustrated that some theoretically established bound can be further improved.
\end{enumerate}

For the positive influence of  noise, we follow the techniques from  \cite{BR2023,pbclocsys}.

Still, the current research is a step in exploration of Ricker systems.
\begin{enumerate}
\item
As Section~\ref{sec:ex}  illustrates, conditions on the control parameters, both in the deterministic and the stochastic cases, are just sufficient.
Obtaining sharp local stabilization conditions is still an open problem.
\item
On the way of advancing the previous item, justify monotonicity of stabilization: stronger control will always keep stabilization property.
In particular, prove that if for some $\alpha_0, \beta_0$ the spectral radius of $J_{\alpha_0, \beta_0}$ is less than one, then  
system \eqref{eq:RickPBCvar} has a locally stable solution at $K$ for any $\alpha_n\ge \alpha_0$,  $\beta_n\ge \beta_0$.
\item
Even a more challenging task is to justify that local stabilization implies global one, as we observed in simulations. Note that even without control, the fact of the equivalence of local and global asymptotic stability of the positive equilibrium is only justified for  $0<r,s \leq 2$ \cite{BHEBL}.
For noisy systems, there is some advantage that eventually a solution can get in the attracting domain leading to stabilization, once this attracting domain is invariant.
Another possible extension is to higher dimensional Ricker systems, exploring local and global stabilization, both in the deterministic and the stochastic cases.
\item
In the current paper, we explored a.s. local and global asymptotic stability. It will be interesting to consider, for instance,  stability
in probability or  mean-square stability. For the latter case, can some bifurcation theory be developed?
This is quite a challenging problem, as, to the best of our knowledge, there is no developed bifurcation theory 
for the planar Ricker model in the deterministic and non-controlled case.

\end{enumerate}




\section*{Acknowledgment}

E. Braverman was supported by the NSERC Grant RGPIN-2020-03934.
The authors are grateful to the anonymous referees whose thoughtful comments significantly improved presentation of our results.



\section{Appendix}
\label{sec:Ap}

\subsection{Proof of Lemma \ref{lem:locmax}}
\label{subsec:Ap1}

{\bf (a)}
First of all, let us analyze critical points of $g$ and find the  inflection  point $y_{\rm infl}$ of $g(\beta,0, \cdot)$ by taking the derivative:
\[
g'_y(\beta, 0, y) =(1-y)(1-\beta)e^{s-y} +  \beta, \quad g''_{y} (\beta, 0, y)=(y-2)(1-\beta)e^{s-y} .
\]
We see that $y_{\rm infl}=2$, $g''_{y} (\beta, 0, y)<0$ for $y<2$ and  $g''_{y} (\alpha, 0, y)>0$ for $y>2$.  Also $g'_y(\beta, 0, 0)=(1-\beta)e^{s} +\beta>0$, $g'_y(\beta, 0, 1)=\beta>0$, $\lim\limits_{y\to \infty}g'_y(\beta, 0, y)=\beta$, $g'_y(\beta, 0, y)<\beta$ for $y>1$.  
The point $y=2$ is the minimum for $g'_y$, so for a local maximum of $g$ to exist,   we need to have $g'_y(\beta, 0, 2) < 0$.
If  $g'_y(\beta, 0, 2) \geq 0$, $g$ increases in $y$ for any $y \in [0,\infty)$.

Next, the inequality $g'_y(\beta, 0, 2)\le 0$ holds if and only if
\[
-(1-\alpha)e^{s-2} +\alpha\le 0 \quad \Leftrightarrow \quad 
\beta \le \frac{e^{s-2}}{1+e^{s-2}} =\beta _1,
\]
see \eqref{def:beta12},  which implies that for 
\begin{equation}
\label{cond:locmax1}
\beta \geq  \beta_1=\frac{e^{s-2}}{1+e^{s-2}}=1-\frac{1}{1+e^{s-2}},
\end{equation}
the  function $g(\beta, 0, \cdot)$  has no local maximum and increases for all $y>0$. 
For $\beta < \beta_1$, the  function $g(\beta, 0, \cdot)$ has two critical points: one on $(1,2)$ with the negative second derivative, which is a local maximum, 
and a locat minimum on $(2,\infty)$ where the second derivative is positive. Thus, there is a unique local maximum of $g$ which is at $\bar y(\beta)\in (1,2)$.

Further, as
\begin{equation}
\label{equilibrium}
g(\beta, 0, y)>y, \quad y\in (0, s), \quad g(\beta, 0, y)<y, \quad y\in (s, \infty), \quad g(\beta, 0, s)=s, 
\end{equation}
we conclude that, once  \eqref{cond:locmax1} is satisfied, $g(\beta, 0, \cdot)$ maps the interval $(0, H)$ into itself for any $H>s \geq H_2(\beta) $.

Assume now that \eqref{cond:locmax1} does not hold, i.e. $\beta<\beta_1$, then there is a point of local maximum $\bar y(\beta) \in (1, 2)$, as demonstrated above.  
Since $ye^{-y}\le e^{-1}$ and $\bar y<2$,  we can estimate
\begin{equation}
\label{est:locmax}
g(\beta, 0, \bar y) =(1-\beta)\bar ye^{-\bar y}e^s +\beta \bar y<(1- \beta)e^{s-1} +2\beta  = H_2(\beta),
\end{equation}
which concludes the proof of Part\,(a).

{\bf (b)}  We have equivalence
\[
H_2(\beta) = (1-\beta)e^{s-1} +2\beta<s  \quad \Leftrightarrow \quad \beta >\frac{e^{s-1}-s}{e^{s-1}-2}=1-\frac {s-2}{e^{s-1}-2}=\beta_2,
\]
see \eqref{def:beta12}, which justifies  Part\,(b). Note that for $s>2$ the above inequality has solutions in $(0,1)$, since $\beta_2\in (0,1)$, 
while for $s\in (1+\ln 2, 2]$ the above inequality has no solutions in $(0,1)$. 
We also conclude that $\max_{y\le s}g(\beta, 0, y) < H_2(\beta) <  s$. 

{\bf (c) } 
First of all, since $0 \leq g(\beta, x, y) \leq g(\beta, 0, y)$ for any $x \geq 0$, we only have to prove that  $\displaystyle \max_{y\in [0,\mathcal H_2(\beta)] }g(\beta, 0, y) \leq \mathcal H_2(\beta)$.
Following Part\,(a), we distinguish between the two cases $\beta < \beta_1$ and $\beta \geq  \beta_1$. 

In the former case, the value of $g(\beta, 0, y)$ at the  only local maximum $\bar y$ is less than $H_2(\beta)$. Note that after $\bar y$, the function $g(\beta, 0, y)$ decreases in $y$ down to the minimum point exceeding  the inflection point $y=2$ and then increases up to the value of $s$ at $y=s$ (once $s>2$). Thus by \eqref{equilibrium}, $g(\beta, 0, \cdot)$  maps the segment $[0,\mathcal H_2(\beta)]$ onto itself, whether  $\mathcal H_2(\beta) =H_2(\beta)$ or $\mathcal H_2(\beta) =s$. 

In the latter case, the function $g(\beta, 0, \cdot)$  is monotone increasing satisfying \eqref{equilibrium} and thus, as is mentioned in the proof of Part\,(a), 
maps any $[0,H]$ to itself for $H \geq s$. Since $\mathcal H_2(\beta) \geq s$, we get $g(\beta, x, y)\le \mathcal H_2(\beta) $ and justify Part\,(c).

{\bf (d)}
To prove Part\,(d), we need to consider only $y\in (\mathcal H_2(\beta), H]$. Since $\mathcal H_2(\beta) \geq s$, we have  by \eqref{equilibrium} $g(\beta, 0, y)<y$
for $y>\mathcal H_2(\beta)$, then $g(\beta, x, y)\le g(\beta, 0, y)\le y<H$, which concludes the proof of the first part for $g(\beta, x, y)$.

The remaining parts (e)-(g)  for $f(\alpha, x,y)$ repeat all the steps of the first part, due to the symmetry of the expressions for the two functions $f$ and $g$.

\subsection{Proof of Lemma \ref{lem:alphainvar0}}
\label{subsec:Ap2}
To prove Part\,(a) we notice that $\mathcal H_1'(\alpha)=2-e^{r-1}<0$, when $\alpha<\alpha_2$ and $\mathcal H_1'(\alpha)=0$, when $\alpha\ge\alpha_2$. 
Part (b) follows from Part (a).

Now we prove Part\,(c).  Parts\,(a)-(b) yield that 
\begin{equation*}
\begin{split}
 &\mathcal H_1( \tilde  \alpha_1) \leq  \mathcal H_1( \tilde  \alpha_2), \quad c_1( \tilde  \beta_1)  \geq c_1( \tilde  \beta_2), \quad \mbox{so} \quad [c_1(\tilde \beta_1), \mathcal H_1(\tilde \alpha_1)]\subseteq [c_1(\tilde \beta_2), \mathcal H_1(\tilde \alpha_2)],\\
 &\mathcal H_2(\tilde \beta_1) \leq  \mathcal H_2(\tilde \beta_2), \quad c_2(\tilde \alpha_1) \geq c_1(\tilde \alpha_2), \quad \mbox{so} \quad [c_2(\tilde \alpha_1), \mathcal H_2(\tilde \beta_1)]\subseteq [c_2(\tilde \alpha_2), \mathcal H_2(\tilde \beta_2)],
\end{split}
\end{equation*}
which, in turn, implies \eqref{ineq:ab12} and 
\[
[\underline c_1(\tilde \alpha_1, \tilde \beta_1), \mathcal H_1(\tilde \alpha_1)]\subseteq  [\underline c_1(\tilde \alpha_2, \tilde \beta_2), \mathcal H_1(\tilde \alpha_2)],
\quad [\underline c_2(\tilde \alpha_2, \tilde \beta_2), \mathcal H_2(\tilde \beta_1)]\subseteq [\underline c_2(\tilde \alpha_2, \tilde \beta_2), \mathcal H_2(\tilde \beta_2)].
\]
To prove Part\,(d), we apply \eqref{def:mathcalH12} and, for  $x\in [c_1(\beta), \mathcal H_1(\alpha)], \, y\in [c_2(\alpha), \mathcal H_2(\beta)]$,  get the estimates
\begin{equation}
\label{est:xbel}
\begin{split}
x[(1-\alpha)e^{r-x-ay}+\alpha]& =xe^{-x}(1-\alpha)e^{r-ay}+\alpha x \\ & \ge u_{1}(\alpha, \beta)(1-\alpha)e^{r-a \mathcal H_2(\beta)}+\alpha c_1(\beta),\\
y[(1-\beta)e^{s-bx-y}+\beta] &=ye^{-x}(1-\beta)e^{s-bx}+\beta y \\ & \ge u_{2}(\alpha, \beta)(1-\beta)e^{s-b\mathcal H_1(\alpha)}+\beta c_2(\alpha).
\end{split}
\end{equation}
To prove the first inequality in Part (e), we notice that 
\[
f(\alpha, c_1(\beta), \mathcal H_2(\beta))=c_1(\beta)[(1-\alpha)e^{r-a\mathcal H_2(\beta)-c_1(\beta)}+\alpha]=c_1(\beta)\ge \underline c_1(\alpha, \beta),
\]
the same comment refers to $g(c_2(\alpha),\beta, \mathcal H_1(\alpha))$, concluding the proof.


\subsection{Proof of Lemma \ref{lem:alphainvar}}
\label{subsec:Ap3}

{\bf (a)} If  $x\in [c_{1}(\beta), \mathcal H_{1}(\alpha)]$, $y\in (0, \mathcal H_{2}(\beta)]$, estimation \eqref{est:xbel} shows that 
$f(\alpha, x, y)\ge \underline c_{1}(\alpha, \beta)$. Similarly, if  $x\in (0, \mathcal H_{1}(\alpha)]$, $y\in (c_2(\alpha), \mathcal H_{2}(\beta)]$, inequalities \eqref{est:xbel} lead to 
$g(\beta, x, y)\ge \underline c_{2}(\alpha, \beta)$. Also, Lemma~\ref{lem:locmax}\,(c) shows that $f(\alpha, x, y) \leq  \mathcal H_{1}(\alpha)$, $g(\beta, x, y) \leq \mathcal H_{2}(\beta)$ when $x\in (0, \mathcal H_{1}(\alpha)]$, $y\in (0 , \mathcal H_{2}(\beta)]$.

Let now $x\in [\underline c_{1}(\alpha, \beta), c_{1}(\beta))$, $y\in (0, \mathcal H_{2}(\beta)]$, then 
\[
e^{r-x-ay}\ge e^{r-a\mathcal H_2(\beta)-x}=e^{c_{1}(\beta)-x}\ge 1, \quad f(\alpha, x, y)=x\left[(1-\alpha)e^{r-ay-x}+\alpha \right]> x\ge \underline c_{1}(\alpha, \beta).
\]
For $x\in [\underline c_{1}(\alpha, \beta)-h, c_{1}(\beta)]$, $y\in (0, \mathcal H_{2}(\beta)]$, we have
\[
e^{r-x-ay}\ge e^{c_{1}(\beta)-x}\ge e^h>1, \quad f(\alpha, x, y)=x\left[(1-\alpha)e^{r-ay-x}+\alpha \right]> x\ge \underline c_{1}(\alpha, \beta)-h.
\] 

Similarly, for $y\in [\underline c_{2}(\alpha, \beta), c_{2}(\alpha))$, $x\in (0, \mathcal H_{1}(\alpha)]$, we get
\[
e^{s-y-bx}\ge e^{s-b\mathcal H_1(\alpha)-y}=e^{c_{2}(\alpha)-x}\ge 1, \quad g(\beta, x, y)=y\left[(1-\beta)e^{s-bx-y}+\beta \right]> y\ge \underline c_{2}(\alpha, \beta),
\]
and for $y\in [\underline c_{2}(\alpha, \beta)-h, c_{2}(\alpha))$, $x\in (0, \mathcal H_{1}(\alpha)]$,
\[
e^{s-y-bx}\ge e^h>1,  \quad g(\beta, x, y)=y\left[(1-\beta)e^{s-bx-y}+\beta \right]> y\ge \underline c_{2}(\alpha, \beta)-h.
\]

{\bf (b)} 
Lemma~\ref{lem:alphainvar0} implies $D_{\tilde \alpha_1, \tilde \beta_1}\subseteq  D_{\tilde \alpha_2, \tilde \beta_2}$. By Part\,(a),  we have
\[
T_{\tilde \alpha_1, \tilde \beta_1}:D_{\tilde \alpha_1, \tilde \beta_1}\to D_{\tilde \alpha_1, \tilde \beta_1}\in D_{\tilde \alpha_2, \tilde \beta_2}.
\]
Consider  now $(x, y)\in D_{\tilde \alpha_2, \tilde \beta_2}\setminus D_{\tilde \alpha_1, \tilde \beta_1}$.

 If $x\in [\mathcal H_1(\tilde \alpha_1), \mathcal H_1(\tilde \alpha_2)]$, we can apply Lemma~\ref{lem:locmax}, Part\,(d) and conclude that $f(\alpha, x, y)\le \mathcal H_1(\tilde \alpha_2)$ for each  $\alpha\in (0,1)$, $y\in [0, e^{s-1}]$. Similarly,  when $y\in [\mathcal H_2(\tilde \beta_1), \mathcal H_2(\tilde \beta_2)]$, by Lemma~\ref{lem:locmax}, Part (4) we get $g(\beta, x, y)\le \mathcal H_2(\tilde \beta_2)$ for each  $\beta\in (0,1)$, $x\in [0, e^{r-1}]$.

Let now $x\in [\underline c_{1}(\tilde \alpha_2, \tilde \beta_2), \underline c_{1}(\tilde \alpha_1, \tilde \beta_1)) $, $y<\mathcal H_2(\tilde \beta_2)$. 
If $x<c_1(\tilde \beta_2)$  then  $e^{r-ay-x}\ge e^{c_1(\tilde \beta_2)-x}>1$, and, for each $\alpha\in (0,1)$,
\begin{align*}
f(\alpha, x, y) & =x[(1-\alpha)e^{r-ay-x}+\alpha]\ge x[(1-\alpha)e^{r-a\mathcal H_2(\tilde \beta_2)-x}+\alpha] \\ &  =x[(1-\alpha)e^{c_1(\tilde \beta_2)-x}+\alpha]
> x>\underline c_{1}(\tilde \alpha_2, \tilde \beta_2).
\end{align*}
 Let  $x\in [c_1(\tilde \beta_2), \mathcal H_1(\tilde \alpha_2)]$ and $y\le \mathcal H_2(\tilde \beta_2)$. We estimate, as above,
\[
f(\tilde \alpha_1, x, y)\ge x[(1-\tilde \alpha_1)e^{c_1(\tilde \beta_2)-x}+\tilde \alpha_1] =  x[e^{c_1(\tilde \beta_2)-x}+\tilde \alpha_1 (1-e^{c_1(\tilde \beta_2)-x})].
\]
Since $e^{c_1(\tilde \beta_2)-x}<1$, $\tilde \alpha_1>\tilde \alpha_2$, we get, by applying definition \eqref{def:mathcalH12},
\begin{align*}
x[e^{c_1(\tilde \beta_2)-x}+\tilde \alpha_1 (1-e^{c_1(\tilde \beta_2)-x})]  & >x[e^{c_1(\tilde \beta_2)-x}+\tilde \alpha_2 (1-e^{c_1(\tilde \beta_2)-x})] \\ & =x[(1-\tilde \alpha_2)e^{c_1(\tilde \beta_2)-x}+\tilde \alpha_2]\ge \underline c_1(\tilde \alpha_2, \tilde \beta_2).
\end{align*}
We deal with $g(\beta, x, y)$ in a similar way. 

\subsection{Proof of Lemma \ref{lem:Tninvar}}
\label{subsec:Ap4}

{\bf (a)} Since $\alpha_n>\underline \alpha>\tilde \alpha$, $\beta_n\ge \underline \beta>\tilde \beta$, for each $n\in \mathbb N$, by Lemma~\ref{lem:alphainvar} (b) we conclude  that $T_{\alpha_{n}, \beta_{n}}(x, y)\in D_{\underline \alpha, \underline\beta}$ as soon as $(x, y)\in D_{\underline \alpha, \underline\beta}$. 
Taking $(x_0, y_0)\in D_{\underline \alpha, \underline\beta}$, applying \eqref{def:Tn+1} and reasoning by induction conclude the proof.

The same argument can be applied in Part {\bf (b)}.

{\bf(c)}  Recall from \eqref{def:Dalphabeta} that $D_{\underline \alpha, \underline\beta}(h)=[\underline c_{1}(\underline\alpha, \underline\beta)-h, \mathcal H_{1}(\underline\alpha)]\times [\underline c_2(\underline\alpha-h, \underline\beta), \mathcal H_{2}(\underline\beta)]$. 
Assume $x_0>\mathcal H_1(\underline \alpha)$ and set 
$\displaystyle 
\Delta_1:= (1-\bar \alpha)\mathcal H_1(\underline \alpha)[ 1-e^{r-\mathcal H_1(\underline \alpha)}], $ $$\bar m_1:=\left\lfloor\frac{x_0-\mathcal H_1(\underline \alpha)}{\Delta_1} \right\rfloor +1, ~  m_1:=\inf\{i\le \bar m_1: x_i<\mathcal H_1(\underline \alpha)\}.
$$
Since $x_0>\mathcal H_1(\alpha_1)\ge r$,  we get $e^{r-\mathcal H_1(\underline \alpha)}< 1$, so $\Delta_1>0$. As $y_0>0$, we have
\[
\begin{split}
x_0-x_1 & = x_0-x_0[ (1-\alpha_1)e^{r-x_0-ay_0}+\alpha_1]=(1-\alpha_1)x_0[ 1-e^{r-x_0-ay_0}]\\&>(1-\bar \alpha)\mathcal H_1(\underline \alpha)[ 1-e^{r-\mathcal H_1(\underline \alpha)}]=\Delta_1,
\end{split}
\]
hence  $x_1<x_0-\Delta_1$. If $x_1>\mathcal H_1(\underline \alpha)$, we repeat this process, but not more than $\bar m_1$ times: $x_i<x_{i-1}-\Delta_1$, $i\le \bar m_1$.  

A similar argument can be applied to $y$ if $y_0>\mathcal H_2(\underline \beta)$,  with
\[
\Delta_2:= (1-\bar \beta)\mathcal H_2(\underline \beta)[ 1-e^{s-\mathcal H_2(\underline \beta)}], ~\bar m_2:=\frac{y_0-\mathcal H_1(\underline \beta)}{\Delta_2}, ~ m_2:=\inf\{i\le \bar m_2: y_i<\mathcal H_2(\underline \beta)\}.
\] 
Set
$$
\underline x_0:=\mathcal H_1(\underline \alpha)\left [(1-\underline \alpha)e^{r-x_0-ay_0}+\underline \alpha\right], \quad \Delta_{1,1}:=\underline x_0(1-\bar \alpha)(e^h-1),
$$ 
$$
 m_{1,1}:= \left\lfloor \frac{\underline c_1(\underline\alpha, \underline\beta)-h-\underline x_0}{\underline x_0(1-\bar \alpha)(e^h-1)} \right\rfloor +1.
$$
By Lemma~\ref{lem:locmax}\,(c),(d),  $x_n\le \mathcal H_{1}(\underline\alpha)$ for all $n\ge \bar m_1$, $y_n\le \mathcal H_{2}(\underline\beta)$ for all $n\ge \bar m_2$. Assume that $\bar m_1<\bar m_2$.

After $m_1$ steps  we have either  $x_{m_1}\in  [\underline c_{1}(\underline\alpha, \underline\beta)-h, \mathcal H_{1}(\underline\alpha)]$  or  $x_{m_1}<\underline c_1(\underline\alpha, \underline\beta)-h$. In the first case the solution stays there. Consider the second case:
 $x_{m_1}<\underline c_1(\underline\alpha, \underline\beta)-h$ and assume  that  $m_1<m_2$ when
 we also have $x_0>x_{m_1-1}>\mathcal H_1(\underline \alpha)$.
Then
\[
\begin{split}
x_{m_1}&=x_{m_1-1}\left [(1-\alpha_{m_1})e^{r-x_{m_1-1}-ay_{m_1-1}}+\alpha_{m_1}\right]\\&\ge \mathcal H_1(\underline \alpha)\left [(1-\underline \alpha)e^{r-x_0-ay_0}+\underline \alpha\right]  =:\underline x_0.
\end{split}
\]
Since $x_i\ge x_{i-1}$, as soon as $x_i<c_1(\underline \alpha)$, we get $x_i\ge \underline x_0$ for all $i\in \mathbb N$. Also, once $y_0\ge y_{m_1-1}\ge y_{m_1} >\mathcal H_2(\underline \beta)$, we have $y_{i}<\mathcal H_2(\underline\beta)$ for all $i\ge m_2$. Then, if $x_{m_2}<\underline c_1(\underline\alpha, \underline\beta)-h$, we get 
\[
\begin{split}
x_{m_2+1}-x_{m_2}&=x_{m_2}\left [(1-\alpha_{m_2+1})e^{r-x_{m_2}-ay_{m_2}}+\alpha_{m_2+1}\right]-x_{m_2}\\&
=x_{m_2}(1-\alpha_{m_2+1})\left [e^{r-x_{m_2}-ay_{m_2}}-1\right]\ge  \underline x_0(1-\bar \alpha)\left [e^{r-\underline c_1(\underline\alpha, \underline\beta)+h-a\mathcal H_2(\underline \beta)}\right],\\
&\ge \underline x_0(1-\bar \alpha)[e^h-1]:=\Delta_{1,1}
\end{split}
\]
and, for any $i$, as soon as $x_{m_2+i}<\underline c_1(\underline \alpha, \underline \beta)-h$,
\[
\begin{split}
x_{m_2+i+1}-x_{m_2+i}&=x_{m_2+i}\left [(1-\alpha_{m_2+i+1})e^{r-x_{m_2+i}-ay_{m_2+i}}+\alpha_{m_2+i+1}\right]-x_{m_2+i}\\&
\ge x_{m_2+i}(1-\bar \alpha)\left [e^{r-x_{m_2+i}-ay_{m_2+i}}-1\right]\ge x_{m_2+i}(1-\bar \alpha)\left [e^{r-\underline c_1(\underline\alpha, \underline\beta)+h-a\mathcal H_2(\underline \beta)}\right],\\
&= x_{m_2+i}(1-\bar \alpha)[e^h-1]\ge \underline x_0(1-\bar \alpha)[e^h-1].
\end{split}
\]

So, after at most $m_{1,1}$ steps,  the solution $x$ is in  $[\underline c_{1}(\underline\alpha, \underline\beta)-h, \mathcal H_{1}(\underline\alpha)]$.  
All the other cases when $x_0<\underline c_1(\underline \alpha, \underline \beta)$  and  $y_0\notin  [\underline c_2(\underline\alpha, \underline\beta), \mathcal H_{2}(\underline\beta)]$ are treated in the same way. We define $m_{2,2}$ similarly to $m_{1,1}$.
Therefore we can set the maximum number of steps $\tilde S$ as
\[
\tilde S=\max\{\bar m_1, \bar m_2\}+\max\{m_{1,1}, m_{2,2}\},
\]
which concludes the proof.

\subsection{Proof of Lemma \ref{lem:locstPBCpar}}
\label{subsec:locstPBCpar}

Part {\bf (a)} was fully justified above.

{\bf  (b) } We set, for $x, y>0$ , 
\begin{equation}
\label{def:G}
\begin{split}
 & G(x,y):=2-x-y+\frac{xy(1-ab)}2, \quad g(x) :=\frac{2-x}{1-\frac{1-ab}2x}, \,\, x\neq \frac2{1-ab},\\
 & \mathcal E:=\{(x, y): G(x,y)>0,\,  x+y<4, \, x,y>0\}.
  \end{split}
\end{equation}
Assume first that $x\in \left(0, \frac 2{1-ab}\right)$. Then  $G(x,y)>0$ if and only if 
\begin{equation*}
\label{def:g_function}
\begin{split}
y < g(x) &=\frac{2-x}{1-\frac{1-ab}2x}=\frac{x-2}{\frac{1-ab}2x-1}=\frac2{1-ab}\frac {x-2}{x-\frac2{1-ab}} \\ & =
\frac2{1-ab}\left[1+\frac {\frac2{1-ab}-2}{x-\frac2{1-ab}}\right]=\frac2{1-ab}+\frac {\frac{4ab}{(1-ab)^2}}{x-\frac2{1-ab}} \, .
 \end{split}
\end{equation*}
Note that $g$ defined in \eqref{def:G} decreases  everywhere in its domain, $x\neq \frac2{1-ab}$,
since $g^{\prime}(x)=-\frac {\frac{4ab}{(1-ab)^2}}{\left(x-\frac2{1-ab}\right)^2}<0$, and $g$  intersects the $y$-axis  at $y=2$ and the $x$-axis at $x=2$. 
Thus, the graph of $g$ on $[0,2]$ is  completely in the square $[0, 2]\times [0,2]$, therefore the inequality  $0<x+y<4$ holds.

When $x\in \left[2, \frac 2{1-ab}\right)$,  we get $y=g(x)<0$,  since
\[
x-\frac 2{1-ab}>2-\frac 2{1-ab}=-\frac {2ab}{1-ab}, \quad y<\frac 2{1-ab}-\frac{\frac {4ab}{(1-ab)^2}}{\frac {2ab}{1-ab}}=\frac {2ab}{1-ab}-\frac {2ab}{1-ab}=0,
\]
so $(x, y)\notin \mathcal E$ for any $y\in (0, g(x))$.

Assume now that $x>\frac2{1-ab}$, then $G(x,y)>0$ if and only if
\[
y>\frac2{1-ab}+\frac {\frac{4ab}{(1-ab)^2}}{x-\frac2{1-ab}}=g(x)>\frac2{1-ab},
\]
while  $x+y<4$ implies  
$y<4-x<\frac{2-4ab}{1-ab}$. This leads to the contradiction, since $\frac{2-4ab}{1-ab}<\frac2{1-ab}.$
Thus, we proved that the domain  $\mathcal E$ also has the form  $\mathcal E=\{( x, y):x\in\left(0, 2\right), y\in (0, g(x)\}$.

Now we substitute $x = p_\alpha=(1-\alpha)p$ and $y=q_\beta=(1-\beta)q$.  Then $(x, y)\in \mathcal E$ if and only if  
$\alpha>\alpha^*=\max\{1-2/p, 0\}$ and  
\[
(1-\beta)q<\frac2{1-ab}+\frac {\frac{4ab}{(1-ab)^2}}{(1-\alpha)p-\frac2{1-ab}},  \quad \mbox{or} \quad \beta>1-\frac2{q(1-ab)}-\frac {\frac{4ab}{(1-ab)^2}}{q[(1-\alpha)p-\frac2{1-ab}]},
\]
which proves the lemma.
\bigskip
\subsection{Proof of Lemma \ref{lem:locstabst1_var}}
\label{subsec:var}
{\bf (a)}
The statement is based on the fact that a constant fixed point $K$ of a non-autonomous system of nonlinear difference equations
is locally asymptotically stable if there is $\lambda \in (0,1)$ such that the norms of the Jacobians satisfy $\| J_n(K)\| \leq \lambda$, see \cite[Chapter 4.6]{Elaydi} for overview of the linearisation method.

{\bf (b)}
Recall that the Jacobian of the controlled system at each step, see \eqref{Jacobian_control} and \eqref{def:paqb}, $p_\alpha=(1-\alpha)p$, $q_\beta=(1-\beta)q$, is
\begin{equation}
\label{Jacobian_control_var}
J_{\alpha_n,\beta_n}  =  \left(\begin{array}{ll}
1-p_{\alpha_n}  &  -ap_{\alpha_n} \\ -bq_{\beta_n}& 1-q_{\beta_n}.
 \end{array}\right).
\end{equation}
Then the max-norm of the Jacobian $\|  J_{\alpha_n,\beta_n}  \|_{\infty}$  is
\[
\|  J_{\alpha_n,\beta_n}  \|_{\infty}=\max\{|1-p_{\alpha_n}|+ap_{\alpha_n}, \,\, |1-q_{\beta_n}|+bq_{\beta_n}\}.
\]
Let us estimate the first expression in the maximum. First, let $p_{\alpha_n} \leq 1$. Then, we have 
\[
0\le |1-p_{\alpha_n}|+ap_{\alpha_n}=1-(1-a)p_{\alpha_n}  <1, \quad \mbox{as} \quad (1-a)p_{\alpha_n}<p_{\alpha_n} \leq 1.
\]
To get a uniform estimate $1-(1-a)p_{\alpha_n}  \leq \lambda <1$, we need to have 
$p(1-\alpha_n) \geq \varepsilon$ for some $\varepsilon > 0$. This is guaranteed if
$\alpha_n \leq \alpha^* <1$, which is assumed in the conditions of the lemma. Thus, we should only consider $p_{\alpha_n} > 1$.

Next, for  $p_{\alpha_n} > 1$  we get
\[
0\le |1-p_{\alpha_n}|+ap_{\alpha_n}=p_{\alpha_n} - 1 + ap_{\alpha_n} = (1+a)p_{\alpha_n}-1\leq \lambda <1 \quad \mbox{if} \quad 1<p_{\alpha_n}<\frac{1+\lambda}{1+a},
\]
which corresponds to $1-\frac2{p(1+a)}+ \varepsilon \leq \alpha_n <1-1/p$ for some $\varepsilon > 0$. As mentioned above, the right inequality can be satisfied or not.
The left inequality is valid due to the choice of the lower bound $\alpha_*$ in \eqref{eq:choice_q_1}.
Similar estimations can be done for $ |1-q_{\beta_n}|+bq_{\beta_n}$.


\begin{remark}
\label{rem:othernorms}
In Lemma~\ref{lem:locstabst1_var} we applied the maximum norm to derive the low bound for the controls to ensure local stability for system \eqref{eq:RickPBCvar} with variable control.  We can use other  norms, such as the traffic or the spectral norm. 
\begin{enumerate}
\item [(a)]
For the traffic norm, the fact that the norm does not exceed $\lambda \in (0,1)$ is equivalent to
$$
|1-(1-\alpha_n) p| +  (1-\beta_n)bq \leq \lambda, \quad  |1-(1-\beta_n)q| + (1-\alpha_n) a p \leq \lambda.
$$
Some sufficient conditions on $\alpha_n$ can be deduced, once the bounds for $\beta_n\in [\beta_*,\beta^*]$  are known, or vice versa.

\item [(b)]
For the spectral norm, we have to evaluate when the modulus of each eigenvalue of 
\begin{align*}
&A_n:=J_{\alpha_n,\beta_n}^T J_{\alpha_n,\beta_n} \\ & =  \left(\begin{array}{ll}
(1-p_{\alpha_n})^2  + b^2  q_{\beta_n}^2  &  -ap_{\alpha_n} (1-p_{\alpha_n}) -  bq_{\beta_n} ( 1-q_{\beta_n}) \\ -
ap_{\alpha_n} (1-p_{\alpha_n}) -  bq_{\beta_n} ( 1-q_{\beta_n}) & (1-q_{\beta_n} )^2  + a^2  p_{\alpha_n}^2 
 \end{array}\right)
\end{align*}
is less than 1. We will illustrate that existence of $\varepsilon_1>0$, $\varepsilon_2>0$   such that 
\begin{equation*}
\label{eq:spectral_norm}
\begin{split}
&   (1-p_{\alpha_n})^2 + b^2  q_{\beta_n}^2 + (1- q_{\beta_n})^2 + a^2  p_{\alpha_n}^2  -  \left[  (1-p_{\alpha_n})   (1- q_{\beta_n}) -  a b  p_{\alpha_n}  q_{\beta_n} \right]^2   \leq  1 - \varepsilon_1,
\\
& \left|  (1-p_{\alpha_n})   (1- q_{\beta_n}) -  a b  p_{\alpha_n}  q_{\beta_n} \right|       \leq 1-\varepsilon_2
\end{split}
\end{equation*}
leads to the required estimate of the spectral norm. Obviously
{\rm tr}~$A_n=  (1-p_{\alpha_n})^2 + b^2  q_{\beta_n}^2 + (1- q_{\beta_n})^2 + a^2  p_{\alpha_n}^2 $,  and also
 ${\rm det~}A_n= \left[  (1-p_{\alpha_n})   (1- q_{\beta_n}) -  a b  p_{\alpha_n}  q_{\beta_n} \right]^2$, since
\begin{align*}
{\rm det~}A_n = &  \left[ (1-p_{\alpha_n})^2 + b^2  q_{\beta_n}^2   \right]  \left[  (1- q_{\beta_n})^2 + a^2  p_{\alpha_n}^2    \right]  -  \left[   a p_{\alpha_n}(1-p_{\alpha_n})  + bq_{\beta_n}(  1- q_{\beta_n} )      \right] ^2
\\
= & \left[  (1-p_{\alpha_n})   (1- q_{\beta_n}) -  a b  p_{\alpha_n}  q_{\beta_n} \right]^2.
\end{align*}
We use a slight modification of the well-known criterion and state that all the eigenvalues of $A_n$ do not exceed $\lambda$ for some $\lambda \in (0, 1)$, once there are $\varepsilon_1, \varepsilon_2\in (0,1)$  such that 
${\rm tr~}A_n - {\rm det~}A_n \leq 1 - \varepsilon_1,$    $ {\rm det~}A_n \leq1 - \varepsilon_2$,  $n \in {\mathbb N}$,
which leads to the sufficient condition above.
\end{enumerate}
Unlike the maximum-norm, these two conditions are harder to verify.
\end{remark}

\subsection{Proof of Lemma \ref{lem:XYxy}}
\label{subsec:XYxy}


{\bf (a) } We have 
\begin{equation*}
\begin{split}
&X=r-x-ay=r-(x-p)-p-a(y-q)-aq
=(p-x)+a(q-y),\\
&Y=(q-y)+b(p-x).
\end{split}
\end{equation*}
Since $X-aY = (1-ab) (p-x)$, $Y-bX = (1-ab) (q-y)$,   $0<a, b<1$, we get 
\[
|p-x|=\frac 1{1-ab}|X-aY|\le\frac 1{1-ab}(|X|+a|Y|)\le\frac 1{1-ab}(|X|+|Y|)
\]
and, similarly, $|q-y|\le \frac1{1-ab}(|X|+|Y|)$. Then,
\[
|p-x|^2+|q-y|^2\le \frac 2{(1-ab)^2}(|X|+|Y|)^2\le \frac 4{(1-ab)^2}(|X|^2+|Y|^2).
\]
So 
\[
|X|^2+|Y|^2\ge  \frac {(1-ab)^2}4 \left( |p-x|^2+|q-y|^2 \right) \ge  \frac {\delta_0^2(1-ab)^2}4=\rho_0^2.
\]
Since $|X|^2+|Y|^2<(|X|+|Y|)^2$, this implies ~$2\max\{|X|, |Y|\}\ge |X|+|Y|>\rho_0$, which completes the proof.

{\bf (b) }is straightforward, and we omit the proof.

{\bf (c)} Fix some $\rho>0$.  It was shown above in the proof of Lemma~\ref{lem:alphaLyap} that the continuous functions satisfy $\Psi(\alpha, X, r)>0$, $\Psi(\beta, Y, s)>0$ for $X, Y\neq 0$.
Since $1\notin [\rho_i, \bar \rho_i]$, $i=1,2$, and $0\notin [-H_1, -\rho]\cup [\rho, H_1]$,
the continuous positive functions $\Psi(\cdot, \cdot, r)$ and  $\Psi(\cdot, \cdot, s)$  in the compact sets 
$[\rho_1, \bar \rho_1]\times \left([-H_1, -\rho]\cup [\rho, H_1]\right)$ and $[\rho_2, \bar \rho_2]\times \left([-H_1, -\rho]\cup [\rho, H_1]\right)$, respectively,
attain their minimum values which are also positive. 
This allows us to set 
\begin{equation}
\begin{split}
\label{def:psi}
&\underline \psi_1(\rho):=\min\{\Psi(\alpha, u, r): (\alpha, u)\in [\rho_1,  \bar \rho_1]\times \left([-H_1, -\rho]\cup [\rho, H_1]\right)\},\\
&\underline \psi_2(\rho):=\min\{\Psi(\beta, u, s): (\beta, u)\in [\rho_2,  \bar \rho_2]\times \left([-H_1, -\rho]\cup [\rho, H_1]\right)\},\\
&\underline \psi(\rho):=\min\{\psi_1(\rho), \psi_2(\rho)\},
\end{split}
\end{equation}
justifying  {\bf (c)} and concluding the proof.

\bigskip

\subsection{Proof of Lemma \ref{lem:stepsconsta}}
\label{subsec:stepsconsta}
{\bf (a)}  Fix some $\delta_0>0$ and  find $\rho_0$ as in \eqref{def:rho}. Using \eqref{def:mathcalH12}, \eqref{def:psi}, \eqref{def:globeta}, \eqref{def:barunder}, we define 
\begin{equation}
\begin{split}
\label{def:steps}
& \mathbf c:=\min\{\underline c_i(\underline \alpha, \underline \beta), i=1,2\},\quad 
S:=  \left\lfloor \frac{2\bar M}{ \min\{b, a\}\mathbf c\underline\psi(\rho_0)}  \right\rfloor + 1,
\end{split}
\end{equation}
where $\lfloor t \rfloor$ is an integer part of $t$. By  \eqref{def:globeta} we have $|V(x_n, y_n)|\le \bar M$, $n\in {\mathbb N}$. 
We assume $(x_0, y_0)\in D_{\underline \alpha, \underline\beta}\setminus B(K, \delta_0)$, then also $|V(x_0, y_0)|\le \bar M$. 
For every $(x,y) \not\in B(K, \delta_0)$, the negative lower bound of $\Delta V(x,y)$ allows to evaluate the maximum number of steps outside this neighbourhood of $K$.

By Lemma~\ref{lem:Tninvar}, $(x_n, y_n)\in D_{\underline \alpha, \underline \beta}$ for all $n\in \mathbb N$, so, by definition \eqref{def:Dalphabeta} of $D_{\underline \alpha, \underline \beta}$, we get
$
|x_n|>c_1(\underline \alpha, \underline \beta)\ge \mathbf c$ and  $|y_n|>c_2(\underline \alpha, \underline \beta)\ge \mathbf c$.
Fix some $n\in \mathbb N$ and assume that $(x_i, y_i)\in D_{\underline \alpha, \underline\beta}\setminus B(K, \delta_0)$ for all $i < n$. In this case we have $|x_i-p|^2+|y_i-q^2|>\delta_0^2$ and then, by Lemma~\ref{lem:XYxy}\,(a), for $i < n$,  either $|X_i|\ge \rho_0$ or $|Y_i|\ge \rho_0.$ Then, by Lemma~\ref{lem:XYxy}\,(c), at least one of the following
inequalities  holds: 
\[
\Psi(\alpha, X_i, r)\ge \underline \psi(\rho_0) \quad \mbox{or} \quad \Psi(\alpha, Y_i, s)\ge \underline \psi(\rho_0).
\]
Now we apply inequality \eqref{ineq:DeltaV} with $x_i, y_i,X_i, Y_i$ replacing $x,y,X,Y$ and $\alpha_{i+1}, \beta_{i+1}$ instead of $\alpha, \beta$, respectively, 
and  Lemma~\ref{lem:XYxy}\,(c) to get
\begin{equation*}
\begin{split}
\Delta V_{\alpha_{i+1}, \beta_{i+1}}(x_{i}, y_{i})&=V(x_{i+1}, y_{i+1})-V(x_{i}, y_{i})\\&\le -bx_i\Psi(\alpha_{i+1}, X_i, r)-ay_i\Psi(\beta_{i+1}, Y_i, s)\le -\min\{a, b\}\mathbf c \psi(\rho_0).
\end{split}
\end{equation*}
Hence 
\begin{equation*}
\begin{split}
 V(x_{n+1}, y_{n+1})   & \le \sum_{i=0}^{n}\bigl[-bx_i\Psi(\alpha_{i+1}, X_i, r)-ay_i\Psi(\beta_{i+1}, Y_i, s)\bigr]+V(x_0, y_0)  \\
&\le -(n+1) \min\{a, b\}\mathbf c \psi(\rho_0)+V(x_0, y_0).
\end{split}
\end{equation*}
By \eqref{def:steps}, for any  $n\ge  S$ we get
$
(n+1)\mathbf c\min\{a,b\}\underline \psi(\rho_0)\ge S\mathbf c\min\{a,b\}\underline \psi(\rho_0)=2\bar M,
$
and therefore,
\[
V(x_{n}, y_{n})<-2\bar M+V(x_0, y_0) <-2\bar M+\bar M=-\bar M,
\]
which contradicts to  definition \eqref{def:globeta} of $\bar M$. Thus at some moment $n\le S$ it must be $(x_n, y_n)\in B(K, \delta_0)$, which proves  Part (a).

{\bf (b)} As $V$ does not increase along the solution, the steps of Part (a) justified that among all $(x_n,y_n)$, there is a total of less than $S$ terms of the sequence not belonging to $ B(K, \delta_0)$.

Part {\bf (c)}  follows from the fact that for any $\delta>0$, there are only a finite number of terms $(x_n, y_n)$  outside of the $\delta$-neighbourhood of $K$.


\bigskip


\subsection{Proof of Lemma \ref{lem:locKolmtauk}}
\label{subsec:locKolmtauk}

For simplicity of notations we set
\[
J(n):=J_{\alpha+\ell \xi_{k+n+1}, \beta+\bar\ell \chi_{k+n+1}}=J_{\alpha_n(k), \beta_n(k)}, \quad \hat G_{n}(\cdot):=\hat G(\alpha_{n}(k), \beta_{n}(k), \cdot),
\]
where $\hat G$ is defined as in \eqref{eq:RickPBCst1}.   From the first line of \eqref{def:ethamu}, we get , for  $k=0, 1, \dots, S$,
 \[
 \sup_{k\in {\mathbb N}, n\le \rm N}\left\{  \left\| \prod_{i=k}^{n+k} J(i)  \right\|  \right\}\le \mathbb J^{\rm N}.
 \]
 By \eqref{eq:RickPBCstk}, after $n$ iterations we arrive at 
 \begin{equation}
\begin{split}
\label{eq:maindecomp}
Z^{[k]}_{n+1}    &=J(n+1)Z^{[k]}_n+ \hat G_{n}(Z^{[k]}_n)\\&=J(n+1)J(n)Z^{[k]}_{n-1}+ J(n+1)\hat G_{n-1}(Z^{[k]}_{n-1})+ \hat G_{n}(Z^{[k]}_n)\\
&=\prod_{i=1}^{n+1} J (i)Z^{[k]}_0+\sum_{i=-1}^{n-1}\prod_{j=0}^i  J(n+1-j)\hat G(Z^{[k]}_{n- i -1}),
\end{split}
\end{equation}
where the product $\displaystyle \prod_{j=k}^p$ with $k > p$ is assumed to be equal to one.

To justify our result,  we are going to show that  $\|Z_n^{[k]}\|\le \eta e^{-\mu n}$, $n\in \mathbb N$. 
We implement the proof by induction in $n$. 

Let $n=1$, then the last  line in  \eqref{def:ethamu} implies 
$\delta_0<1$, $\delta_0( \mathbb J+\hat C \delta_0)e^{\mu} \leq \eta$, and then, putting $n=1$ in \eqref{eq:maindecomp}, 
on $\Lambda_{\gamma,k}$, defined as in \eqref{def:POgamk}, we get 
\[
\| Z_1^{[k]}  \| \le \| \tilde J(1) \|   \| Z_0^{[k]} \| + \| \hat G_0 (Z^{[k]}_0) \|
\leq \left[ \delta_0( \mathbb J+\hat C)e^{\mu} \right] e^{-\mu} \le \eta e^{-\mu}.
\]
For the induction step from $n\in {\mathbb N}$ to $n+1$, assume that $\|Z_s^{[k]}\|\le \eta e^{-\mu s}$ on $\Lambda_{\gamma, k}\cap \Omega_{\tau k}$, 
where $\Lambda_{\gamma, k}$,  $\Omega_{i k}$   are defined in \eqref{def:POgamk},  for all $s\le n<\rm N$. It means in particular that $\|Z_s^{[k]}\|\le \eta<\delta$ for all $s\le n<\rm N$. Applying the two last lines in \eqref{def:ethamu}, we get  from  \eqref{eq:maindecomp}
 \begin{equation*}
     \begin{split}
    &  \|Z_{n+1}^{[k]}\|\le \left\| \prod_{i =1}^{n+1}J(i)  \right\| \|Z_0^{[k]}\|+\sum_{i=-1}^{n-1} \left\|\prod_{j=0}^{i}J(n+1-j)
		\right\|  \hat C  \|  Z_{n- i -1}^{[k]}  \|^2   \\
    \le & e^{-\mu (n+1)} e^{\mu (n+1)} \mathbb J^{n+1}  \| Z^{[k]}_0 \| +e^{-\mu (n+1)} \sum_{ i =-1}^{n-1} e^{\mu (n+1)} \mathbb J^{i+1} \hat C \eta^2 e^{-2\mu(n- i -1)} \\
 \\ \le  &  \eta e^{-\mu (n+1)}\left[ \eta^{-1} e^{\rm N+1}\mathbb J^{\rm N+1}  \delta_0+  \eta \hat C\mathbb J^{\rm N+1}\sum_{i =-1}^{\rm \rm N-1} e^{3\mu-\mu n+2\mu  i}\right]\le  \frac 23\eta e^{-\mu (n+1)}<\eta e^{-\mu (n+1)}.
        \end{split}
     \end{equation*}  
 Now we show that $\| Z^{[k]}_j \|\le \eta e^{-\mu j}$ on  $\Lambda_{\gamma, k}\cap \Omega_{i k}$, 
for all $i \le n$ and $n\ge \rm N$. 
Splitting  the second term in the second line of \eqref{eq:maindecomp} 
 into two parts and estimating them separately, we get
		 \begin{equation*}
     \begin{split}
& 
\left\| \sum_{i =-1}^{\rm N -1} \prod_{j=0}^{i} J(n+1-j)\hat G(Z^{[k]}_{n- i -1})+\sum_{i =\rm N}^{n-1}\prod_{j=0}^{i}J(n+1-j)\hat G(Z^{[k]}_{n- i -1}) \right\|\\
 \le  & \sum_{ i =-1}^{\rm \rm N-1} 
		\left\|\prod_{j=0}^{i} J(n+1-j)\right\| \hat C \| Z^{[k]}_{n- i -1} \|^2+\sum_{i =\rm \rm N}^{n-1} \left\|\prod_{j=0}^{i }J(n+1-j) \right\| \hat C \|Z^{[k]}_{n- i -1}\|^2   = : \mathcal A_1+\mathcal A_2.
  \end{split}
     \end{equation*} 

     To estimate $\mathcal A_1$, we act as in the case $n<\rm N$, applying the inequality 
		$$ \displaystyle \eta\le \frac 13 \left(\hat C \mathbb J^{\rm N+1}\sum_{i  =-1}^{\rm \rm N-1} e^{3\mu+2\mu  i}\right)^{-1}$$ and getting
  \begin{equation*}
     \begin{split}
 \mathcal A_1 & \le \eta e^{-\mu (n+1)}  \hat C \eta\mathbb J^{\rm N+1}\sum_{i =-1}^{\rm \rm N-1} e^{\mu (n+1)} e^{-2\mu(n- i -1)}\le \eta e^{-\mu (n+1)}  \hat C \eta\mathbb J^{\rm N+1}\sum_{i =-1}^{\rm \rm N-1} e^{3\mu-\mu n+2\mu i}\\
  &\le \eta e^{-\mu (n+1)} e^{-\mu n} \eta \left[ \hat C \mathbb J^{\rm N+1}\sum_{i =-1}^{\rm \rm N-1} e^{3\mu+2\mu i}\right]<\frac 13 \eta e^{-\mu(n+1) }.
  \end{split}
     \end{equation*} 
To estimate $\mathcal A_2$,  we apply \eqref{def:ethamu} and \eqref{est:Jprod}   to obtain
\begin{equation*}
     \begin{split}
  \mathcal A_2  & \le \sum_{i =\rm \rm N}^{n-1} e^{-  i \nu/2}\hat C\eta^2 e^{-2\mu(n- i -1)}  \le \eta e^{-\mu(n+1)}\hat C\eta\sum_{i = \rm \rm N}^{n-1}  e^{\mu(n+1)}e^{- i \nu/2}e^{-2\mu(n- i -1) }\\
  &\le \eta e^{-\mu(n+1)} \eta  \hat C e^{3\mu} e^{-\mu n}\sum_{ i= \rm \rm N}^{n-1} e^{\mu i} e^{-\tau (\nu/2-\mu)}\le \eta e^{-\mu(n+1)} \eta \mathbb M<\frac 13 \eta e^{-\mu(n+1) }.
     \end{split}
     \end{equation*} 
  Now we proceed to the  estimation of the first term in the second line of  \eqref{eq:maindecomp}.  
By the last line in \eqref{def:ethamu}, we get  $\delta_0<\left( \mathbb J+\hat C\right)^{-1}\eta e^{-\mu}<\eta$. Applying  \eqref{est:Jprod} and since $\mu\le \nu/2-\frac {\ln 3}{\rm N +1}$, we get
 \begin{equation*}
     \begin{split}
 \left\| \prod_{i =1}^{n+1}J(i) Z^{[k]}_0 \right\|  \le \left\| \prod_{i=1}^{n+1}J(i) \right\| |Z^{[k]}_0|\le  \delta_0 e^{-\nu(n+1)/2}\le  &\eta e^{-\nu(n+1)/2}<\eta e^{-\mu(n+1)-\frac {\ln 3}{\rm N +1}(n+1)}\\  < & \eta e^{-\mu(n+1) - \ln 3} =\frac 13 \eta e^{-\mu(n+1) }.
    \end{split}
     \end{equation*} 
 So $\|Z^{[k]}_{n+1}\|\le \eta e^{-\mu (n+1)}$ on $\Lambda_{\gamma, k}\cap \Omega_{\tau k}$, concluding the induction step and the proof for all $n\in \mathbb N$. 

\bigskip

\subsection{About $\mathcal A$ and $\mathcal B$ for equal controls}
\label{subsec:mathcalAB}

Assume that $a=b$, $q=(1-\varepsilon)p$, for some $\varepsilon\in (0, 1)$.  Then
\[
\frac{4}{p+q +\sqrt{(p-q)^2+4abpq}}>\frac{2}{p(1+a)} \, ,
\]
since
\[
2p+2pa>2p-p\varepsilon+\sqrt{\varepsilon^2p^2+4a^2p^2(1-\varepsilon)}, \quad  p(2a+\varepsilon)>\sqrt{\varepsilon^2p^2+4a^2p^2-\varepsilon4a^2p^2}, \quad 
\]
\[
4a^2+4\varepsilon a+\varepsilon^2>\varepsilon^2+4a^2-\varepsilon4a^2,\quad 1>-a,
\]
which is true. So in this case $\mathcal A<\mathcal B$, the low estimate obtained by \eqref{def:mathcalA} is better than the estimate obtained by \eqref{def:mathcalB}. 

 If, however, $a=0.5$, $b=0.7$, $q=0.9p$, then we have
\[
2p(1+a)=1.33 p<p+q +\sqrt{(p-q)^2+4abpq}=1.9p+\sqrt{p^2\times 0.01+4\times 0.5\times 0.7\times  0.9 p^2} \, , 
\]
since $1.33<3.026,$ so in this case $\mathcal A>\mathcal B$.




\begin{thebibliography}{99}
\bibitem{aberkane}
S. Aberkane  and V.  Dragan, 
On the existence of the stabilizing solution of generalized Riccati equations arising in zero-sum stochastic difference games: the time-varying case,
 J. Difference Equ. Appl.  26 (7)  (2020), 
pp. 913--951.


\bibitem{Ayers}
K. Ayers, D.  Dmitrishin, A. Radunskaya, A. Stokolos, and K., Stokolos, 
Search for invariant sets of the generalized tent map,  J. Difference Equ. Appl. 29  (2003), pp. 1156--1183. 


 \bibitem {BHEBL}
S. Baigent, Z. Hou, S. Elaydi, E. C. Balreira, and R. Luís, 
A global picture for the planar Ricker map: convergence to
fixed points and identification of the stable/unstable
manifolds,   J.  Difference Equ. Appl.
29 (2023),  
pp. 575--591.


\bibitem{BRAllee}
E. Braverman and A. Rodkina,
Stochastic difference equations with the Allee effect, 
Discrete Contin. Dyn. Syst. Ser.
Discrete and Continuous Dynamical Systems Series 
 A  36 (2016),
pp. 5929--5949.

\bibitem{BKRPhys}
  E. Braverman, C. Kelly,  and A. Rodkina,  Stabilisation of Difference Equations with Noisy Prediction-Based Control,  
Physica D: Nonlinear Phenomena  326 (2016), pp. 21-31. 

\bibitem{BRMedv}
E. Braverman and A. Rodkina,
Stochastic control stabilizing an unstable or chaotic maps, 
 J.  Difference Equ.Appl.  25  (2019), pp. 151-178.


 \bibitem{BKRcycl}   
E. Braverman, C. Kelly, and  A. Rodkina, 
 \newblock {Stabilization of cycles with stochastic prediction-based and target-oriented
control,} 
  \newblock {Chaos 30 (2020) 
093116, pp.1-15.}

\bibitem{BRmult} 
\newblock E. Braverman and A. Rodkina  
 \newblock {Stabilizing multiple equilibria and cycles with noisy prediction-based control},
  \newblock {Discrete Contin. Dyn. Syst. Ser.  B, (2022), 
pp.  5419-5446.}

%





\bibitem{BR2023}
E. Braverman and A. Rodkina,
Noisy prediction-based control leading to stability switch,
Mathematics and Computers in Simulation   213  (2023)
pp. 418--443.


\bibitem{pbclocsys}
E. Braverman and A. Rodkina,
Including stochastics in Prediction-Based Control of difference systems:
Stabilizing and destabilizing by noise, 
Systems \& Control Letters  193  (2024), 20 pp, 105918.


\bibitem{cushing2024chaos} 
J.M. Cushing., R.F.  Costantino, B. Dennis, R. Desharnais, and S.M.  Henson,
 Capturing chaos: a multidisciplinary approach to nonlinear population dynamics,
J. Difference Equ. Appl. 30 (2024), 
pp. 1002–1039.

\bibitem{deng}
G. Deng,  G. Chen, and L. Qian,
On the global asymptotic stability and oscillation of solutions in a stochastic business cycle model,
 J. Difference Equ. Appl.   22  (2016), 
pp. 1609--1620.


 \bibitem{desharnais2001chaos}
    R.A.    Desharnais, R.F.  Costantino, J.M.  Cushing, S.M. Henson, and B. Dennis, 
        Chaos and population control of insect outbreaks,
        \newblock Ecology Letters \textbf{4}(3)  (2001), pp. 229--235.


\bibitem{Elaydi}
S. Elaydi, 
An Introduction to Difference Equations, Springer, 2005.

\bibitem{Kang1}
S. Elaydi, Y.  Kang, and R. Luís,
Global asymptotic stability of evolutionary periodic Ricker competition models,
J. Difference Equ. Appl. 30(8) (2024), 
pp. 1222--1252.

\bibitem{han}
C. Han,  H.  Li, H. Zhang, and M.  Fu,
Optimal control and stabilization for discrete-time Markov jump linear systems with input delay,
SIAM J. Control Optim.  59(5)   (2021), 
pp. 3524--3551.

\bibitem{Medvedev}
P. Hitczenko and G. Medvedev, Stability of equilibria of randomly perturbed maps, 
Discrete Contin. Dyn. Syst. Ser. B 22 (2017), pp. 269--281.

\bibitem{parasitoid}
M. Garić-Demirović, D. Kovačević and  M. Nurkanović,
Stability analysis of solutions of certain May's host-parasitoid model by using KAM theory
AIMS Math. 9(6)  (2024), 
pp. 15584–15609.

\bibitem{girard}
A. Girard and P. Mason, 
 Lyapunov functions for shuffle asymptotic stability of discrete-time switched systems,
IEEE Control Syst. Lett. 3 (2019), 
pp. 499--504.




\bibitem{kulenovic}
M.R.S. Kulenović and S. Van Beaver, 
Global dynamics of modified discrete Lotka-Volterra model,
Springer Proc. Math. Stat., 416
Springer, Cham, 2023, pp. 309--338.

\bibitem{Liz2007}
E.  Liz, 
Local stability implies global stability in some one-dimensional discrete single-species models,
Discrete Contin. Dyn. Syst. Ser. B 7(1) (2007), 191--199.


\bibitem{Liz2022}
E.  Liz, 
Bifurcation analysis of a piecewise-smooth Ricker map with proportional threshold harvesting,  
 J. Difference Equ. Appl.  28 
(2022), pp. 1268--1281.

\bibitem{FL2010}
E. Liz  and D. Franco,
Global stabilization of fixed points using predictive control,
Chaos 20 (2010), pp. 9,  023124.

\bibitem{LP2014}
E. Liz  and C.  P\"{o}tzsche, 
PBC-based pulse stabilization of periodic orbits, 
 Phys. D   272  (2014), pp. 26--38.


\bibitem{LSO}
R. Luis, S. Elaydi, and H. Oliveira. Stability of a Ricker-type competition
model and the competitive exclusion principle. Journal of Biological Dynamics, 5(6) (2011), 
pp. 636--660.

\bibitem{Kang2}
L. Rodriguez Rodriguez, C.  Bustamante Orellana, N. Cooke, M.  Demir, and Y. Kang, 
Communication dynamics of a two-agent interaction model with applications to human-autonomy teaming,
J. Difference Equ. Appl. 30 (2024), 
pp. 1222--1252.


\bibitem{sah2013stabilizing}
 P. Sah, J.P. Salve, and S. Dey, 
Stabilizing biological populations and
	metapopulations through adaptive limiter control,
J. Theoret.  Biol.  320  (2013), pp. 113--123 .


\bibitem{Segura2019}
J. Segura,  F.M. Hilker, and D.  Franco, 
Enhancing population stability with combined adaptive limiter control
and finding the optimal harvesting–restocking balance,
Theoretical Population Biology 130 (2019), pp. 1--12.

\bibitem{Singer}
D. Singer, 
 Stable orbits and bifurcation of maps of the interval,
SIAM J. Appl. Math. 35(2) (1978), 260--267.


\bibitem{Shiryaev96} 
A.N. Shiryaev,
\newblock Probability, 2nd edition,
\newblock Springer (Berlin), 1996.


\bibitem{Smith}
H.L. Smith, 
Planar competitive and cooperative difference equations,
J. Differ. Equations Appl.  3  (1998), 
pp. 335--357.

\bibitem{tang_li} 
X. Tang and X. Li, 
The stability of two area-preserving biological models,
Discrete Contin. Dyn. Syst. Ser. B 30 (2025), 
pp. 2619--2644.

\bibitem{uy99} 
T. Ushio and S. Yamamoto,
Prediction-based control of chaos, 
Phys. Lett. A   264  (1999),   pp. 30--35.

\bibitem{liberzon}
G.S. Vicinansa  and D.  Liberzon, 
  Estimation entropy, Lyapunov exponents, and quantizer design for switched linear systems,
SIAM J. Control Optim.  61(1)  (2023), 
pp. 198--224.

\bibitem{yamamoto}
Y. Yamamoto, T. Yusaku, T. Kuroiwa, K. Oka, E. Ishiwata, M. Iwasaki, 
Discrete Lotka-Volterra systems with time delay and its stability analysis,
Phys. D 474 (2025), Paper No. 134562, 9 pp.

\bibitem{yedaneh}
H.\,N. Yeganeh and S. Baigent, 
Convexity of non-compact carrying simplices in logarithmic coordinates,
J. Difference Equ. Appl. 30 (2024), 
pp. 1671--1691.

\bibitem{Zelinka2023}
I. Zelinka and R. Senkerik, 
 Chaotic attractors of discrete dynamical systems used in the core of evolutionary algorithms: state of art and perspectives,
J. Difference Equ. Appl.  29 
(2023), pp. 1202--1227.
\end{thebibliography}
\end{document}